\documentclass[reqno]{amsart}

\usepackage{amsmath,amsthm,amssymb}
\usepackage{color}
\usepackage[top=3.4cm,bottom=3.2cm,left=3.2cm,right=3.2cm]{geometry}
\usepackage{esint}
\usepackage[utf8]{inputenc}
\usepackage{hyperref}
\usepackage[active]{srcltx}
\usepackage{enumitem}
\usepackage{mathtools}

\title[Exponential convergence to equilibrium for coupled systems]{Exponential convergence to equilibrium for coupled systems of nonlinear degenerate drift diffusion equations}

\author{Lisa Beck}
\address{Lisa Beck \\ Institut für Mathematik \\ Universit\"at Augsburg \\ Universit\"atsstraße 14 \\ D-86159 Augsburg \\ Germany}
\email{lisa.beck@math.uni-augsburg.de}
\author{Daniel Matthes}
\address{Daniel Matthes \\ Zentrum Mathematik/M8 \\ Technische Universit\"at M\"unchen \\ Boltzmannstra\ss e 3 \\ D-80538 Garching \\ Germany}
\email{matthes@ma.tum.de}
\thanks{This research was supported by the DFG Collaborative Research Center TRR 109, ``Discretization in Geometry and Dynamics''.}
\author{Martina Zizza}
\address{Martina Zizza \\ SISSA-ISAS \\ Via Bonomea, 265 \\ 34136 Trieste TS \\ Italy}
\email{mzizza@sissa.it}

\newcommand{\N}{\mathbb{N}}
\newcommand{\R}{\mathbb{R}}
\newcommand{\Rp}{\mathbb{R}_{>0}}
\newcommand{\Rnn}{\mathbb{R}_{\ge0}}
\newcommand{\ball}{\mathbb{B}}
\newcommand{\eins}{\mathbf{1}}

\newcommand{\intrd}{\int_{\mathbb{R}^d}}
\newcommand{\dn}{\mathrm{d}}
\newcommand{\dd}{\,\mathrm{d}}
\newcommand{\dff}{\mathrm{D}}
\newcommand{\eps}{\varepsilon}
\newcommand{\dv}{\operatorname{div}}

\newcommand{\hausdorff}{{\mathcal H}}
\newcommand{\ee}{\mathbf{e}}
\newcommand{\themap}{\Gamma}
\newcommand{\leb}{{\mathcal L^d}}
\newcommand{\knl}{\mathcal K}
\newcommand{\opt}{\text{opt}}

\newcommand{\prbtwo}{\mathcal{P}_2^r(\mathbb{R}^d)}
\newcommand{\prbtwotwo}{[\mathcal{P}_2^r(\mathbb{R}^d)]^2}
\newcommand{\wass}{\mathbf{W}_2}
\newcommand{\dwass}{\widetilde{\mathbf{W}}_2}
\newcommand{\auxe}{\mathcal{E}}

\newcommand{\spt}[1]{\Omega^{#1}_\eps}

\newcommand{\ent}{\mathbf{H}}
\newcommand{\entt}{\widetilde{\mathbf{H}}}

\newcommand{\dss}{\mathbf{D}}
\newcommand{\nrg}{\mathbf{E}}
\newcommand{\lyp}{\mathbf{L}}
\newcommand{\fnc}{\mathbf{F}}

\newcommand{\evi}{\textup{(}EVI${}_0$\textup{)}}

\newtheorem{theorem}{Theorem}[section]
\newtheorem{proposition}[theorem]{Proposition}
\newtheorem{lemma}[theorem]{Lemma}
\newtheorem{corollary}[theorem]{Corollary}
\newtheorem{remark}[theorem]{Remark}
\newtheorem{example}[theorem]{Example}
\newtheorem{definition}[theorem]{Definition}

\usepackage[normalem]{ulem}
\newcommand{\stkout}[1]{\ifmmode\text{\sout{\ensuremath{#1}}}\else\sout{#1}\fi}

\numberwithin{equation}{section}

\begin{document}

\begin{abstract}
  We study the existence and long-time asymptotics of weak solutions
  to a system of two nonlinear drift-diffusion equations
  that has a gradient flow structure in the Wasserstein distance.
  The two equations are coupled through a cross-diffusion term
  that is scaled by a parameter $\eps\ge0$.
  The nonlinearities and potentials are chosen such that in the decoupled system for $\eps=0$,
  the evolution is metrically contractive, with a global rate $\Lambda>0$.
  The coupling is a singular perturbation in the sense that for any $\eps>0$,
  contractivity of the system is lost.
  
  Our main result is that for all sufficiently small $\eps>0$,
  the global attraction to a unique steady state persists, 
  with an exponential rate $\Lambda_\eps=\Lambda-K\eps$.
  The proof combines results from the theory of metric gradient flows
  with further variational methods and functional inequalities.
\end{abstract}

\maketitle

\section{Introduction}

In this paper, we analyze existence and long-time asymptotics of non-negative unit-mass solutions~$u$ and~$v$
of the following coupled system of two degenerate nonlinear drift-diffusion equations on $\R^d$:
\begin{equation}
  \label{eq:eq}
  \begin{split}
    \partial_tu &= \dv\big(u\,\nabla[F'(u) +\eps\partial_uh(u,v)+\Phi]\big),  \\
    \partial_tv &= \dv\big(v\,\nabla[G'(v) +\eps\partial_vh(u,v) + \Psi]\big).
  \end{split}
\end{equation}
Notice that the diffusive contributions $\dv ( u  \nabla F'(u))$ and $\dv ( v \nabla G'(v)) $ of the system~\eqref{eq:eq} can also be expressed as $\Delta f(u)$ and $\Delta g(v)$, respectively, by introducing functions $f$ and $g$ via the relations $f'(r)=rF''(r)$ and $g'(r)=rG''(r)$ for $r >0$.
The precise hypotheses on the various functions are formulated in Section~\ref{sct:hypo} below;
in brief: the nonlinearities $F$, $G$ for the individual components
are smooth convex functions that degenerate at zero, i.e., with $F'(0)=G'(0)=0$;
the coupling is moderated by a nonlinear function $h$ that is regular and bounded;
the coupling strength $\eps>0$ will be assumed small;
and the potentials $\Phi$, $\Psi$ are $\Lambda$-convex, with some $\Lambda>0$.
We prove the global existence of transient solutions $(u_\eps(t),v_\eps(t))_{t\ge0}$ to~\eqref{eq:eq} for initial data of finite energy;
we show existence and uniqueness of a stationary solution $(\bar u_\eps,\bar v_\eps)$, and analyze its regularity;
finally, we obtain convergence of $(u_\eps(t),v_\eps(t))$ to $(\bar u_\eps,\bar v_\eps)$ in $L^1(\R^d)$ as $t\to\infty$, at an exponential rate.

Our approach is a variational one:
we consider --- at least formally  --- the system~\eqref{eq:eq} as a metric gradient flow of the energy functional
\begin{align}
  \label{eq:nrg}
  \nrg_\eps(u,v) = \intrd \big[F(u)+G(v) + u\Phi + v\Psi + \eps h(u,v) \big]\dd x
\end{align}
on the cross product of two copies of the space $\prbtwo$ of probability densities of finite second moment,
endowed with the $L^2$-Wasserstein distance;
the definitions are recalled in Section~\ref{sct:prelims} below.

What we consider as our main contribution
--- going beyond the analysis of the specific system~\eqref{eq:eq} ---
is a blending of the abstract machinery of uniformly contractive metric gradient flows
with more hands-on variational methods by which we extend the relevant estimates into the totally non-convex regime.
Indeed, under the hypotheses stated below,
the two evolution equations of the \emph{decoupled system}, i.e.,~\eqref{eq:eq} for $\eps=0$,
constitute $\Lambda$-contractive gradient flows in the $L^2$-Wasserstein metric.
Stated differently, $\nrg_\eps$ is $\Lambda$-uniformly convex along geodesics.
A consequence from the general theory --- see e.g.~\cite[Chapter 11]{AGS} --- is the convergence of both components $u(t)$ and $v(t)$
to respective stationary solutions $\bar u_0$ and $\bar v_0$ in $L^1(\R^d)$ at exponential rate $\Lambda$. 

It is a remarkable fact that, in the eyes of the Wasserstein metric,
the coupling via $h$ in~\eqref{eq:eq} is a singular perturbation of the two decoupled equations,
no matter how tame $h$ is.
As soon as $\eps>0$, geodesic semi-convexity of the energy $\nrg_\eps$ is lost entirely, i.e.,
$\nrg_\eps$ is \emph{not} $\lambda$-uniformly convex along geodesics, not for any $\lambda\in\R$,
see Proposition~\ref{prp:nonconvex}.
In that case, the gradient flow machinery does not provide any information on the long-time behaviour of solutions to~\eqref{eq:eq}.
Rather, it implies that the coupled systems is \emph{not} globally uniformly semi-contractive anymore,
that is, there is no Lipschitz bound on the instantaneous divergence of trajectories.
Of course, this does not rule out global exponential convergence to equilibrium, but the latter behaviour is inaccessible to that theory.

To prove equilibration at exponential rate, we build on associated functional inequalities, see e.g.~\cite{FuncIneq},
which can be obtained as a consequence of metric contractivity, but which are more ``structurally robust'' than the latter.
Specifically, we obtain in Lemma \ref{lem:lypconvex} that 
\begin{equation}
  \label{eq:thisone}
  \begin{split}
    2(\Lambda-K_0\eps)&\intrd\big[F(u) - F(\bar u_\eps) + (u-\bar u_\eps)\big(\Phi + \eps \partial_u h(\bar u_\eps,\bar v_\eps)\big) \big]\dd x \\
    &\le \intrd u\big|\nabla\big[F'(u)+\Phi+\eps\partial_uh(\bar u_\eps,\bar v_\eps)\big]\big|^2\dd x    
  \end{split}
\end{equation}
(and a similar inequality for $v$ in place of $u$),
which expresses the relation between energy and dissipation
for the auxiliary scalar gradient flow $\partial_tu=\dv(u\nabla[F'(u) + \Phi+\eps\partial_uh(\bar u_\eps,\bar v_\eps)])$.
The constant $\Lambda-K_0\eps$ in~\eqref{eq:thisone} is a lower bound
on the modulus of convexity of $\Phi+\partial_uh(\bar u_\eps,\bar v_\eps)$,
which is controlled thanks to our $C^2$-estimates on the steady state in Theorem~\ref{thm:stationary}.
With~\eqref{eq:thisone} at hand, 
the key step in our approach, see Proposition~\ref{prp:core}, amounts to proving that
the negative time-derivative of the integral on the left-hand side of~\eqref{eq:thisone}
bounds the expression on the right-hand side, up to a factor $1-K_1\eps$.
This eventually implies convergence to equilibrium in $L^1(\R^d)$
at a global exponential rate of at least $\Lambda_\eps \coloneqq \Lambda-K\eps$, with $K \coloneqq K_0+\Lambda K_1$.

We emphasize that the novelty of our result does not lie in the proof of convergence to equilibrium as such
--- a \emph{qualitative} result could be obtained at almost no cost e.g. from the LaSalle principle ---
but in the \emph{quantitative} estimate on the convergence to equilibrium,
with the exponential rate~$\Lambda_\eps$ that is an $\eps$-perturbation of the decoupled rate~$\Lambda$.
It is further significant that~\eqref{eq:eq} is considered on~$\R^d$,
and that the steady state $(\bar u_\eps,\bar v_\eps)$ is compactly supported;
hence, there is no standard inequality like Poincar\'{e} or log-Sobolev to conclude exponential convergence,
not even at \emph{some} rate smaller than~$\Lambda_\eps$,
but a more adapted inequality like~\eqref{eq:thisone} appears necessary. 

\subsection{Positioning of our results}
Coupled systems of nonlinear drift-diffusion equations are ubiquitous.
They are used in the modeling of chemical reactions~\cite{Chemistry}, 
flows in porous media~\cite{Flows}, semi-conductor devices~\cite{Semiconductors}
population dynamics~\cite{Population}, rival gangs in a city~\cite{Gangs},
just to name a few of the countless applications. 
The literature concerning the very natural question about long-time asymptotics is huge,
albeit mostly focused on such systems
with a particular rigid algebraic structure of the diffusion (being diagonal, or even linear)
but with additional source terms, describing e.g. reactions.

We briefly recall the situation for \emph{scalar} drift-diffusion equations.
The first proof of exponential convergence to equilibrium in a degenerate parabolic equations
of the type $\partial_tu=\Delta f(u)+\dv(u\nabla\Phi)$ has been given 
in the case $f(u)=u^m$ and $\Phi(x)=\frac12|x|^2$ on~$\R^d$
by three different methods:
by a nonlinear extension of the Bakry--Emery method~\cite{CarTos},
by a variational proof of the entropy-dissipation inequality~\cite{Dolbeault},
and by virtue of gradient flows in the $L^2$-Wasserstein metric~\cite{Otto}.
These methods have been extended later on to more general~$f$'s and~$\Phi$'s,
and also to bounded domains $\Omega\subset\R^d$,
see e.g.~\cite{Monatshefte} and~\cite{FuncIneq} for such generalizations of the first and of the third approach, respectively.
The common fundamental result is that if~$f$ has certain convexity properties,
and if~$\Phi$ is uniformly convex of modulus~$\Lambda>0$,
then solutions~$u$ converge to the unique equilibrium in~$L^1$ at exponential rate~$\Lambda$. 

There appears to be no result of comparable simplicity and generality for \emph{coupled systems} of parabolic equations.
All of the aforementioned methods of proof break down as soon as multi-component densities are considered,
except in some particular systems with a very special algebraic structure, see e.g.~\cite{Liero,Zinsl-metric}.
System~\eqref{eq:eq} is a paradigmatic illustration of this effect on the gradient flow approach,
where geodesic convexity turns out to be highly fragile with respect to the coupling of components:
the decoupled functional~$\nrg_0$ from~\eqref{eq:nrg} is $\Lambda$-convex along geodesics,
but~$\nrg_\eps$ completely loses convexity for any $\eps>0$, see Proposition~\ref{prp:nonconvex} below.
For the sake of completeness,
we mention that although the gradient flow approach loses its quantitative estimates,
various qualitative results, e.g.,
on the shape of steady states or on the convergence to equilibrium (without rates),
have been obtained along these lines for specific systems, see for instance~\cite{Matioc}.

Limited generalizations of the scalar theory
have been developed for reaction-diffusion systems, and recently also for cross-diffusion systems (with or without reactions).
Although many of these systems still bear a gradient flow structure~\cite{Alex-GF},
the more robust energy method has proven better adapted to study long-time asymptotics.
In reaction-diffusion systems, the substantial challenge is in the control of the growth induced by the reactions,
while the diffusion itself is typically decoupled, and frequently just linear.
Prototypical results on exponential equilibration have been obtained in \cite{DF-1,Markowich, Mielke-1,Mielke-2}
for systems with linear diffusion,
and in~\cite{Bao-1} for component-wise non-linear diffusion.
In (reaction-)cross-diffusion systems, the diffusion matrix is non-diagonal,
but usually subject to restrictive structural conditions.
Recent results on exponential convergence to equilibrium have been obtained, e.g.,
for systems with volume filling~\cite{Zamponi},
of Maxwell--Stefan type~\cite{Ansgar-1},
or with SKT-structure~\cite{Bao-2}.

None of the above results covers the exponential equilibration presented in Theorem~\ref{thm:longtime} below,
i.e., for a system which is fully nonlinear with a general (albeit small) coupling.
Note in particular that one of the main challenges in the analysis of~\eqref{eq:eq}
is that the steady state is compactly supported,
which rules out the use of the standard log-Sobolev
--- that is a key element in the proofs of essentially all of the aforementioned results ---
but requires more adapted functional inequalities like~\eqref{eq:thisone}.
A recent result that is closer in spirit to our approach to proving Theorem~\ref{thm:longtime}
has been obtained in~\cite{Mary}:
the authors treat a system with a small nonlinear coupling like~\eqref{eq:eq};
however, there, linearity of~$F$ and~$G$, a bounded spatial domain, and an a priori $L^\infty$-bound are assumed,
which allows to perform the estimates in a much simpler way, using Poincar\'{e}'s inequality.

For completeness, we finally mention two results~\cite{Zinsl-KS,Zinsl-NP},
where a similar strategy has been applied to the parabolic-parabolic Keller--Segel and the Nernst--Planck system,
respectively.
There, the coupling is between a Wasserstein and an $L^2$-gradient flow,
unlike the coupling between two Wasserstein gradient flows here.
A variant of the above has been explored in~\cite{zizza}.
The paper at hand uses elements from~\cite{Zinsl-KS,zizza}.

We briefly comment on the positioning of our existence result, Theorem~\ref{thm:transient}.
We use the celebrated JKO scheme~\cite{JKO} to obtain solutions by a variational time-discrete approximation.
This scheme has been used for proving existence of various non-linear parabolic equations
like doubly degenerate parabolic PDEs~\cite{doubly}, including the $p$-Laplace equation~\cite{Agueh};
in nonlinear diffusion-aggregation equations~\cite{five}, including the parabolic-elliptic Keller--Segel model~\cite{Blanchet-1},
in fourth order quantum and thin film equations~\cite{GST,MMS,Liso};
and in many further instances.
Applications to coupled systems are numerous as well,
including for instance systems with non-local aggregation~\cite{MdF-1,MdF-2} and cross-diffusion~\cite{Schlake,Berendsen},
and also combinations of Wasserstein and $L^2$-gradient flows,
like the parabolic-parabolic Keller--Segel~\cite{Zinsl-KS,other-KS} or the Nernst--Planck system~\cite{other-NP}.
In several of these cases, the existence proof could have been obtained also by more elementary methods.
Also for~\eqref{eq:eq}, the \emph{boundedness-by-entropy method}~\cite{Ansgar-bbe},
albeit not directly applicable, would have paved an alternative way.
For us, the time-discrete approximation via a minimizing movement with respect to the $L^2$-Wasserstein distance
is crucial for making our main result, namely the long-time asymptotics, fully rigorous.
The work most closely related to our proof of Theorem~\ref{thm:transient} below is~\cite{MdF-2}:
there, the same variational method has been used to construct time-discrete approximations to a system similar to~\eqref{eq:eq},
augmented with additional non-local interaction terms.
The hypotheses of~\cite{MdF-2} are complementary to ours;
while our conditions are more restrictive on the coupling~$h$, we allow for more general~$F$ and~$G$.

\subsection{General hypotheses}
\label{sct:hypo}
The following hypotheses are assumed throughout the paper.
Several of the assumptions are made for convenience of the presentation, and are far from being necessary.
\begin{itemize}
\item \textbf{Potentials:} For $\Phi,\Psi\in C^\infty(\R^d)$, we assume:
  \begin{itemize}
  \item there are positive constants~$\Lambda$ and~$M$
    such that
    \begin{align}
      \label{hyp:Phi}
      \Lambda\eins \le \nabla^2\Phi \le M\eins,
      \quad
      \Lambda\eins \le \nabla^2\Psi \le M\eins;
    \end{align}
  \item $\Phi$ and~$\Psi$ vanish at their respective minima $\underline x_\Phi,\underline x_\Psi\in\R^d$,
    i.e.
    \begin{align}
      \label{hyp:Phi0}
      0=\inf_{\R^d} \Phi = \Phi\big(\underline x_\Phi\big),
      \quad
      0=\inf_{\R^d} \Psi = \Psi\big(\underline x_\Psi\big).
    \end{align}
  \end{itemize}
\item \textbf{Nonlinearities:} We assume $F,\,G\in C^\infty(\Rp)\cap C^1(\Rnn)$ such that:
  \begin{itemize}
  \item $F''(r)>0$ and $G''(r)>0$ for all $r>0$,
    and
    \begin{align}
    \label{hyp:quadratic_growth}
     \liminf_{r\to\infty}F''(r) > 0, \quad \liminf_{r\to\infty}G''(r)> 0;
    \end{align}
  \item they degenerate at zero to first order, i.e.,
    \begin{align}
      \label{eq:degeneracy}
      F(0)=G(0)=0, \quad F'(0)=G'(0)=0;
    \end{align}
  \item there are exponents $m,n\ge2$ such that
    \begin{align}
      \label{hyp:powerF}
      \lim_{r\downarrow 0}r^{-(m-2)}F''(r) \in (0,\infty),
      \quad
      \lim_{r\downarrow 0}r^{-(n-2)}G''(r) \in (0,\infty);
    \end{align}
  \item they satisfy the (infinite dimensional) McCann condition, i.e., for all $r>0$,
    \begin{align}
      \label{eq:mccann}
      rF'(r) \le F(r) + r^2F''(r), \quad rG'(r) \le G(r) + r^2G''(r);
    \end{align}
  \item they satisfy the doubling condition, i.e., there is a constant~$D$ such that for all $r,s>0$,
    \begin{align}
      \label{eq:doubling}
      F(r+s) \le D(1+F(r)+F(s)), \quad G(r+s) \le D(1+G(r)+G(s)).
    \end{align}
  \end{itemize}
\item \textbf{Coupling:} Concerning $h\in C^\infty(\Rp^2)\cap C^1(\Rnn^2)$, we assume that:
  \begin{itemize}
  \item $h$ vanishes to first order on $\partial\Rnn^2$,
    \begin{align}
      \label{hyp:hvanish}
      h = \partial_uh = \partial_vh \equiv0 \quad\text{on $\partial\Rnn^2$};
    \end{align}
  \item there is an $\eps^*>0$ such that
    \begin{align}
      \label{hyp:hconvex}
      (u,v) \mapsto  F(u)+G(v) + 2\eps^*h(u,v) \quad \text{is convex};
    \end{align}
  \item with the same~$\eps^*$, there holds, for all $u,v>0$,
    \begin{align}
      \label{hyp:hbound}
      2 \eps^* |h(u,v)| \leq F(u)+G(v) . 
    \end{align}
  \end{itemize}
  \item \textbf{Degeneracy, boundedness, and swap condition:}
  Define $\theta_u,\theta_v \colon \Rnn^2\to\R$ by
  \begin{align}
    \label{eq:deftheta}
    \theta_u(\rho,\eta) \coloneqq  \partial_uh\big((F')^{-1}(\rho),(G')^{-1}(\eta)\big),
    \quad
    \theta_u(\rho,\eta) \coloneqq  \partial_uh\big((F')^{-1}(\rho),(G')^{-1}(\eta)\big).
  \end{align}
  We say that the triple $(F,G,h)$ \ldots 
  \begin{itemize}
  \item \ldots satisfies the \emph{swap condition} if there is some constant~$W$
    such that, for all $u,v>0$,
    \begin{align}
      \label{hyp:swap}
      \partial_\eta\theta_u\big(F'(u),G'(v)\big)\le W\sqrt{v/u},
      \quad
      \partial_\rho\theta_v\big(F'(u),G'(v)\big)\le W\sqrt{u/v};
    \end{align}
  \item \ldots is \emph{$k$-bounded} for some $k\in\N$
    if $\theta_u,\theta_v\in C^k(\Rnn^2)$
    and all partial derivatives of order less or equal to~$k$ are bounded on~$\Rnn^2$;  
  \item \ldots is \emph{$k$-degenerate} for some $k\in\N$
    if $\theta_u,\theta_v\in C^k(\Rnn^2)$
    and all partial derivatives of order less or equal~$k$ vanish on $\partial\Rnn^2$.
  \end{itemize}
\end{itemize}
\begin{remark}
\label{remark_consequence_hypotheses}
\leavevmode
\begin{enumerate}[labelindent=\parindent,leftmargin=*,font=\normalfont, label=(\arabic{*}), ref=(\arabic{*})]
 \item Hypotheses~\eqref{hyp:Phi}\&\eqref{hyp:Phi0} imply that~$\Phi$ and~$\Psi$ are bounded from above and from below by parabolas:
  \begin{align}
    \label{eq:Phiupbound}
    \frac{\Lambda}2|x-\underline x_\Phi|^2 \le\Phi(x)\le \frac M2|x-\underline x_\Phi|^2,
    \quad
    \frac{\Lambda}2|x-\underline x_\Psi|^2 \le\Psi(x)\le\frac  M2|x-\underline x_\Psi|^2.   
  \end{align}
  These estimates are directly obtained by a Taylor expansion about the respective minima.
  Similarly, one bounds the norm of the gradients
  and thus obtains in combination with~\eqref{eq:Phiupbound} that
  \begin{align}
    \label{eq:Phigradbound}
    \frac{2\Lambda^2}{M}\Phi(x)\le|\nabla\Phi(x)|^2\le \frac{2M^2}{\Lambda}\Phi(x),
    \quad
    \frac{2\Lambda^2}{M}\Psi(x)\le|\nabla\Psi(x)|^2\le \frac{2M^2}{\Lambda}\Psi(x).
  \end{align}  
  \item A consequence of the hypotheses on~$F$ and~$G$ is that
  both are \emph{uniformly} convex on each interval of the form $[r,\infty)$ with $r>0$.
  Further, in combination with the doubling condition,
  it follows that for all $r>0$,
  \begin{align}
    \label{eq:tripling}
    rF'(r) \le D(1+2F(r)), \quad rG'(r) \le D(1+2G(r)).
  \end{align}
  Indeed, convexity implies $F(2r)\ge F(r)+rF'(r)$, and~\eqref{eq:tripling} now follows via~\eqref{eq:doubling} for $s=r$.
  \item If $(F,G,h)$ is $2$-bounded and $2$-degenerate, there exists a constant $A \geq 0$ such that
  \begin{equation}
  \label{eqn_def_A}
    \omega(\rho,\eta)\le A \min\{1,\rho,\eta\},
  \end{equation}
  where $\omega\colon\Rnn^2\to\R$ is any of functions $\theta_u$, $\theta_v$, $\partial_\rho\theta_u$, $\partial_\eta\theta_u$, $\partial_\rho\theta_v$, or $\partial_\eta\theta_v$.
\end{enumerate}  
\end{remark}

\begin{example}
  \label{xmp:hypos}
  Consider~$F$,~$G$ and~$h$ of the form
  \begin{align*}
    F(u) = \frac{u^m}{m},\ G(v) = \frac{v^n}{n}, \quad h(u,v) = u^pv^qe^{-\lambda(u+v)}
  \end{align*}
  with positive exponents $m,n,p,q$ and~$\lambda$. We claim that~$F$,~$G$ and~$h$ satisfy their respective hypotheses above and also the swap condition provided that
  \begin{align*}
    m,n\ge 2,\quad p\ge m,\,q\ge n.
  \end{align*}
  Moreover, given $k\in\N$, we claim that $(F,G,h)$ is $k$-bounded and $k$-degenerate, if additionally
  \begin{align}
    \label{eq:puremadness}
    \frac{p-1}{m-1}>k,\quad \frac{q-1}{n-1}>k.
  \end{align}
  The verification of these claims is deferred to the appendix.
\end{example}

\subsection{Results}
Our first result concerns the existence of stationary solutions to~\eqref{eq:eq},
characterized as minimizers of~$\nrg_\eps$.
\begin{theorem}
  \label{thm:stationary}
  For each $\eps\in[0,\eps^*]$, there is a unique minimizer $(\bar u_\eps,\bar v_\eps)$ of~$ \nrg_\eps$.
  The densities~$\bar u_\eps$ and~$\bar v_\eps$ have compact supports that are sublevels of~$\Phi$ and~$\Psi$, respectively.
  Further, if~$h$ degenerates to order $k\in \N$,
  then the restrictions of $F'(\bar u_\eps)$ and $G'(\bar v_\eps)$ to their respective supports
  are bounded in $C^k$, uniformly with respect to $\eps\in[0,\eps^*]$.
\end{theorem}
A point of crucial importance is that for degeneracy of order two,
the functions $\partial_uh(\bar u_\eps,\bar v_\eps)$ and $\partial_vh(\bar u_\eps,\bar v_\eps)$ are in $C^2(\R^d)$,
with a global bound on second derivatives that is independent of $\eps\in[0,\eps^*]$.
This is needed to establish the functional inequalities~\eqref{eq:thisone},
which are essential for our proofs of the following results.
Concerning existence of transient solutions, we have:
\begin{theorem}
  \label{thm:transient}
  Assume in addition that $(F,G,h)$ is $2$-bounded and $2$-degenerate,
  and that the swap condition holds.
  There is some $\bar\eps>0$, such that for each $\eps\in[0,\bar\eps]$
  and any initial data $(u_0,v_0)\in\prbtwotwo$ of finite energy $\nrg_\eps(u_0,v_0)<\infty$,
  there exists a transient weak solution $(u_\eps(t),v_\eps(t))_{t\ge0}$ to the initial value problem for~\eqref{eq:eq}.
\end{theorem}
Here the most significant point is not the mere existence but the way of construction,
namely via the minimizing movement scheme for~$\nrg_\eps$ in the combined $L^2$-Wasserstein distances.

Finally, the main result of this paper is about the long-time asymptotics of transient solutions.
\begin{theorem}
  \label{thm:longtime}
  Under the same conditions as in Theorem~\ref{thm:transient} above,
  there exist constants $K>0$ and $B\ge1$ such that the following is true for all $\eps\in[0,\bar\eps]$:
  the transient solution $(u_\eps(t),v_\eps(t))_{t\ge0}$ constructed in the proof of Theorem~\ref{thm:transient}
  converges to the unique global minimizer $(\bar u_\eps,\bar v_\eps)$ from Theorem~\ref{thm:stationary}
  at exponential rate $\Lambda_\eps=\Lambda-K\eps$.
  More precisely, with~$\lyp$ being the Lyapunov functional defined in~\eqref{eq:deflyp},
  there holds
  \begin{align}
    \label{eq:lypdecay}
    \lyp\big(u_\eps(t),v_\eps(t)\big) \le\lyp(u_0,v_0) \, \exp\big(-2\Lambda_\eps t\big),
  \end{align}
  and in particular,~$u_\eps(t)$ and~$v_\eps(t)$ converge in $L^1(\R^d)$ to~$\bar u_\eps$ and~$\bar v_\eps$, respectively,
  with
  \begin{align}
    \label{eq:1}
    \|u_\eps(t)-\bar u_\eps\|_{L^1}^2 + \|v_\eps(t)-\bar v_\eps\|_{L^1}^2 \le B (\nrg_\eps(u_0,v_0) + 1) \, \exp\big(-2\Lambda_\eps t\big).
  \end{align}
\end{theorem}


\section{Preliminaries}
\label{sct:prelims}
\subsection{Wasserstein distance}
\label{sct:wasserstein}
$\ball_R \coloneqq \{x\in\R^d \colon|x|<R\}$ is the ball of radius $R>0$.
$\leb$ denotes the standard Lebesgue measure on $\R^d$.
For a probability measure $\mu$ on $\R^d$ and a measurable map $T:\R^d\to\R^d$,
the \emph{push-forward} of $\mu$ under $T$ is the uniquely determined probability measure $T\#\mu$ such that
\begin{align}
  \label{eq:defpush}
  \intrd \omega(y)\dd\big(T\#\mu\big)(y) = \intrd \omega\circ T(x)\dd\mu(x),
\end{align}
for any test function $\omega\in C(\R^d)$.
If both $\mu=u\leb$ and $T\#\mu=\hat u\leb$ are absolutely continuous, then we write $T\#u=\hat u$ for brevity.

$\prbtwo$ denotes the space of probability densities $u \colon \R^d\to\Rnn$ of finite second moment.
The natural notion of convergence on $\prbtwo$ is the \emph{narrow} one, i.e., weak convergence in duality with bounded continuous functions.
By Prokhorov's and by Alaoglu's theorem, subsets of densities with uniformly bounded second moment and $L^p$-norm (for some $p>1$)
are sequentially compact in $\prbtwo$.

The $L^2$-Wasserstein distance~$\wass$ is a metric on $\prbtwo$.
Convergence in $\wass$ is equivalent to weak convergence and convergence of the second moments.
Among the various possible definitions of~$\wass$
the following --- known as the (pre-)dual Kantorovich formulation ---
is the most suitable one for our needs:
for $u,\hat{u}\in\prbtwo$,
\begin{align}
  \label{eq:defW}
  \frac12\wass(u,\hat{u})^2 \coloneqq  \sup\left\{ \intrd \varphi(x)u(x)\dd x + \intrd\psi(y)\hat{u}(y)\dd y
  \colon \varphi(x)+\psi(y)\le \frac12|x-y|^2\right\}.
\end{align}
(Note the square and the factor $1/2$ on the left-hand side.)
A priori, the maximization above is carried out over all pairs $(\varphi,\psi)$ for which the integrals are well-defined,
i.e., $\varphi\in L^1(\R^d;u\leb)$ and $\psi\in L^1(\R^d;\hat u\leb)$.
However, it suffices to consider pairs $(\varphi,\psi)$ from the very restrictive class
of \emph{$c$-conjugate \footnote{The $c$ refers to the cost function, which is the standard one here, $c(x,y)=\frac12|x-y|^2$.} potentials}.
The latter means that the \emph{auxiliary potentials} $\tilde\varphi,\tilde\psi:\R^d\to\R\cup\{+\infty\}$ given by
\begin{align*}
  \tilde\varphi(x) = \frac12|x|^2-\varphi(x), \quad
  \tilde\psi(y) = \frac12|y|^2 -\psi(y)
\end{align*}
are proper, lower semi-continuous, convex, and Legendre-dual to each other,
$\tilde\varphi^\ast=\tilde\psi$ and $\tilde\psi^\ast=\tilde\varphi$.
Note that knowledge of either $\varphi$ or $\psi$ determines the respective other.
Further note that $\varphi(x)+\psi(y)\le\frac12 |x-y|^2$ is automatically satisfied since $\tilde\varphi(x)+\tilde\psi(y)\ge x\cdot y$.

The supremum in~\eqref{eq:defW} is attained by an optimal pair $(\varphi_\opt,\psi_\opt)$ of $c$-conjugate potentials.
Uniqueness of optimal pairs is delicate in general, even after removing
the global gauge invariance $(\varphi,\psi)\leadsto(\varphi+C,\psi-C)$.
For us, the following special case is important:
if~$\hat u$ is positive $\leb$-a.e.~on a ball $\ball_R\subset\R^d$ and zero outside,
then $\varphi_\opt$ is unique $u\leb$-a.e.~up to a global constant.

For an optimal pair $(\varphi_\opt,\psi_\opt)$,
the \emph{optimal transport map} $T \colon \R^d\to\R^d$ from~$u$ to~$\hat{u}$ is given by
\begin{align}
  \label{eq:defT}
  T(x) \coloneqq  x - \nabla\varphi_\opt(x),
\end{align}
which is well-defined $u\leb$-a.e.
It satisfies
\begin{align}
  \label{eq:pushbyT}
  \hat{u} = T\#u,
\end{align}
and
\begin{align}
  \label{eq:monge}
  \wass(u,\hat{u})^2 = \intrd |T(x)-x|^2u(x)\dd x.
\end{align}
$T$ is $u\leb$-a.e.\ unique for a given pair $(u,\hat u)$, which implies the following converse:
if $T=\nabla\tilde\varphi \colon \R^d\to\R^d$ is the $u\leb$-a.e.~defined
gradient of a proper, lower semi-continuous and convex function $\tilde\varphi \colon \R^d\to\R$,
and satisfies \eqref{eq:pushbyT},
then $T$ also satisfies \eqref{eq:monge},
and $\varphi_\opt(x)\coloneqq \frac12|x|^2-\tilde\varphi(x)$ gives rise to an optimal pair $(\varphi_\opt,\psi_\opt)$ of $c$-conjugate potentials.

Finally, we recall a characterization of geodesics:
define the interpolating maps $T_s \colon \R^d\to\R^d$ for all $s\in[0,1]$
by $T_s (x) \coloneqq  (1-s)x+sT(x) = x - s\nabla\varphi_u(x)$.
Then the curve $(u_s)_{s\in[0,1]}$ in $\prbtwo$ given by $u_s\coloneqq T_s\#\rho$ is a geodesic
joining $u=u_0$ to $\hat{u}=u_1$, that is,
\begin{align*}
  \wass(u,u_s) = s\wass(u,\hat{u}), \quad \wass(u_s,\hat{u}) = (1-s)\wass(u,\hat{u}).
\end{align*}
The natural space for solutions $(u,v)$ to~\eqref{eq:eq} is the cross product $\prbtwotwo$.
We endow it with a metric~$\dwass$ in the straight-forward way:
\begin{align*}
  \dwass\big((u,v),(\hat{u},\hat{v})\big) \coloneqq  \sqrt{\wass(u,\hat{u})^2+\wass(v,\hat{v})^2}.
\end{align*}
The following is easily seen.
\begin{lemma}
  \label{lem:doublegeodesic}
  A curve $(u_s,v_s)_{0\le s\le1}$ in $\prbtwotwo$ is a geodesic in $\prbtwotwo$ between $(u_0,v_0)$ and $(u_1,v_1)$
  if and only if  $(u_s)_{s\in[0,1]}$ and $(v_s)_{s\in[0,1]}$ are geodesics
  in $\prbtwo$ between~$u_0$,~$u_1$, and between~$v_0$,~$v_1$, respectively.  
\end{lemma}

\subsection{Displacement convexity}
\begin{definition}
  A functional~$\fnc$ on $\prbtwo$ is \emph{$\lambda$-uniformly displacement convex} with some modulus $\lambda\in\R$ if the real function
  \begin{align*}
    [0,1]\ni s\mapsto \fnc\big(T_s\#u)-\frac\lambda2s(1-s)\wass(u,\hat{u})^2
  \end{align*}
  is convex for any family $(T_s)_{s\in[0,1]}$ realizing the geodesic between~$u$ and $\hat{u}=T_1\#u$.
\end{definition}
Displacement convex functionals are rare.
An important class of examples is given by the sum of internal and potential energy:
\begin{align}
  \label{eq:cvxxmp}
  \fnc(u) = \intrd \big[e(u) + V u\big]\dd x.
\end{align}
In this case~$\fnc$ is $\lambda$-uniformly displacement convex
provided that the convex function $e \colon \Rnn\to\R$ satisfies McCann's condition,
and that $V \colon \R^d\to\R$ is $\lambda$-convex in the usual sense.
A consequence of that property is the validity of a functional inequality,
see e.g.~\cite[Theorem 2.1]{FuncIneq}.
\begin{lemma}
  \label{lem:FuncIneq}
  Assume that the functional~$\fnc$ of the type~\eqref{eq:cvxxmp} is
  such that the convex function $e \colon \Rnn\to\R$ satisfies McCann's condition,
  and such that $V \colon \R^d\to\R$ is $\lambda$-convex for $\lambda>0$.
  Then~$\fnc$ possesses a unique minimizer $u_*\in\prbtwo$,
  and for all $u\in\prbtwo$, there holds:
  \begin{align}
    \label{eq:FuncIneq}
    2\lambda\big[\fnc(u)-\fnc(u_*)\big] \le \intrd u \big|\nabla\big[e'(u)+V\big]\big|\dd x.
  \end{align}
\end{lemma}
Functionals of the type~\eqref{eq:cvxxmp} with $\lambda\ge0$ actually even enjoy the stronger property
of being convex along generalized geodesics, which has a variety of consequences.
The only consequence needed below is for the special case $h(u)=u\log u$ and $V\equiv0$,
when $\fnc=\ent$ is the entropy functional,
\begin{align}
  \label{eq:entropy}
  \ent(u) = \intrd u\log u\dd x.
\end{align}
The metric gradient flow of~$\ent$ is the heat equation $\partial_sU_s=\Delta U_s$,
and (thanks to convexity along generalized geodesics)
it satisfies the so-called \emph{evolution variational inequality} \evi, see~\cite[Theorem~4.0.4]{AGS}, which is
\begin{align}
  \label{eq:evi}
  \frac12\frac{\dn^+}{\dd s}\bigg|_{s=0^+}\wass\big(U_s,w\big)^2 \le \ent(w)-\ent(U_0)
\end{align}
for all $w\in\prbtwo$ and all solutions to  $\partial_sU_s=\Delta U_s$.

\subsection{Loss of displacement convexity for mixtures}
We indicate why the aforementioned general theory of $\lambda$-uniformly displacement convex functionals does not apply to the energy functional~$\nrg_\eps$ for proving exponential convergence to equilibrium in~\eqref{eq:eq}. Specifically, we show that (the two-component analogue of) displacement convexity cannot be expected for a functional of the form~$\nrg_\eps$ on the space $\prbtwotwo$.
\begin{proposition}
  \label{prp:nonconvex}
  Assume that~$h$ is not identically zero, and that $\eps>0$.
  Then, there is no $\lambda\in\R$ such that~$\nrg_\eps$ is $\lambda$-convex along geodesics in $\prbtwotwo$.
  More specifically, for each $\omega\in\Rp$,
  there are functions $u^\omega,v^\omega\in\prbtwo\cap C^\infty(\R^d)$
  such that
  \begin{align}
    \label{eq:nonconvex}
    \frac{\dn^2}{\dd s^2} \bigg|_{s=0}\intrd H_\eps(T_s\#u^\omega,v^\omega)\dd x \le -C(\omega-1),
  \end{align}
  where~$T_s$ is the translation by $s\ge0$ in $x_1$-direction, i.e., $T_s(x)=x+s \ee_1$ and~$H_\eps$ is the function defined in~\eqref{eqn_def_H}. 
\end{proposition}
Note that, by Lemma~\ref{lem:doublegeodesic}, the curve $(T_s\#u^\omega,v^\omega)_{0\le s\le1}$
is a geodesic in $\prbtwotwo$.
\begin{remark}
  With a little technical effort,
  the construction in the proof below can be used to show that such pairs $(u_\omega,v_\omega)$
  are actually dense in $\prbtwotwo$.
  For the sake of clarity, we only give the construction for one such pair.
\end{remark}
\begin{proof}
  In view of $h\equiv0$ on $\partial\Rnn^2$ because of the degeneracy condition \eqref{hyp:hvanish},
  $h$ not being identically zero implies the existence of a some $(U,V)\in\Rp^2$ such that $\partial_{uv}h(U,V)\neq0$.
  For the construction below, we assume $\partial_{uv}h(U,V)>0$, and we comment on the other case at the end of the proof.
  Choose $u^0,v^0\in\prbtwo\cap C^\infty_c(\R^d)$ such that
  $u^0(x)=U$ and $v^0(x)=V$ for all $|x|<r$, with some sufficiently small $r>0$.
  For all sufficiently large $\omega>0$, define $u^\omega,v^\omega\in\prbtwo\cap C^\infty_c(\R^d)$
  by
  \begin{align}
    \label{eq:vomega}
    u^\omega(x) = u^0(x)+\omega^{-1/2}\delta_r(x)\sin(\omega x_1),
    \quad
    v^\omega(x) = v^0(x)+\omega^{-1/2}\delta_r(x)\sin(\omega x_1),
  \end{align}
  where $\delta_r\in C^\infty_c(\R^d)$ is radially symmetric about the origin,
  with $\delta_r(x)=1$ for $|x|<r/2$ and $\delta_r(x)=0$ for $|x|>r$.
  For the integral in~\eqref{eq:nonconvex}, we obtain via an integration by parts:
  \begin{align*}
   \frac{\dn^2}{\dd s^2} \bigg|_{s=0}\intrd H_\eps(T_s\#u^\omega,v^\omega)\dd x
    &= \intrd \big[\partial_uH_\eps(u^\omega,v^\omega)\partial_{x_1x_1}u^\omega
      + \partial_{uu}H_\eps(u^\omega,v^\omega)\big(\partial_{x_1}u^\omega\big)^2\big]\dd x\\
    &= -\intrd\partial_{uv}H_\eps(u^\omega,v^\omega)\partial_{x_1}u^\omega\partial_{x_1}v^\omega\dd x.
  \end{align*}
  By construction, the contribution of this integral over $|x|<r$
  is roughly proportional to~$\omega$, with a negative sign since $\partial_{uv}H_\eps(U,V)>0$.
  The contribution on $|x|>r$ has some finite value, independent of~$\omega$. 

  Now if $\partial_{uv}h(U,V)$ is negative instead of positive,
  we only change the definition of $v^\omega$ in \eqref{eq:vomega} above into
  \begin{align*}
    v^\omega(x) = v^0(x)-\omega^{-1/2}\delta_r(x)\sin(\omega x_1),    
  \end{align*}
  which makes the product $\partial_{x_1}u^\omega\partial_{x_1}v^\omega$ negative instead of positive for $|x|<r/2$.
\end{proof}


\section{Stationary solutions}
\label{sct:stationary}

In this section, Theorem~\ref{thm:stationary} is proven.
It is a consequence of the (more detailed) results stated in Propositions~\ref{prp:steady} and~\ref{prp:regularity} below,
For brevity, define $H_\eps \colon \Rnn^2\to\R$ by
\begin{equation}
\label{eqn_def_H}
  H_\eps (u,v) \coloneqq  F(u) + G(v) + \eps h(u,v),
\end{equation}
which allows to write
\begin{align*}
  \nrg_\eps(u,v) = \intrd \big[H_\eps(u,v) + u\Phi+v\Psi\big]\dd x.
\end{align*}

\begin{remark}
  \label{remark_H_eps}
  We notice some important properties of the function $H_\eps \colon \Rnn^2 \to \R$ for $\eps\in[0,\eps^*]$, which are used in this section.
  First, by hypothesis~\eqref{hyp:hconvex},~$H_\eps$ is non-negative, strictly convex and satisfies the explicit convexity estimate
  \begin{equation}
    \label{eq:Hqual}
    \dff^2 H_\eps(u,v) \ge \frac12 \begin{pmatrix} F''(u) & 0 \\ 0 & G'(v) \end{pmatrix} \quad \text{for all } (u,v) \in \Rp^2 .
  \end{equation}
  Second, as a consequence of the strict convexity of~$H_\eps$,
  its differential~$\dff H_\eps$ is a strictly monotone continuous map on the cone~$\Rnn^2$,
  i.e., it satisfies $(\dff H_\eps(u,v)-\dff H_\eps(\tilde{u},\tilde{v}))\cdot(u-\tilde{u}, v - \tilde{v})>0$ for all $(u,v), (\tilde{u},\tilde{v})\in\Rnn^2$ with $(u,v) \neq (\tilde{u},\tilde{v})$.
  In view of the identities $\dff H_\eps(u,0)=(F'(u),0)$ and $\dff H_\eps(0,v)=(0,G'(v))$,
  and since $F'$ and $G'$ are monotone and unbounded with $F'(0)=G'(0)=0$, the image of~$\dff H_\eps$ is~$\Rnn^2$.
  Hence,~$\dff H_\eps$ is a homeomorphism of~$\Rnn^2$ onto itself,
  and also a homeomorphism of $\Rp^2$ onto itself.
\end{remark}
In the following, $\bar U\ge2$ denotes the smallest number such that
\begin{align}
  \label{eq:A}
  \frac12F'(\bar U) \ge F'(2) + dM + \eps^* \sup_{uv\le2}\partial_uh(u,v) \quad\text{and}\quad
  \frac12G'(\bar U) \ge G'(2) + dM + \eps^* \sup_{u,v\le2}\partial_vh(u,v),
\end{align}
with~$M$ from~\eqref{hyp:Phi}.
\begin{proposition}
  \label{prp:steady}
  Let $\eps\in[0,\eps^*]$. There exists a unique global minimizer $(\bar u_\eps,\bar v_\eps)$ of~$\nrg_\eps$.
  The components~$\bar u_\eps$,~$\bar v_\eps$ are continuous functions of compact support,
  bounded by~$\bar U$ from~\eqref{eq:A}.
  Moreover, there are constants $U_\eps,\,V_\eps>0$
  such that $\bar u_\eps,\,\bar v_\eps$ satisfy
  \begin{equation}
    \label{eq:EL}
    \begin{split}
      F'(\bar u_\eps) + \eps\partial_uh(\bar u_\eps,\bar v_\eps) &= (U_\eps - \Phi)_+\,, \\
      G'(\bar v_\eps) + \eps\partial_vh(\bar u_\eps,\bar v_\eps) &= (V_\eps - \Psi)_+\,.
    \end{split}
  \end{equation}
  The supports of~$\bar u_\eps$ and~$\bar v_\eps$ are convex and given
  by the closures of the sublevel sets
  \begin{align}
    \label{eq:supps}
    \spt{u} \coloneqq  \{\Phi< U_\eps\},
    \qquad
    \spt{v} \coloneqq  \{\Psi< V_\eps\},
  \end{align}
  respectively.
  Finally, there is an upper bound on~$U_\eps$ and~$V_\eps$,
  and also on the diameters of~$\spt{u}$ and~$\spt{v}$,
  uniformly for $0\le\eps\le\eps^*$.
\end{proposition}
\begin{remark}
  Thanks to hypothesis~\eqref{hyp:hvanish}, we have $\partial_uh(u,0)=\partial_vh(0,v)=0$,
  and thus the system~\eqref{eq:EL} can be made a bit more explicit:
  On $\spt{u}\setminus\spt{v}$, one has $\bar u_\eps = (F')^{-1}(U_\eps-\Phi)$,
  on $\spt{v}\setminus\spt{u}$, one has $\bar v_\eps = (G')^{-1}(V_\eps-\Psi)$.
  Finally, on $\spt{u}\cap\spt{v}$, the values of~$\bar u_\eps$ and~$\bar v_\eps$
  are obtained as pointwise solution of~\eqref{eq:EL},
  with right-hand sides $U_\eps-\Phi>0$ and $V_\eps-\Psi>0$.
\end{remark}
\begin{remark}
  The explicit representation~\eqref{eq:supps} of the supports is related to hypothesis~\eqref{hyp:hvanish},
  specifically to
  \begin{align}
    \label{eq:morevil}
    \partial_uh(0,v) = \partial_vh(v,0) = 0.
  \end{align}
  If just the part~\eqref{eq:morevil} of our set of hypotheses was removed,
  then the conclusions of Proposition~\ref{prp:steady} are essentially still valid,
  but the supports of~$\bar u_\eps$ and~$\bar v_\eps$ are only \emph{subsets}
  of the respective sublevel sets of~$\Phi$ and~$\Psi$ in general.

  For an illustration of this situation,
  consider the choice $F(u)=\frac{u^2}2$, $G(v)=\frac{v^2}2$ and $h(u,v)=uv$, for which~\eqref{eq:morevil} is false.
  Proceeding as in the proof of Proposition~\ref{prp:steady} below,
  one obtains existence and uniqueness of a minimizer with densities $(\bar u_\eps,\bar v_\eps)$,
  and~\eqref{eq:EL} turns into a \emph{linear} system for the values of~$\bar u_\eps$ and~$\bar v_\eps$ on the intersection of their supports.
  Thereon, the explicit solution is given by
  \begin{align*}
    (1-\eps^2)\bar u_\eps= (U_\eps-\Phi)_+-\eps(V_\eps-\Psi)_+,
    \quad
    (1-\eps^2)\bar v_\eps= (V_\eps-\Psi)_+-\eps(U_\eps-\Phi)_+.
  \end{align*}
  From this representation it is clear that if the respective sublevel sets of~$\Phi$ and~$\Psi$ overlap, then the supports of~$\bar u_\eps$ and~$\bar v_\eps$ are genuinely smaller. 
\end{remark}
\begin{proof}[Proof of Proposition~\ref{prp:steady}]
  Thanks to the assumed at least quadratic growth of~$F$ and~$G$,
  the qualified convexity~\eqref{eq:Hqual} and the $\Lambda$-convexity of~$\Phi$,~$\Psi$
  imply that
  \begin{align*}
    \nrg_\eps(u,v) \ge c\intrd (u^2+v^2)\dd x + \frac\Lambda 2\intrd |x|^2(u+v)\dd x - C
  \end{align*}
  for all $(u,v) \in [L^2(\R^d)]^2 \cap \prbtwotwo$, with some constants~$C$ and $c>0$.
  That is, the sublevel sets of~$\nrg_\eps$ are compact in $\prbtwotwo$, and weakly compact in $[L^2(\R^d)]^2$.
  Further, strict convexity of~$H_\eps$ and of~$\Phi$,~$\Psi$ imply strict convexity of~$\nrg_\eps$.
  Existence and uniqueness of the global minimizer $(\bar u_\eps,\bar v_\eps) \in [L^2(\R^d)]^2 \cap \prbtwotwo$
  now follow via the direct methods of the calculus of variations.
  
  Next, we verify that~$\bar U$ is an upper bound
  by showing that if~$\bar u_\eps$ or~$\bar v_\eps$ would exceed~$\bar U$,
  then there is a competitor $(\tilde u_\eps,\tilde v_\eps)$, bounded by~$\bar U$, of a lower $\nrg_\eps$-score.
  Assume that $\bar u_\eps> \bar U$ on a set $S\subset\R^d$ of positive Lebesgue measure;
  if $\bar u_\eps\le \bar U$ a.e.~but $\bar v_\eps> \bar U$, the argument is analogous.
  Define
  \begin{align}
    \label{eq:sig}
    \sigma \coloneqq  \int_S (\bar u_\eps-\bar U)\dd x \in(0,1).
  \end{align}
  Consider the cube $Q\subset\R^d$ of volume $V=3$ (i.e., of side length $3^{1/d}$),
  centered around the minimum point~$\underline x_\Phi$ of~$\Phi$.
  Since~$\bar u_\eps$ and~$\bar v_\eps$ are of unit mass,
  the subsets of~$Q$ on which $\bar u_\eps\ge1$ or $\bar v_\eps\ge1$, respectively,
  are of measure at most one.
  Hence, there is a set $T\subset Q$ of unit Lebesgue measure on which $\bar u_\eps\le1$ and $\bar v_\eps\le1$.
  Since $\bar U\ge2$, the sets~$S$ and~$T$ are disjoint.
  We define~$\tilde u_\eps$ as a modification of~$\bar u_\eps$ as follows:
  we set $\tilde u_\eps\coloneqq \bar U$ on~$S$,
  we set $\tilde u_\eps\coloneqq \bar u_\eps+\sigma$ on~$T$,
  and we set $\tilde u_\eps\coloneqq \bar u_\eps$ otherwise.
  By definition of~$\sigma$, and since~$T$ is of unit measure,~$\tilde u_\eps$ is a probability density,
  and thus $(\tilde u_\eps,\bar v_\eps)$ is an admissible competitor.
  On the one hand, on~$S$, where $\tilde u_\eps=\bar U\le\bar u_\eps$,
  \begin{align*}
    H_\eps(\bar u_\eps,\bar v_\eps) - H_\eps(\tilde u_\eps,\bar v_\eps)
    \ge (\bar u_\eps-\tilde u_\eps)\partial_uH_\eps(\tilde u_\eps,\bar v_\eps)
    \ge \frac12 (\bar u_\eps-\bar U) F'(\bar U),
 \end{align*}
 using the convexity estimate~\eqref{eq:Hqual}, the degeneracy~\eqref{eq:degeneracy} of~$F$ and~\eqref{hyp:hvanish} of~$h$.
 Hence, recalling the definition~\eqref{eq:sig} of~$\sigma$ and the non-negativity of~$\Phi$,
  \begin{align*}
    \int_S \big(H_\eps(\bar u_\eps,\bar v_\eps) + \bar u_\eps\Phi + \bar v_\eps \Psi\big)\dd x
    \ge \int_S \big(H_\eps(\tilde u_\eps,\bar v_\eps) + \tilde u_\eps \Phi + \bar v_\eps \Psi\big)\dd x
    + \frac12F'(\bar U) \sigma.
  \end{align*}
  On the other hand, on~$T$, where $\tilde u_\eps=\bar u_\eps+\sigma\le2$,
  \begin{align*}
    H_\eps(\bar u_\eps,\bar v_\eps) - H_\eps(\tilde u_\eps,\bar v_\eps)
    \ge \big(\bar u_\eps-\tilde u_\eps\big)\partial_uH_\eps(\tilde u_\eps,\bar v_\eps)
    \ge -\Big[F'(2)+ \eps^* \sup_{a,b\le2}\partial_uh(a,b)\Big]\sigma.
  \end{align*}
  With~$T$ being of unit measure,
  and recalling that $0\le\Phi(x)\le \frac M2|x-\underline x_\Phi|^2\le  (3/2)^{2/d}dM/2\le dM$ for all $x\in T$ thanks to~\eqref{eq:Phiupbound},
  it follows that
  \begin{align*}
    \int_T \big(H_\eps(\bar u_\eps,\bar v_\eps) + \bar u_\eps \Phi + \bar v_\eps \Psi\big)\dd x
    &\ge \int_T \big(H_\eps(\tilde u_\eps,\bar v_\eps) + \tilde u_\eps \Phi + \bar v_\eps \Psi\big)\dd x \\
    &\qquad -\Big[ dM+F'(2)+\eps^* \sup_{a,b\le2}\partial_uh(a,b)\Big]\sigma.
  \end{align*}
  In summary, recalling the implicit definition~\eqref{eq:A} of~$\bar U$, we find
  \begin{align*}
    \nrg_\eps(\bar u_\eps,\bar v_\eps)
    \ge \nrg_\eps(\tilde u_\eps,\bar v_\eps) + \sigma.
  \end{align*}
  This contradicts the minimality of $(\bar u_\eps,\bar v_\eps)$.
  Consequently, we have $\bar u_\eps\le \bar U$ and $\bar v_\eps\le \bar U$.
  
  To characterize the minimizer, we perform variations of the form
  \begin{align*}
    \bar u_\eps^s = (1-\alpha s)\bar u_\eps+s\xi,
    \qquad
    \bar v_\eps^s = (1-\beta s)\bar v_\eps+s\eta,
  \end{align*}
  with appropriate functions $\xi,\eta\in L^\infty(\R^d)$ of compact support,
  and parameters
  \begin{align*}
    \alpha = \intrd \xi\dd x, \qquad \beta = \intrd\eta\dd x.
  \end{align*}
  We define
  \begin{align*}
    U_\eps \coloneqq  \intrd \big[\partial_u H_\eps(\bar u_\eps,\bar v_\eps)+\Phi\big]\bar u_\eps\dd x,
    \quad
    V_\eps \coloneqq  \intrd \big[\partial_v H_\eps(\bar u_\eps,\bar v_\eps)+\Psi\big]\bar v_\eps\dd x.
  \end{align*}  
  First, let $\xi,\eta\in C_c(\R^d)$ be non-negative.
  Then $\bar u_\eps^s,\,\bar v_\eps^s\in\prbtwo$ for all $s\ge0$ sufficiently small,
  and
  \begin{align*}
    0&\le \lim_{s\downarrow0}\frac{\nrg_\eps(\bar u_\eps^s,\bar v_\eps^s)-\nrg_\eps(\bar u_\eps,\bar v_\eps)}s \\
     &= \intrd \big[\partial_u H_\eps(\bar u_\eps,\bar v_\eps)+\Phi\big]\xi\dd x
       + \intrd\big[\partial_v H_\eps(\bar u_\eps,\bar v_\eps)+\Psi\big]\eta\dd x
       -\alpha U_\eps - \beta V_\eps.
  \end{align*}
  This shows that, a.e.~on~$\R^d$,
  \begin{equation}
    \label{eq:EL2}
    \begin{split}
      F'(\bar u_\eps)+\eps\partial_uh(\bar u_\eps,\bar v_\eps)+\Phi &\ge U_\eps, \\
      G'(\bar v_\eps)+\eps\partial_vh(\bar u_\eps,\bar v_\eps)+\Psi &\ge V_\eps.    
    \end{split}
  \end{equation}
  Consequently, with $\partial_uh(0,v)=0$ for all $v\ge0$ by hypothesis~\eqref{hyp:hvanish} and the fact that~$F$ is degenerate at zero to first order by hypothesis~\eqref{eq:degeneracy},
  we necessarily have $\bar u_\eps>0$ a.e.~on $\{U_\eps>\Phi\}$,
  and similarly $\bar v_\eps>0$ a.e.~on $\{V_\eps>\Psi\}$.

  Next let $\xi,\eta\in C_c(\R^d)$ be arbitrary.
  For any $\delta>0$,
  we may perform the aforementioned variations with
  \begin{align*}
    \xi_\delta \coloneqq 
    \begin{cases}
      \xi & \text{where $\bar u_\eps>\delta$} \\
      0 & \text{otherwise}
    \end{cases},
          \quad
          \eta_\delta \coloneqq 
          \begin{cases}
            \eta & \text{where $\bar v_\eps>\delta$} \\
            0 & \text{otherwise}
          \end{cases},
  \end{align*}
  even for all $-s\ge0$ sufficiently small,
  and can thus conclude the opposite inequalities in~\eqref{eq:EL2}.
  This means that equality holds in the first inequality in~\eqref{eq:EL2} a.e.~on $\{\bar u_\eps>0\}$,
  and in the second inequality a.e.~on $\{\bar v_\eps>0\}$.

  We next consider the case that $\bar u_\eps>0$ and $\bar v_\eps=0$.
  Then $F'(\bar u_\eps)>0$, and since $\partial_uh(\bar u_\eps,\bar v_\eps)=0$ by hypothesis~\eqref{hyp:hvanish},
  it follows that $\Phi<U_\eps$.
  Analogously, $\bar v_\eps>0$ and $\bar u_\eps=0$ implies $\Psi<V_\eps$.
  Finally, suppose $\bar u_\eps>0$ and $\bar v_\eps>0$,
  so that~\eqref{eq:EL2} becomes a system of two equations.
  Recalling the definition of~$H_\eps$ above,
  that system can be written as a single vectorial equation as follows:
  \begin{align*}
    \dff H_\eps(\bar u_\eps,\bar v_\eps) =
    \begin{pmatrix} U_\eps - \Phi \\ V_\eps-\Psi \end{pmatrix}.
  \end{align*}
  By Remark \ref{remark_H_eps},
  the positive cone~$\Rp^2$ is mapped into itself under~$\dff H_\eps$.
  Therefore, $\bar u_\eps>0$ and $\bar v_\eps>0$ implies that $\Phi<U_\eps$ and $\Psi<U_\eps$, respectively.
  Consequently, the positivity sets of~$\bar u_\eps$ and~$\bar v_\eps$
  are indeed given by the sublevel sets~$\spt{u}$ and~$\spt{v}$ from~\eqref{eq:supps}, respectively.
  
  To sum up: a.e.~on $\{\Phi<U_\eps\}$,
  we have $\bar u_\eps>0$ and equality in the first inequality of~\eqref{eq:EL2};
  and a.e.~on the complement $\{\Phi\ge U_\eps\}$, we have $u_\eps=0$,
  which, thanks to $F'(0)=0$ and $\partial_uh(0,\bar v_\eps)=0$,
  can be written as
  \begin{align*}
    F'(\bar u_\eps)+\eps\partial_uh(\bar u_\eps,\bar v_\eps)= 0.
  \end{align*}
  This (and an analogous argument for~$\bar v_\eps$) implies~\eqref{eq:EL}.

  Next, concerning the uniform boundedness of~$U_\eps$ and~$V_\eps$:
  it suffices to observe that, in view of $\bar u_\eps,\bar v_\eps\le \bar U$,
  \begin{align*}
    \partial_u H_\eps\big(\bar u_\eps,\bar v_\eps)
    \le C\coloneqq F'(\bar U) + \eps^* \sup_{a,b\le \bar U}h(a,b),
  \end{align*}
  and further that at $x=\underline x_\Phi$,
  \begin{align*}
    U_\eps = U_\eps-\Phi(\underline x_\Phi) = \partial_u H_\eps\big(\bar u_\eps(\underline x_\Phi),\bar v_\eps(\underline x_\Phi)\big).
  \end{align*}
  In combination, this yields $U_\eps\le C$, and we can argue analogously for~$V_\eps$.  
  We further notice that the uniform boundedness of the sets~$\spt{u}$ and~$\spt{v}$ is a direct consequence of the estimates in~\eqref{eq:Phiupbound} and the uniform bound on the constants~$U_\eps$ and~$V_\eps$.
  
  Finally, we verify that the $L^2$-representatives of~$\bar u_\eps$ and~$\bar v_\eps$
  that are given as pointwise solution of~\eqref{eq:EL} are continuous.
  To that end, we write~\eqref{eq:EL} as
  \begin{align}
    \label{eq:Hsys}
    \dff H_\eps(u_\eps,v_\eps) =
    \begin{pmatrix}
      (U_\eps-\Phi)_+ \\ (V_\eps-\Psi)_+
    \end{pmatrix}.
  \end{align}
  Since~$\dff H_\eps$ has a continuous inverse, cp.~Remark~\ref{remark_H_eps}, and since the functions on the right-hand side of~\eqref{eq:Hsys} are continuous on~$\R^d$,
  so are the solutions~$\bar u_\eps$ and~$\bar v_\eps$.
\end{proof}
\begin{proposition}
  \label{prp:regularity}
  In addition to the general hypotheses, assume that $(F,G,h)$ is $k$-degenerate for some $k\in\N$.
  Then $F'(\bar u_\eps)$ and $G'(\bar v_\eps)$ are~$k$ times continuously differentiable
  in~$\spt{u}$ and in~$\spt{v}$, respectively.
  Moreover, all partial derivatives $\partial^\alpha F'(\bar u_\eps)$ and $\partial^\alpha G'(\bar v_\eps)$ of order $|\alpha|\le k$
  are bounded on~$\spt{u}$ and~$\spt{v}$, respectively, uniformly in $\eps\in[0,\eps^*]$.
\end{proposition}
\begin{remark}
  \label{rmk:Liponly}
  The essential point of Proposition~\ref{prp:regularity} is that $F'(\bar u_\eps)$ is~$k$ times continuously differentiable
  across the boundary of the support~$\partial\spt{v}$ of \emph{the other} component~$\bar v_\eps$, and vice versa.
  Across the boundary of \emph{its own} support~$\partial\spt{u}$,
  the function $F'(\bar u_\eps)$ is generally Lipschitz but no better,
  see Remark~\ref{rmk:reallyLiponly} after the proof.
  This translates into mere H\"older continuity for~$\bar u_\eps$,
  and in particular, one cannot expect the derivatives of~$\bar u_\eps$ itself to be bounded on~$\spt{u}$.
\end{remark}
\begin{proof}
  The claim will follow by an application of the inverse function theorem.  
  To that end, define the map $\themap_\eps \colon \R^2\to\R^2$ by
  \begin{align*}
    \themap_\eps(\rho,\eta) =
    \dff H_\eps\big((F')^{-1}(\rho),(G')^{-1}(\eta)\big) =
    \begin{pmatrix}
      \rho + \eps\theta_u(\rho,\eta) \\
      \eta + \eps\theta_v(\rho,\eta) 
    \end{pmatrix},
  \end{align*}
  with $\theta_u,\,\theta_v$ from~\eqref{eq:deftheta},
  and the convention that $\theta_u(\rho,\eta)=\theta_v(\rho,\eta)=0$ if $\rho\le0$ or $\eta\le0$.
  By $k$-degeneracy,~$\themap_\eps$ is $C^k$-regular.
  
  Recall from Remark~\ref{remark_H_eps} that~$\dff H_\eps$ is a homeomorphism of~$\Rnn^2$. Since~$F'$ and~$G'$ are continuous and strictly monotone on~$\Rnn$, the restriction of~$\themap_\eps$ to~$\Rnn^2$ possesses a continuous inverse as well.
  Moreover, on $\R^2\setminus\Rp^2$, the inverse of $\themap_\eps$ is simply the identity.
  In conclusion, $\themap_\eps$ has a global continuous inverse $\themap_\eps^{-1} \colon \R^2\to\R^2$
  that is the identity on $\R^2\setminus\Rp^2$.

  Next, we show that the inverses of the derivative matrices
  \begin{align}
    \label{eq:thetaprime}
    \dff\themap_\eps =
    \begin{pmatrix}
      1+\eps\partial_\rho\theta_u & \eps\partial_\eta\theta_u \\
      \eps\partial_\rho\theta_v & 1+\eps\partial_\eta\theta_v
      \end{pmatrix}
  \end{align}
  are locally bounded.
  On $\R^2\setminus\Rp^2$, this is trivial since
  \begin{align}
    \label{eq:deriveunit}
    \dff\themap_\eps(\rho,\eta) =
    \begin{pmatrix}
      1 & 0 \\ 0 & 1
    \end{pmatrix}
                   \quad\text{if $\rho\le0$ or $\eta\le0$}.
  \end{align}
  On~$\Rp^2$, we use the alternative representation
  \begin{align*}
    \dff\themap_\eps (\rho,\eta) =
    \dff^2 H_\eps(u,v)\,
    \begin{pmatrix}
      \frac1{F''(u)} & 0 \\
      0 & \frac1{G''(v)}
    \end{pmatrix},
          \quad\text{with}\quad u=(F')^{-1}(\rho),\,v=(G')^{-1}(\eta).
  \end{align*}
  Since $u,v>0$, the Hessian matrix $\dff^2H_\eps(u,v)$ is positive definite.
  Since $R\ge S$ for symmetric, positive definite matrices implies $\det R\ge\det S>0$, 
  the estimate~\eqref{eq:Hqual} shows that
  \begin{align}
    \label{eq:detbelow}
    \det \dff\themap_\eps (\rho,\eta) \ge \frac14 F''(u)G''(u) \cdot \frac1{F''(u)G''(v)} = \frac14.
  \end{align}
  It follows that the inverses
  \begin{align*}
    (\dff\themap_\eps)^{-1}
    = \frac1{\det\dff\themap_\eps}
    \begin{pmatrix}
      1+\eps\partial_\eta\theta_v & -\eps\partial_\eta\theta_u \\
      -\eps\partial_\rho\theta_v & 1+\eps\partial_\rho\theta_u
    \end{pmatrix}
  \end{align*}
  are bounded on any compact subset of~$\Rnn^2$, uniformly with respect to $\eps\in[0,\eps^*]$.
  The inverse function theorem is applicable and shows that~$\themap_\eps^{-1}$ is of class~$C^k$,
  with derivatives up to $k$-th order $\eps$-uniformly bounded on each compact set of~$\R^2$.

  To conclude regularity of $F'(\bar u_\eps)$ and $G'(\bar v_\eps)$ from here,
  we perform the following change of variables:
  \[ \rho(x) = F'(u(x)), \quad \eta(x) = G'(v(x)). \]
  Written in terms of $\bar\rho_\eps\coloneqq F'(\bar u_\eps)$ and $\bar\eta_\eps\coloneqq G'(\bar v_\eps)$,
  the system~\eqref{eq:EL} of Euler--Lagrange equations becomes
  \begin{align}
    \label{eq:thetasys}
    \themap_\eps(\bar\rho_\eps,\bar\eta_\eps) =
    \begin{pmatrix}
      (U_\eps-\Phi)_+ \\ (V_\eps-\Psi)_+
    \end{pmatrix},
  \end{align}
  and its solution is given by
  \begin{align}
    \label{eq:ithetasys}
    \begin{pmatrix}
      \bar\rho_\eps \\ \bar\eta_\eps
    \end{pmatrix}
    = (\themap_\eps)^{-1}
    \bigg(\begin{pmatrix}
      (U_\eps-\Phi)_+ \\ (V_\eps-\Psi)_+
    \end{pmatrix}\bigg)
    .
  \end{align}
  On $\spt{u}\cap\spt{v}$, where we have
  \begin{align*}
    (U_\eps-\Phi)_+ = U_\eps-\Phi, \quad (V_\eps-\Psi)_+    = V_\eps-\Psi,
  \end{align*}
  it now follows directly from~\eqref{eq:ithetasys}
  that~$\bar\rho_\eps$ and~$\bar\eta_\eps$ inherit the $C^k$-regularity of~$\Phi$ and~$\Psi$.
  Recalling that $\bar u_\eps,\bar v_\eps\le\bar U$,
  we conclude from the $\eps$-uniform local boundedness of the derivatives of~$\themap_\eps^{-1}$
  that also the partial derivatives of~$\bar\rho_\eps$ and~$\bar \eta_\eps$ of order $\le k$ are bounded,
  uniformly in $\eps\in[0,\eps^*]$ on $\spt{u}\cap\spt{v}$.

  Next, recalling that $\themap_\eps(\rho,0)=(\rho,0)$,
  we observe that on $\spt{u}\setminus\overline{\spt{v}}$,
  \begin{align}
    \label{eq:thetrivialone}
    \bar\rho_\eps = U_\eps-\Phi.
  \end{align}
  Therefore,~$\bar\rho_\eps$ inherits the $C^k$-regularity and bounds from~$\Phi$.
  The analogous statement holds for~$\bar\eta_\eps$ on $\spt{v}\setminus\overline{\spt{u}}$.
  
  It remains to verify the existence and continuity of all partial derivatives $\partial^\alpha\bar\rho_\eps$ of order $|\alpha|\le k$
  across the interfaces $\spt{u}\cap\partial\spt{v}$.
  We will do this by showing that at any point $x^*\in \spt{u}\cap\partial\spt{v}$,
  the limits of $\partial^\alpha\bar\rho_\eps$ from the inside and from the outside of~$\spt{v}$ agree.
  Since, according to~\eqref{eq:thetrivialone}, the outside limit amounts to $-\partial^\alpha\Psi$,
  which is smooth on~$\partial\spt{v}$,
  this also proves continuity of~$\bar \rho_\eps$ across the boundary $\spt{u}\cap\partial\spt{v}$.

  Thus, let $x^*\in \spt{u}\cap\partial\spt{v}$ be fixed.
  At points $x\in\spt{u}\cap\spt{v}$,~\eqref{eq:ithetasys} simplifies to
  \begin{align*}
    \begin{pmatrix}
      \bar\rho_\eps(x)) \\ \bar\eta_\eps(x))
    \end{pmatrix}
    = \themap_\eps^{-1}
    \bigg(
    \begin{pmatrix}
      U_\eps-\Phi(x) \\ V_\eps-\Psi(x)
    \end{pmatrix}\bigg).    
  \end{align*}
  We start by considering first derivatives. 
  Since
  \begin{align*}
    \dff\big(\themap_\eps^{-1}\big) = (\dff\themap_\eps)^{-1}\circ\themap_\eps^{-1}
    = \left[\frac1{\det\dff\themap_\eps}
    \begin{pmatrix}
      1+\eps\partial_\eta\theta_v & -\eps\partial_\eta\theta_u \\
      -\eps\partial_\rho\theta_v & 1+\eps\partial_\rho\theta_u
    \end{pmatrix}\right]
                                   \circ\themap_\eps^{-1},
  \end{align*}
  it follows that 
  \begin{align}
    \label{eq:1st}
    \partial_{x_k}\bar\rho_\eps = \frac1{\det\dff\themap_\eps(\bar\rho_\eps,\bar\eta_\eps)}
    \big[-\big(1+\eps\partial_\eta\theta_v(\bar\rho_\eps,\bar\eta_\eps)\big)\partial_{x_k}\Phi
    -\eps\partial_\eta\theta_u(\bar\rho_\eps,\bar\eta_\eps)\big)\partial_{x_k}\Psi\big].
  \end{align}
  Since~$\bar\rho_\eps$ and~$\bar\eta_\eps$ are continuous at~$x^*$ with $\bar\eta_\eps(x^*)=0$, 
  it follows thanks to $k$-degeneracy
  that also $\partial_\rho\theta_u(\bar\rho_\eps,\bar\eta_\eps)$ and $\partial_\eta\theta_u(\bar\rho_\eps,\bar\eta_\eps)$
  converge to zero as $x\to x^*$. 
  Hence,~\eqref{eq:1st} implies $\partial_{x_k}\bar\rho_\eps(x)\to -\partial_{x_k}\Phi(x^*)$,
  as desired.
  Higher-order partial derivatives can now be obtained by induction on the degree of differentiability.
  Assume that for $\ell<k$, uniform boundedness of all partial derivatives of order $\le\ell$ has been shown;
  note that the proof for $\ell=1$ is above.
  Application of $\partial^{\alpha'}$ with $|\alpha'|=\ell$ to~\eqref{eq:1st} yields
  \begin{align*}
    \partial^{\alpha'}\partial_{x_k}\bar\rho_\eps
    = \sum_{\beta+\beta'+\beta''=\alpha'}
    & C_{\beta,\beta',\beta''}
      \partial^\beta\left(\frac1{\det\dff\themap_\eps(\bar\rho_\eps,\bar\eta_\eps)}\right) \\
    &\times\big[-\partial^{\beta'}\big(1+\eps\partial_\eta\theta_v(\bar\rho_\eps,\bar\eta_\eps)\big)\partial^{\beta''}\partial_{x_k}\Phi
      -\eps\partial^{\beta'}\partial_\eta\theta_u(\bar\rho_\eps,\bar\eta_\eps)\big)\partial^{\beta''}\partial_{x_k}\Psi\big],
  \end{align*}
  with certain combinatorial coefficients $C_{\beta,\beta',\beta''}$.
  The partial derivatives of~$\Phi$ and~$\Psi$ remain uniformly bounded as $x\to x^*$.
  The same is true for the partial derivatives of $1/\det\dff\themap_\eps$;
  the entries in $\dff\themap_\eps$ are $C^{k-1}$-regular functions by the $k$-degeneracy condition,
  and $\det\dff\themap_\eps$ is bounded below by virtue of~\eqref{eq:detbelow}.
  Further, observe that $\partial^{\beta'}\partial_\eta\theta_u(\bar\rho_\eps,\bar\eta_\eps)$ can be written
  as a finite weighted sum, where each term is
  a partial derivative (with respect to $(\rho,\eta)$) of order $\le\ell+1=k$ of~$\theta_u$ at $(\bar\rho_\eps(x),\bar\eta_\eps(x))$
  --- and hence continuous ---
  multiplied by a product of partial derivatives (with respect to~$x$) of~$\bar\rho_\eps$ and~$\bar\eta_\eps$ of order~$\le\ell$
  --- and hence uniformly bounded by induction hypothesis.
  As above, using that $\bar\eta_\eps(x)\to0$ as $x\to x^*$,
  it follows
  by means $k$-degeneracy
  that $\partial^{\beta'}\partial_\eta\theta_u\big(\bar\rho_\eps(x),\bar\eta_\eps(x)\big) \to 0$
  and that $\partial^{\beta'}\partial_\eta\theta_v\big(\bar\rho_\eps(x),\bar\eta_\eps(x)\big) \to 0$
  as $x\to x^*$.
  And so,
  \begin{align*}
    \partial^{\alpha'}\partial_{x_k}\bar\rho_\eps(x) \to -\partial^{\alpha'}\partial_{x_k}\Phi(x^*),
  \end{align*}
  finishing the proof.
\end{proof}
\begin{remark}
  \label{rmk:reallyLiponly}
  If $(F,G,h)$ is $1$-degenerate, then~\eqref{eq:deriveunit} implies
  that~$\themap_\eps^{-1}$ is a differentiable perturbation of the identity near any point $(0,z)$ with $z\in\R$. 
  Thus, the first component~$\bar\rho_\eps$ in~\eqref{eq:ithetasys}
  is of the form $(U_\eps-\Phi)_+$ plus some differentiable perturbation near the boundary~$\partial\spt{u}$.
  This confirms Remark~\ref{rmk:Liponly} that in general only Lipschitz regularity
  can be expected from $F'(\bar u_\eps)$ across the boundary of its own support.
\end{remark}
\begin{example}
  The following example illustrates that without the hypothesis of $2$-degeneracy,
  one cannot expect $C^2$-regularity of $F'(\bar u_\eps)$ and $G'(\bar v_\eps)$.
  We consider 
  \begin{align*}
    F(u) = \frac{u^2}2, \quad G(v) = \frac{v^2}2, \quad h(u,v) = (uv)^2,
  \end{align*}
  which is $1$-degenerate, but not $2$-degenerate.
  The system~\eqref{eq:EL} attains the form
  \begin{align*}
    \bar u_\eps(1+2\eps\bar v_\eps^2) = (U_\eps-\Phi)_+, \quad \bar v_\eps(1+2\eps \bar u_\eps^2) = (V_\eps-\Psi)_+\,.
  \end{align*}
  If $x^*\in\spt{u}\cap\partial\spt{v}$,
  then $\bar v_\eps\approx(V_\eps-\Psi)_+$ is Lipschitz but not differentiable at~$x^*$, see Remark~\ref{rmk:reallyLiponly} above.
  Thus~$\bar v_\eps^2$ is once continuously differentiable, but fails to be twice differentiable at~$x^*$. Consequently,
  \[ F'(\bar u_\eps) = \bar u_\eps = \frac{U_\eps-\Phi}{1+2\eps\bar v_\eps^2}\]
  fails to be twice differentiable at~$x^*$. 
\end{example}
\begin{corollary}
  \label{cor:semiconvex}
  Assume that $(F,G,h)$ is $2$-degenerate.
  Then $\partial_uh(\bar u_\eps,\bar v_\eps)$ and $\partial_vh(\bar u_\eps,\bar v_\eps)$
  are $\eps$-uniformly semi-convex, that is, 
  there is a $K_0\ge0$ such that
  \begin{align}
    \label{eq:semiconvex}
    \nabla^2 \partial_uh(\bar u_\eps,\bar v_\eps) \ge -K_0\eins,
    \quad
    \nabla^2 \partial_vh(\bar u_\eps,\bar v_\eps) \ge -K_0\eins
  \end{align}
  on~$\R^d$, for all $\eps\in[0,\eps^*]$.
\end{corollary}
\begin{proof}
  By Proposition~\ref{prp:regularity} above,
  it follows that the gradients and the Hessians of $F'(\bar u_\eps)$ and $G'(\bar u_\eps)$
  are $\eps$-uniformly bounded on the respective supports~$\spt{u}$ and~$\spt{v}$.
  With
  \begin{align*}
    \partial_uh(\bar u_\eps,\bar v_\eps)
    = \theta_u\big(F'(\bar u_\eps),G'(\bar v_\eps)\big),
  \end{align*}
  and with $\theta_u\in C^2(\Rnn^2)$,
  the first and second order derivatives of $\partial_uh(\bar u_\eps,\bar v_\eps)$
  are $\eps$-independently bounded on~$\spt{u}$.
  In fact, omitting the arguments, we have the representation
  \begin{align*}
    \nabla^2  \partial_uh
    &= \partial_\rho\theta_u\,\nabla^2F' + \partial_\eta\theta_u\,\nabla^2G' \\
    &\quad + \partial_{\rho\rho}\theta_u\,\nabla F'\otimes\nabla F'
      + \partial_{\eta\eta}\theta_u\,\nabla G'\otimes\nabla G'
      + \partial_{\rho\eta}\theta_u\,\big(\nabla F'\otimes\nabla G'+\nabla G'\otimes\nabla F'\big).
  \end{align*}
  Fix some~$x^*$ at the boundary of the support~$\spt{u}$, and consider the expression above for~$x$ approaching~$x^*$ from inside~$\spt{u}$.
  The gradients and Hessians of~$F'$ and~$G'$ remain bounded,
  whereas the partial derivatives of~$\theta_u$ converge to zero because of $2$-degeneracy,
  and because their argument $F'(\bar u_\eps)$ converges to zero.
  Thus $\nabla^2 \partial_uh(\bar u_\eps,\bar v_\eps)$ is continuous across the boundary of the support~$\spt{u}$.
  This implies an $\eps$-uniform bound on the second derivatives of $\partial_uh(\bar u_\eps,\bar v_\eps)$,
  and in particular the semi-convexity estimate~\eqref{eq:semiconvex}.
\end{proof}


\section{Time-discrete variational approximation}
\label{sct:discrete}

In this section, we assume all hypotheses on~$\Phi$,~$\Psi$,~$F$,~$G$ and~$h$ from Section~\ref{sct:hypo},
and further that $(F,G,h)$ is $2$-bounded, $2$-degenerate, and satisfies the swap condition.
We assume that $\eps \in [0,\bar\eps]$ is fixed,
where $\bar \eps \in (0,\eps^*]$ is chosen such that 
\begin{equation}
\label{eqn_eps_bar}
 12 \bar{\eps}^2 (A^2+W) \leq 1 \quad \text{and} \quad 2 K_0 \bar \eps \leq \Lambda,
\end{equation}
with~$\eps^*$ from hypothesis~\eqref{hyp:hconvex},~$A$ the constant from Remark~\ref{remark_consequence_hypotheses} related to the $2$-boundedness and $2$-degeneracy of $(F,G,h)$,~$W$ the constant from the swap condition~\eqref{hyp:swap},~$K_0$ the constant determined in Corollary~\ref{cor:semiconvex}, and~$\Lambda$ the lower ellipticity bound of~$\Phi, \Psi$ from hypothesis~\eqref{hyp:Phi}. Note that this choice of~$\bar \eps$ in particular implies non-negativity and strict convexity of~$\nrg_\eps$.

As above, $(\bar u_\eps,\bar v_\eps) \in\prbtwotwo$ denotes the unique stationary pair of densities.
Further, recall the definition of~$\theta_u$ and~$\theta_v$ in~\eqref{eq:deftheta}.
We shall use the abbreviation
\begin{align}
  \label{eq:abbrev}
  \bar\Theta_u\coloneqq  \theta_u\big(F'(\bar u_\eps),G'(\bar v_\eps)\big) = \partial_uh(\bar u_\eps,\bar v_\eps),
  \quad
  \bar\Theta_v\coloneqq  \theta_v\big(F'(\bar u_\eps),G'(\bar v_\eps)\big) = \partial_vh(\bar u_\eps,\bar v_\eps).
\end{align}

\subsection{Yosida-regularization and Results}
The fundamental object that we use
for proving the results on existence and long-time asymptotics of solutions
is the following Yosida-type regularization~$\nrg_{\eps,\tau}$ of~$\nrg_\eps$
with a time step $\tau>0$:
\begin{align*}
  \nrg_{\eps,\tau}\big((u,v)\big|(\bar u,\bar v)\big)
  \coloneqq  \frac1{2\tau}\dwass\big((u,v),(\bar u,\bar v)\big)^2+\nrg_\eps(u,v).
\end{align*}
A time-discrete approximation of solutions to~\eqref{eq:eq} will be obtained in JKO-style~\cite{JKO} by means of inductive minimization of~$\nrg_{\eps,\tau}$.
The following certifies well-posedness of that induction.
\begin{lemma}
  \label{lem:jko}
  Given any pair $(\hat u,\hat v)\in\prbtwotwo$ of finite energy $\nrg_{\eps}(\hat u,\hat v) < \infty$ with $\eps \in [0,\eps^*]$,
  there is a unique minimizing pair $(u^*,v^*)\in\prbtwotwo$ of $\nrg_{\eps,\tau}\big(\cdot |(\hat u,\hat v)\big)$.
  Moreover, one has
  \begin{align}
    \label{eq:onestepmono}
    \nrg_\eps(u^*,v^*) + \frac{\tau}{2} \left( \frac{\dwass\big((u^*,v^*),(\hat u,\hat v)\big)}{\tau} \right)^2
    \le \nrg_\eps(\hat u,\hat v).    
  \end{align}
  In particular, $\nrg_\eps(u^*,v^*)\le\nrg_\eps(\hat u,\hat v)$.
\end{lemma}
\begin{proof}
  Existence and uniqueness follows from the direct methods in the calculus of variations, applied to the functional~$\nrg_{\eps,\tau}$ in $\prbtwotwo$.
  Just notice that~$\nrg_{\eps,\tau}$ is non-negative (as $\eps \in [0,\eps^*]$), 
  that $\wass(u,\hat u)^2$ is convex in~$u$,
  that $\wass(v,\hat v)^2$ is convex in~$v$,
  and that $\nrg_\eps(u,v)$ is strictly jointly convex in $(u,v)$.
  Coercivity on $\prbtwotwo$ follows thanks to the control on the second moments of~$u$ and~$v$ by $\nrg_\eps(u,v)$.
  Concerning the inequality~\eqref{eq:onestepmono}:
  by $(u^*,v^*)$ being a minimizer, we have
  \begin{align*}
    \nrg_\eps(u^*,v^*) + \frac{\dwass\big((u^*,v^*),(\hat u,\hat v)\big)^2}{2\tau}
    = \nrg_{\eps,\tau}\big((u^*,v^*)\big|(\hat u,\hat v)\big)
    \le  \nrg_{\eps,\tau}\big((\hat u,\hat v)\big|(\hat u,\hat v)\big)
    = \nrg_\eps(\hat u,\hat v),
  \end{align*}
  which proves the claim.
\end{proof}
The goal of this section is to prove the three results in Propositions \ref{prp:dweak}, \ref{prp:regular}, and \ref{prp:core} below on the minimizing pairs $(u^*,v^*)$.
In the statements of the propositions, it is understood that $\eps \in [0,\bar{\eps}]$ is fixed,
that $(\hat u,\hat v) \in \prbtwotwo$ is a given datum of finite energy $\nrg_{\eps}(\hat u,\hat v) < \infty$,
and that $(u^*,v^*)$ is the associated minimizer of $\nrg_{\eps,\tau}\big(\cdot |(\hat u,\hat v)\big)$ in $\prbtwotwo$.

The \emph{first result} shows that $(u^*,v^*)$ satisfies a time-discrete weak formulation of the evolution equations~\eqref{eq:eq}.
\begin{proposition}
  \label{prp:dweak}
  For each $\zeta\in C^\infty_c(\R^d)$ there holds
  \begin{equation}
    \label{eq:dweak}
    \begin{split}
      \intrd \frac{u^*-\hat u}{\tau}\zeta\dd x &= \intrd u^*\nabla\big[F'(u^*)+\eps\partial_uh(u^*,v^*)+\Phi\big]\cdot\nabla\zeta\dd x + R_u, \\
      \intrd \frac{v^*-\hat v}{\tau}\zeta\dd x &= \intrd v^*\nabla\big[G'(v^*)+\eps\partial_vh(u^*,v^*)+\Psi\big]\cdot\nabla\zeta\dd x + R_v,
    \end{split}
  \end{equation}
  with remainder terms~$R_u$ and~$R_v$ satisfying
  \begin{align}
    \label{eq:rest}
    |R_u|+|R_v| \le \|\zeta\|_{C^2}\big(\nrg_\eps(\hat u,\hat v)-\nrg_\eps(u^*,v^*)\big).
  \end{align}
\end{proposition}
The \emph{second result}, which is the key ingredient for our proof of Theorem~\ref{thm:transient},
is an estimate on $F'(u^*)$ and $G'(v^*)$ in~$H^1$.
It is formulated with help of the entropy functional~$\ent$, introduced in~\eqref{eq:entropy}.
For brevity, we also define~$\entt$ on $\prbtwotwo$ by $\entt(u,v)\coloneqq \ent(u)+\ent(v)$.
By Lemma~\ref{lem:gero} from the Appendix, $\entt(u,v)>-\infty$. 
\begin{proposition}
  \label{prp:regular}
  There is a constant~$C$ independent of $(\hat u,\hat v)$ such that
  \begin{equation}
    \label{eq:allregular}
    \begin{split}
      &\intrd \big( |\nabla F'(u^*)|^2 + |\nabla G'(v^*)|^2 \big) \dd x \\
      &\le C\bigg[1+\nrg_\eps(\hat u,\hat v)
      + \frac{\nrg_\eps(\hat u,\hat v)-\nrg_\eps(u^*,v^*)}\tau
      + \frac{\entt(\hat u,\hat v)-\entt(u^*,v^*)}\tau\bigg].
    \end{split}
  \end{equation}
\end{proposition}
The \emph{third result}, which contains the essence of the proof of Theorem~\ref{thm:longtime},
is concerned with proximity of $\hat u$ and $\hat v$ to the respective stationary solutions $\bar u_\eps$ and $\bar v_\eps$.
Recall for a given strictly convex function $J \colon \Rnn\to\R$ the definition of the Bregman divergence $d_J(\cdot|\cdot) \colon \Rnn\times\Rnn\to\Rnn$ as
\begin{align}
  \label{eq:bregman}
  d_J(s|\bar s) \coloneqq J(s) - \big[J(\bar s) + (s-\bar s)J'(\bar s)\big].
\end{align}
By strict convexity of~$J$, $d_J(s|\bar s)$ is always non-negative, and is zero if and only if $s=\bar s$.
Next, introduce the relative entropy functionals~$\lyp_1$ and~$\lyp_2$ on $\prbtwo$ by
\begin{align}
  \label{eq:lypuv}
  \lyp_1(u) \coloneqq  \intrd \big[ d_F(u|\bar u_\eps) + u(\Phi-U_\eps)_+\big]\dd x,
  \quad
  \lyp_2(v) \coloneqq  \intrd \big[ d_G(v|\bar v_\eps) + v(\Psi-V_\eps)_+\big]\dd x,  
\end{align}
which are clearly non-negative (as sums of non-negative parts), and zero if and only if $u=\bar u_\eps$ and $v=\bar v_\eps$, respectively.
Finally, define~$\lyp$ on $\prbtwotwo$ by
\begin{align}
  \label{eq:deflyp}
  \lyp(u,v) = \lyp_1(u) + \lyp_2(v).
\end{align}
\begin{proposition}
  \label{prp:core}
  There is a constant $K>0$ independent of $(\hat u,\hat v)$ such that
  \begin{align}
    \label{eq:coreest}
    \lyp(\hat u,\hat v) - \lyp(u^*,v^*) \ge 2\tau(\Lambda-K\eps) \lyp(u^*,v^*).
  \end{align}
\end{proposition}



Our strategy for proving Propositions \ref{prp:dweak}, \ref{prp:regular}, and \ref{prp:core} is the following.
First, we prove these results \emph{under the additional hypothesis on $(\hat u,\hat v)$}:
\begin{align}
  \label{eq:addhypo}
  \text{$\hat u$, $\hat v$ are positive a.e.\ on a ball $\ball_R$ of some radius $R>0$, and vanish a.e.\ outside.}
\end{align}
In subsection~\ref{sct:removehypo},
we remove this additional hypothesis and generalize the propositions to arbitrary data $(\hat u,\hat v)$ of finite energy.

In subsections \ref{sct:weakform}--\ref{sct:dissipation} below,
the datum $(\hat u,\hat v)$ is fixed and satisfies~\eqref{eq:addhypo}.
Accordingly, $(u^*,v^*)$ is the minimizer of $\nrg_{\eps,\tau}\big(\cdot |(\hat u,\hat v)\big)$ in $\prbtwotwo$.
To fix notations, denote by $(\varphi_u,\psi_u)$ and by $(\varphi_v,\psi_v)$ the optimal pairs of $c$-conjugate potentials
for the transports from $u^*$ to $\hat u$, and from $v^*$ to $\hat v$, respectively.
By Section~\ref{sct:wasserstein}, these potentials are $u^*\leb$-a.e.\ uniquely determined, up to a global constant;
we fix points $\bar x_u$ and $\bar x_v$ in the support of $u$ and of $v$, respectively,
and normalize $\varphi_u(\bar x_u)=0$ and $\varphi_v(\bar x_v)=0$. 

\subsection{Euler--Lagrange equation}
\label{sct:weakform}
In this section, we prove Proposition~\ref{prp:dweak} under the additional hypothesis~\eqref{eq:addhypo}.
\begin{lemma}
  \label{lem:optpotentials}
  There are constants $C_u,C_v\in\R$ such that
  \begin{equation}
    \label{eq:optpotentials}
    \begin{split}
      C_u-\frac{\varphi_u}{\tau} &= F'(u^*) + \eps\partial_uh(u^*,v^*) + \Phi \qquad \text{$u^*\leb$-a.e.}, \\
      C_v-\frac{\varphi_v}{\tau} &= G'(v^*) + \eps\partial_vh(u^*,v^*) + \Psi \qquad \text{$v^*\leb$-a.e.}.
    \end{split}
  \end{equation}
\end{lemma}
\begin{proof}
  We only consider the $u$-component.
  Let $\rho\in L^\infty(\R^d)$ be given, which vanishes outside of some open set $\Omega\subset\R^d$ with compact closure.
  We assume that
  \begin{align}
    \label{eq:tozero}
    \intrd \rho u^* \dd x = 0,
  \end{align}
  which can always be achieved by addition of a suitable multiple of the indicator function of~$\Omega$ to~$\rho$.
  For $\delta\in\R$ such that $0 < |\delta|\,\|\rho\|_{L^\infty}<1$,
  let $u^\delta\coloneqq (1+\delta\rho)u^*$, and note that $u^\delta\in\prbtwo$ thanks to identity~\eqref{eq:tozero}.
  Next, let $(\varphi_u^\delta,\psi_u^\delta)$ be an optimal pair of $c$-conjugate potentials
  for the transport from~$u^\delta$ to~$\hat u$ with the normalization $\varphi_u^\delta(\bar x_u)=0$.
  Recalling that $\hat u$ satisfies the additional hypothesis~\eqref{eq:addhypo},
  $\varphi_u^\delta$ is uniquely determined $u^* \leb$-a.e.\ (note that $u^\delta\leb$ and $u^*\leb$ have the same negligible sets),
  up to a global constant.
  Since $\bar x_u$ is in the support of $u^\delta$, we can normalize $\varphi_u^\delta$ by $\varphi_u^\delta(\bar x_u)=0$.
  
  For later reference, recall that the auxiliary potential $\tilde\varphi_u^\delta(x):=\frac12|x|^2-\varphi_u^\delta(x)$
  is a proper, lower semi-continuous and convex function.
  Moreover, as the associated optimal transport map $T_u^\delta=\nabla\tilde\varphi_u^\delta$
  maps $u^*\leb$-almost surely onto  the support of $\hat u$, i.e., into $\ball_R$,
  it follows that $|\nabla\tilde\varphi_u^\delta|\le R$ $u^*\leb$-a.e.
  For convenience and without loss of generality, we may actually assume that on $u^*\leb$-negligible sets,
  $\varphi_u^\delta$ is defined such that
  \begin{align}
    \label{eq:phiLip}
    |\nabla\tilde\varphi_u^\delta|\le R\quad \text{a.e. on $\R^d$}.
  \end{align}  
  Clearly, $(\varphi_u^\delta,\psi_u^\delta)$ is a (in general sub-optimal) pair of $c$-conjugate potentials for the transport from~$u^*$ to~$\hat u$.
  Recalling further the definition of $(u^*,v^*)$ as minimizer of $\nrg_{\eps,\tau}\big(\cdot\big|(\hat u,\hat v)\big)$,
  we conclude the following chain of inequalities:
  \begin{align*}
    &\frac1{\tau}\left(\intrd \varphi_u^\delta(x)u^*(x)\dd x+ \intrd\psi_u^\delta(y)\hat u(y)\dd y
      + \frac12\wass(v^*,\hat v)^2\right) + \nrg_\eps(u^*,v^*) \\
    &\le \nrg_{\eps,\tau}\big((u^*,v^*)\big|(\hat u,\hat v)\big) \\
    &\le \nrg_{\eps,\tau}\big((u^\delta,v^*)\big|(\hat u,\hat v)\big) \\
    & = \frac1{\tau}\left(\intrd \varphi_u^\delta(x)u^\delta(x)\dd x+ \intrd\psi_u^\delta(y)\hat u(y)\dd y
      + \frac12\wass(v^*,\hat v)^2\right) + \nrg_\eps(u^\delta,v^*) .
  \end{align*}
  It thus follows that
  \begin{align}
    \label{eq:dividebydelta}
    0\le\intrd\frac{\varphi_u^\delta}{\tau}\frac{u^\delta-u^*}\delta\dd x
    + \frac1\delta\big(\nrg_\eps(u^\delta,v^*)-\nrg_\eps(u^*,v^*)\big).
  \end{align}
  We shall now pass to the limit $\delta\downarrow0$.
  Since $\tilde\varphi_u^\delta(\bar x_u)=0$ is a fixed,
  and since~\eqref{eq:phiLip} gives a uniform Lipschitz bound,
  the Arzel\`{a}--Ascoli theorem yields local uniform convergence of $\tilde\varphi_u^{\delta_k}$
  to a limit $\tilde\varphi_u^0$ along a suitable sequence $\delta_k\downarrow0$.
  We wish to show that
  \begin{align}
    \label{eq:filippo}
    \tilde\varphi_u^0(x)=\tilde\varphi_u(x):=\frac12|x|^2-\varphi_u(x) \quad \text{for $u^*\leb$-a.e. $x$}.
  \end{align}
  By convexity of the $\tilde\varphi_u^{\delta_k}$, local uniform convergence of the function values
  implies $\leb$-a.e.\ convergence of the gradients.
  So in particular, $T^{\delta_k}_u\to T^0_u:=\nabla\tilde\varphi^0_u$ $u^*\leb$-a.e.
  But $T^0_u\#u^*=\hat u$, since for every $\omega\in C_c(\R^d)$:
  \begin{align*}
    \intrd \omega\,\hat u\dd y = \intrd \omega\, T^{\delta_k}_u\#u^\delta\dd y = \intrd \omega\circ T^{\delta_k}_u\, u^\delta\dd x
    \to \intrd \omega\circ T^0_u\, u^*\dd x = \intrd \omega\, T^0_u\#u^*\dd x.
  \end{align*}
  The limit in the chain above is justified by the dominated convergence theorem,
  since $\omega\circ T^{\delta_k}_u\to\omega\circ T^0_u$ $u^*\leb$-a.e.,
  since $\omega$ is bounded,
  and since $u^\delta-u^*=\delta\rho u^*$ by construction.
  This means that $\tilde\varphi^0_u$ is an auxiliary optimal potential for the transport from $u^*$ to $\hat u$.
  Using again $u^*\leb$-a.e.\ uniqueness of such an optimal potential up to a global constant thanks to~\eqref{eq:addhypo},
  and observing that $\tilde\varphi_u^0(\bar x_u)=0$ by local uniform convergence,
  we conclude~\eqref{eq:filippo}.  
  Now, since $(u^\delta-u^*)/\delta=u^*\rho$ with~$\rho$ bounded,
  it follows that
  \begin{align*}
    \lim_{k\to\infty}\intrd\frac{\varphi_u^{\delta_k}}{\tau}\frac{u^{\delta_k}-u^*}{\delta_k}\dd x
    = \intrd \frac{\varphi_u}{\tau}u^*\rho\dd x.    
  \end{align*}
  Note that local uniform convergence has been sufficient here since $\rho$ vanishes outside the compact set $\bar\Omega$.

  Further, since $H_\eps(u^*,v^*)$ is integrable on~$\R^d$,
  so is $H_\eps\big((1+\delta\rho)u^*,v^*)$ thanks to the doubling condition~\eqref{eq:doubling} and to~\eqref{hyp:hbound},
  and the variational derivative of~$\nrg_\eps$ in direction~$\rho u^*$ is readily computed by standard methods.
  With~\eqref{eq:dividebydelta}, this leads to
  \begin{align}
    \label{eq:oneside}
    0\le \intrd \left[\frac{\varphi_u}{\tau}+F'(u^*)+\Phi+\eps\partial_uh(u^*,v^*)\right]u^*\rho\dd x.
  \end{align}
  The same argument applies for $0>\delta>-1/\|\rho\|_{L^\infty}$, when the relation in~\eqref{eq:dividebydelta} is reversed.
  Consequently,~\eqref{eq:oneside} holds with the reversed relation as well, i.e., it is an equality.
  Since $\rho$ has been an arbitrary bounded function of compact support,
  only subject to the normalization~\eqref{eq:tozero},
  the term in square parenthesis above equals to a global constant $C_u$.
  The value of $C_u$ is determined by our normalization $\varphi_u(\bar x_u)=0$.
  This finishes the proof of the lemma.
\end{proof}
\begin{proof}[Proof of Proposition~\ref{prp:dweak} under the additional hypothesis~\eqref{eq:addhypo}]
  Let $T(x)=x-\nabla\varphi_u(x)$ be the optimal transport map from $u^*$ to $u$,
  Recalling that $\hat u=T\#u^*$ and the definition~\eqref{eq:defpush} of the push-forward,
  we obtain
  \begin{align*}
    \intrd \frac{u^*-\hat u}{\tau}\zeta\dd x
    &= \frac1{\tau} \left(\intrd \zeta(x)u^*(x)\dd x - \intrd \zeta\circ T(x)u^*(x)\dd x\right) \\
    &= \frac1{\tau}  \intrd \big( \zeta-\zeta\circ T \big) u^*\dd x \\
    & = \frac1{\tau}  \intrd \Big[ \nabla\zeta(x)\cdot\big(x-T(x)\big) - \frac12\big(x-T(x)\big)^T\nabla^2\zeta(m_x)\big(x-T(x)\big) \Big] u^*(x)\dd x,    
  \end{align*}
  where~$m_x$ is a suitable intermediate point on the line connecting~$x$ to~$T(x)$.
  In summary,
  \begin{align}
    \label{eq:preweak}
    \intrd \frac{u^*-\hat u}{\tau}\zeta\dd x
    = \intrd u^*\nabla\left(\frac{\varphi_u}{\tau}\right)\cdot\nabla\zeta\dd x + R_u,
  \end{align}
  where, thanks to~\eqref{eq:monge},
  \begin{align}
    \label{eq:prerest}
    |R_u| \le \frac1{2\tau}\sup\|\nabla^2\zeta\|\intrd |\nabla\varphi_u|^2u^*(x)\dd x
    \le \frac{\|\zeta\|_{C^2}}{2\tau}\wass(u^*,\hat u)^2.
  \end{align}
  Substitution of~\eqref{eq:optpotentials} -- which is relying on the additional hypothesis~\eqref{eq:addhypo} -- into~\eqref{eq:preweak} above produces the first equation in~\eqref{eq:dweak}.

  The $v$-component is treated in a similar way, leading to the second equation in~\eqref{eq:dweak}.
  Adding the two estimates of the form~\eqref{eq:prerest} and using~\eqref{eq:onestepmono} we obtain
  \begin{equation*}
    |R_u|+|R_v|  \le \frac{\|\zeta\|_{C^2}}{2\tau}\big[\wass(u^*,\hat u)^2+\wass(v^*,\hat v)^2\big]
    \le \|\zeta\|_{C^2}\big[\nrg_\eps(\hat u,\hat v)-\nrg_\eps(u^*,v^*)\big]. \qedhere
  \end{equation*}
\end{proof}

\subsection{Regularity estimates}
In this section, we prove Proposition~\ref{prp:regular}, subject to~\eqref{eq:addhypo}.
The proof follows from Lemmas~\ref{lem:onestepuFprime} and~\ref{lem:onestepuFprime_cut} below,
where the first condition in~\eqref{eqn_eps_bar} concerning the smallness of~$\bar \eps$ becomes relevant.
The additional hypothesis~\eqref{eq:addhypo} only enters indirectly, via Lemma~\ref{lem:optpotentials}.
\begin{lemma}
  \label{lem:onestepuFprime}
  With a constant $C$ independent of $(\hat u,\hat v)$, there holds
  \begin{align}
    \label{eq:highregular}
    \intrd \big[u^* |\nabla F'(u^*)|^2+ v^*|\nabla G'(v^*)|^2\big]\dd x
    \le C\left(\nrg_\eps(\hat u,\hat v)
    + \frac{\nrg_\eps(\hat u,\hat v)-\nrg_\eps(u^*,v^*)}{\tau}
    \right).
  \end{align}
  In particular, $\nabla F'(u^*)\in L^2(\R^d;u^*\leb)$ and $\nabla G'(v^*)\in L^2(\R^d;v^*\leb)$.
\end{lemma}
\begin{proof}
  Thanks to Lemma~\ref{lem:optpotentials} and the properties of $c$-conjugate potentials,
  the sum $F'(u^*)+\eps\partial_uh(u^*,v^*)+\Phi$ is differentiable $u^*\leb$-a.e.
  We apply the binomial theorem to the first equation in~\eqref{eq:optpotentials} and use~\eqref{eq:Phigradbound},
  obtaining
  \begin{align*}
    \frac13|\nabla F'(u^*)|^2
    &\le \big|\nabla\big[F'(u^*)+\eps\partial_uh(u^*,v^*)+\Phi\big]\big|^2
      + \eps^2|\nabla\partial_uh(u^*,v^*)|^2 + |\nabla\Phi|^2 \\
    &\le \left|-\nabla\frac{\varphi_u}\tau\right|^2
      + 2\eps^2\big(\partial_\rho\theta_u\big)^2|\nabla F'(u^*)|^2 
      + 2\eps^2\big(\partial_\eta\theta_u\big)^2|\nabla G'(v^*)|^2
      + \frac{2M^2}{\Lambda}\Phi.
  \end{align*}
  By means of~\eqref{eq:monge},
  the bound~\eqref{eqn_def_A} thanks to 2-boundedness and 2-degeneracy of $(F,G,h)$,
  and the swap condition~\eqref{hyp:swap},
  it follows that
  \begin{align*}
    \frac13\intrd u^*|\nabla F'(u^*)|^2\dd x
    &\le \left(\frac{\wass(u^*,\hat u)}\tau\right)^2
    +2\eps^2A^2\intrd u^*|\nabla F'(u^*)|^2\dd x \\
    & \qquad +2\eps^2W\intrd v^*|\nabla G'(v^*)|^2\dd x
    + \frac{2M^2}{\Lambda}\intrd\Phi u^*\dd x.
  \end{align*}
  In combination with the analogous estimate for~$v^*$ in place of~$u^*$,
  and observing that $H_\eps\ge0$,
  we obtain that
  \begin{align*}
    &\left(\frac13-2\eps^2(A^2+W)\right)
    \left(\intrd u^*|\nabla F'(u^*)|^2\dd x + \intrd v^*|\nabla G'(v^*)|^2\dd x\right) \\
    &\qquad \le \left(\frac{\wass(u^*,\hat u)}\tau\right)^2+\left(\frac{\wass(v^*,\hat v)}\tau\right)^2
    + \frac{2M^2}{\Lambda}\nrg_\eps(u^*,v^*).
  \end{align*}
  The result now follows using~\eqref{eq:onestepmono} and the choice of~$\bar \eps$ in~\eqref{eqn_eps_bar}.
\end{proof}
\begin{lemma}
  \label{lem:onestepuFprime_cut}
  With a constant $C$ independent of $(\hat u,\hat v)$, there holds
  \begin{align}
    \label{eq:lowregular}
    \intrd \big( \big|\nabla [F'(u^*)]_{F'(1)}\big|^2 + \big|\nabla [G'(v^*)]_{G'(1)}\big|^2\big) \dd x
    \le C\Big(1+\frac{\entt(\hat u,\hat v)-\entt(u^*,v^*)}\tau\Big),
  \end{align}
  where $[z]_k\coloneqq \min\{k,z\}$ is the cut-off at the value~$k$. 
\end{lemma}
\begin{proof}
  The proof uses the method of flow interchange,
  which estimates the effect of variations
  of a Yosida-regularized \emph{non-convex} functional
  along the gradient flow of an auxiliary \emph{convex} functional.
  The method has been introduced in~\cite{MMS},
  unifying several similar ideas from the literature, see e.g.~\cite{JKO,GST}. 

  For all $s>0$, define perturbations $(U_s,V_s)\in\prbtwotwo$ of $(U_0,V_0)\coloneqq (u^*,v^*)$ as follows:
  \begin{align*}
    U_s \coloneqq  \knl_s\ast u^*,\ V_s \coloneqq  \knl_s\ast v^* \quad \text{with}\quad \knl_s(z) = (4\pi s)^{-d/2}\exp\left(-\frac{|z|^2}{4s}\right).
  \end{align*}
  Since $\knl_s(z)$ is the fundamental solution of the heat equation, it is well-known that $(s,x)\mapsto U_s(x)$ and $(s,x)\mapsto V_s(x)$ are $C^\infty$-smooth on $(0,\infty)\times\R^d$,
  with
  \begin{align}
    \label{eq:heat}
    \partial_sU_s = \Delta U_s, \quad \partial_sV_s = \Delta V_s
  \end{align}
  in the classical sense, and that $U_s\to u^*$ and $V_s\to v^*$ in $L^1(\R^d)$.
  Moreover, $(U_s)_{s>0}$ and $(V_s)_{s>0}$ --- considered as flows on $\prbtwo$ --- satisfy the \evi, see~\eqref{eq:evi}.

  We perform a detailed comparison of the $\nrg_{\eps,\tau}$-scores of $(U_s,V_s)$ and of $(u^*,v^*)$.
  By minimality,
  we know that $\nrg_{\eps,\tau}\big((u^*,v^*)\big|(\hat u,\hat v)\big)\le\nrg_{\eps,\tau}\big((U_r,V_r)\big|(\hat u,\hat v)\big)$;
  consequently, for each $\sigma>0$,
  \begin{align}
    \label{eq:preinterchange}
    \frac{\nrg_\eps(u^*,v^*)-\nrg_\eps(U_\sigma,V_\sigma)}\sigma
    \le \frac1{2\tau}\left(\frac{\wass(U_\sigma,\hat u)^2-\wass(u^*,\hat u)^2}\sigma+\frac{\wass(V_\sigma,\hat v)^2-\wass(v^*,\hat v)^2}\sigma\right).
  \end{align}
  We consider the limes superior as $\sigma\downarrow0$ on both sides.
  On the right-hand side, the \evi~from~\eqref{eq:evi} is applicable and yields
  \begin{align}
    \label{eq:limsup}
    \limsup_{\sigma\downarrow0}\frac{\nrg_\eps(u^*,v^*)-\nrg_\eps(U_\sigma,V_\sigma)}\sigma
    \le \frac{\ent(\hat u)-\ent(u^*)}\tau
    + \frac{\ent(\hat v)-\ent(v^*)}\tau
    = \frac{\entt(\hat u,\hat v)-\entt(u^*,v^*)}\tau.
  \end{align}
  For estimation on the left-hand side, we use the heat equation~\eqref{eq:heat}.
  By regularity and convexity of~$H_\eps$, it easily follows that $s\mapsto \nrg_\eps(U_s,V_s)$ is continuous at $\sigma=0^+$.
  Thanks to smoothness for $\sigma>0$, we can now write
  \begin{align*}
    \frac{\nrg_\eps(u^*,v^*)-\nrg_\eps(U_\sigma,V_\sigma)}\sigma = \frac1\sigma\int_0^\sigma\partial_s\nrg_\eps(U_s,V_s)\dd s
  \end{align*}
  by means of the fundamental theorem of calculus.
  Next, using also smoothness in~$x$ and the convexity estimate~\eqref{eq:Hqual}, we obtain
  \begin{align*}
    -\partial_s\nrg_\eps(U_s,V_s)
    &= -\intrd \big([\partial_uH_\eps(U_s,V_s)+\Phi]\Delta U_s + [\partial_vH_\eps(U_s,V_s) +\Psi]\Delta V_s\big)\dd x \\
    &= \intrd \big( \nabla\partial_uH_\eps(U_s,V_s)\cdot\nabla U_s + \nabla\partial_vH_\eps(U_s,V_s)\cdot\nabla V_s
      - \Delta\Phi\, U_s - \Delta\Psi\,V_s\big)\dd x \\
    &\ge \intrd \sum_{j=1}^d
      \begin{pmatrix}
        \partial_{x_j}U_s \\ \partial_{x_j}V_s 
      \end{pmatrix}
    \cdot \dff^2 H_\eps(U_s,V_s)\cdot
    \begin{pmatrix}
      \partial_{x_j}U_s \\ \partial_{x_j}V_s 
    \end{pmatrix}
    - dM\intrd (U_s + V_s)\dd x     \\
    &\ge \frac12\sum_{j=1}^d \intrd
      \begin{pmatrix}
        \partial_{x_j}U_s \\ \partial_{x_j}V_s 
      \end{pmatrix}
    \cdot
    \begin{pmatrix}
      F''(U_s) & 0 \\ 0& G''(V_s)
    \end{pmatrix}
                         \cdot
                         \begin{pmatrix}
                           \partial_{x_j}U_s \\ \partial_{x_j}V_s 
                         \end{pmatrix}
    -2dM \\
    &\ge \frac{1}{2B} \intrd \big[\big|\nabla [F'(U_s)]_{F'(1)}\big|^2+\big|\nabla[G'(V_s)]_{G'(1)}\big|^2\big]\dd x
      -2dM,
  \end{align*}
  where $B=\max_{0\le r\le1}F''(r)$,
  since by monotonicity of~$F'$, we have a.e.
  \begin{align*}
    \big|\nabla [F'(U_s)]_{F'(1)}\big|^2
    &=
    \begin{cases}
      F''(U_s)^2|\nabla U_s|^2 & \text{if $0\le U_s< 1$}, \\
      0 & \text{if $U_s\ge1$}
    \end{cases}
    \\
    &\le B F''(U_s)|\nabla U_s|^2.
  \end{align*}
  In combination with~\eqref{eq:limsup},
  we have
  \begin{align}
    \label{eq:soon}
    \limsup_{\sigma\downarrow0}\frac{1}{2B\sigma}\int_0^\sigma\intrd \big[\big|\nabla [F'(U_s)]_{F'(1)}\big|^2+\big|\nabla[G'(V_s)]_{G'(1)}\big|^2\big]\dd x \dd \sigma
    \le \frac{\entt(\hat u,\hat v)-\entt(u^*,v^*)}\tau + 2dM.
  \end{align}
  In addition, observe that, thanks to the at most linear growth of~$F'$ near zero,
  $[F'(U_s)]_{F'(1)}$ is uniformly controlled in $L^2(\R^d)$ by the $L^1$-norm of~$U_s$, which is one.
  With~\eqref{eq:soon} it now follows that $[F'(U_\sigma)]_{F'(1)}$ is uniformly bounded in $H^1(\R^d)$
  at least along some sequence with $\sigma\downarrow0$.
  By Rellich's lemma, one may assume strong convergence of that sequence in $L^2(\R^d)$,
  and thus identify the limit as $[F'(u^*)]_{F'(1)}$, since $U_s\to u^*$ strongly in $L^1(\R^d)$ for $s\downarrow0$.
  By Alaoglu's theorem, $F'(U_\sigma)$ converges weakly in $H^1(\R^d)$ to $F'(u^*)$,
  and by lower semi-continuity of norms, we finally obtain
  from~\eqref{eq:soon}
  \begin{equation*}
    \frac{1}{2B}\intrd \big[\big|\nabla [F'(u^*)]_{F'(1)}\big|^2+\big|\nabla[G'(v^*)]_{G'(1)}\big|^2\big]\dd x
    \le \frac{\entt(\hat u,\hat v)-\entt(u^*,v^*)}\tau + 2dM, 
  \end{equation*}
  which immediately yields the claim~\eqref{eq:lowregular}.
\end{proof}
\begin{proof}[Proof of Proposition~\ref{prp:regular} under the additional hypothesis~\eqref{eq:addhypo}]
  Combine~\eqref{eq:highregular} and~\eqref{eq:lowregular}, using that
  \[ |\nabla F'(u^*)|^2 \le |\nabla[F'(u^*)]_{F'(1)}|^2 + u^*|\nabla F'(u^*)|^2  \]
  because $\nabla F'(u^*)=\nabla[F'(u^*)]_{F'(1)}$ if $u^*<1$, and $|\nabla F'(u^*)|^2\le u^*|\nabla F'(u^*)|^2$ if $u^*\ge1$.
\end{proof}

\subsection{Definition of auxiliary functionals}
In this section and the next, we lay the the basis for the proof of Proposition~\ref{prp:core} in Section~\ref{sct:dissipation}.

We start with an alternative representation of the functionals~$\lyp_1$ and~$\lyp_2$ defined in~\eqref{eq:lypuv}.
Recalling~\eqref{eq:abbrev}, we have
\begin{equation}
  \label{eq:lypuvalt}
  \begin{split}
    \lyp_1(u) &= \intrd\big[F(u) + (\Phi+\eps\bar\Theta_u)u\big]\dd x - \intrd\big[F(\bar u_\eps) + (\Phi+\eps \bar\Theta_u)\bar u_\eps\big]\dd x, \\
    \lyp_2(v) &= \intrd\big[G(v) + (\Psi+\eps\bar\Theta_v)v\big]\dd x - \intrd\big[G(\bar v_\eps) + (\Psi+\eps\bar\Theta_v)\bar v_\eps\big]\dd x.
  \end{split}
\end{equation}
Notice that in both lines, the second integral is simply a normalization depending on~$\eps$ but not on~$u$ or~$v$.
The equivalence of the first formula above to the definition of~$\lyp_1$ in~\eqref{eq:lypuv} is obtained
by using --- in that order --- the fact that $\bar u_\eps=0$ on $\{\Phi\ge U_\eps\}$,
then the definition of~$d_F$,
next the first Euler--Lagrange equation from~\eqref{eq:EL}
in combination with the identity $\Phi-U_\eps = (\Phi-U_\eps)_+-(U_\eps-\Phi)_+$
--- which yields $F'(\bar u_\eps) +\eps\bar\Theta_u=(\Phi-U_\eps)_+-(\Phi-U_\eps)$ ---,
and finally equality of mass of~$u$ and~$\bar u_\eps$. In this way, we find
\begin{align*}
  \lyp_1(u)
  &= \intrd \big[ d_F(u|\bar u_\eps) + (u-\bar u_\eps)(\Phi-U_\eps)_+\big]\dd x \\
  &= \intrd \big[ F(u) - F(\bar u_\eps) + (u-\bar u_\eps)\big( (\Phi-U_\eps)_+-F'(\bar u_\eps)\big)\big] \dd x \\
  &= \intrd \big[ F(u) - F(\bar u_\eps) + (u-\bar u_\eps)\big( \Phi-U_\eps+\eps\bar\Theta_u\big)\big] \dd x \\
  &= \intrd \big[ F(u) - F(\bar u_\eps) + (u-\bar u_\eps)\big(\Phi+\eps\bar\Theta_u\big)\big]\dd x.
\end{align*}
The second formula in~\eqref{eq:lypuvalt} is justified analogously. Note that in view of~\eqref{eq:lypuvalt},
the relation between~$\nrg_\eps$ and~$\lyp$ is simply
\begin{align}
  \label{eq:12}
  \nrg_\eps(u,v) - \nrg_\eps(\bar u_\eps,\bar v_\eps)
  = \lyp(u,v) + \eps \intrd\big[h(u,v) - h(\bar u_\eps,\bar v_\eps)
  - \big(u-\bar u_\eps)\bar\Theta_u-(v-\bar v_\eps)\bar\Theta_v\big]\dd x.
\end{align}
For later reference, observe that thanks to
the estimates on~$\bar u_\eps$ and~$\bar v_\eps$ from Section~\ref{sct:stationary},
and the bounds on $|\nabla\bar\Theta_u|$ and $|\nabla\bar\Theta_v|$ from Corollary~\ref{cor:semiconvex} in particular,
there is a constant~$C$ independent of $\eps \in [0,\eps^*/2]$ such that
\begin{align}
  \label{eq:PhibyLyp}
  \intrd \Phi u \dd x \le \lyp_1(u) + C  , \quad \intrd \Psi v \dd x \le \lyp_2(v) + C,
\end{align}
and
\begin{equation}
  \label{eqn:comp_L_E}
  \lyp(u,v) \leq 2 \nrg_\eps(u,v) + C, \quad \nrg_\eps(u,v) \leq 2 \lyp(u,v) + C.
\end{equation}
Next, we introduce dissipation functionals that accompany $\lyp_1$ and $\lyp_2$:
\begin{align}
  \label{eq:dssuv}
  \dss_1(u) = \intrd u\big|\nabla\big[F'(u)+\Phi+\eps\bar\Theta_u\big]\big|^2\dd x,
  \quad
  \dss_2(v) = \intrd v\big|\nabla\big[G'(v)+\Psi+\eps\bar\Theta_v\big]\big|^2\dd x.
\end{align}
In the language of subdifferential calculus in the $L^2$-Wasserstein metric, see e.g.~\cite[Chapter 10]{AGS},
and in view of the representation~\eqref{eq:lypuvalt} above,
one can characterize the functionals above as $\dss_1=|\partial\lyp_1|^2$ and $\dss_2=|\partial\lyp_2|^2$.
Our method of proof does not require the full machinery of metric subdifferentials,
but only the following consequence.
\begin{lemma}
  \label{lem:lypconvex}
  There is a constant~$K_0$ such that for all $\eps\in[0,\eps^*]$ with $K_0\eps<\Lambda$,
  the functionals~$\lyp_1$ and~$\lyp_2$ are uniformly displacement convex of modulus $\Lambda-K_0\eps$.
  In particular, for all $u,v \in \prbtwo$ with $\dss_1(u)<\infty$ and $\dss_2(v)<\infty$, there hold
  \begin{equation}
    \label{eq:mildconvex}
    2(\Lambda-K_0\eps) \lyp_1(u) \le \dss_1(u), \quad
    2(\Lambda-K_0\eps) \lyp_2(v) \le \dss_2(v).
  \end{equation}
\end{lemma}
\begin{proof}
  By Corollary~\ref{cor:semiconvex} (see~\eqref{eq:semiconvex}),
  the function $\bar\Theta_u=\partial_uh(\bar u_\eps,\bar v_\eps)$ 
  is $\eps$-uniformly semi-convex with some modulus~$-K_0$.
  Recalling the $\Lambda$-uniform convexity of~$\Phi$,
  we conclude that the sum $\Phi+\eps \bar\Theta_u$ is uniformly convex of modulus $\Lambda-K_0\eps$ as long as $K_0\eps<\Lambda$.
  By assumption,~$F$ satisfies McCann's condition~\eqref{eq:mccann}.
  The result now follows from the general theory of displacement convexity, see Lemma~\ref{lem:FuncIneq}.   
  The argument for~$\lyp_2$ is completely analogous.
\end{proof}
Notice that by the second condition in~\eqref{eqn_eps_bar} for the choice of $\bar \eps >0$,
we are able to bound for every $\eps \in [0,\bar \eps]$ the dissipation functionals~$\dss_1$ and~$\dss_2$ from below
by some positive multiple of the relative entropy functionals~$\lyp_1$ and~$\lyp_2$, respectively.

\subsection{An estimate by Bregman distances}
The sole purpose of this section is to show Lemma~\ref{lem:bregman} below,
which becomes essential for the estimate of ``garbage terms'' in the proof of Proposition~\ref{prp:core}.
The (technical) proof of the lemma heavily uses $2$-degeneracy and $2$-boundedness of $(F,G,h)$.
\begin{lemma}
  \label{lem:bregman}
  There is a constant~$\kappa$ independent of $\eps\in[0,\eps^*]$ 
  such that the estimate
  \begin{align}
    \label{eq:omegaest}
    \intrd (u+v)\big[\omega\big(F'(u),G'(v)\big)-\omega\big(F'(\bar u_\eps),G'(\bar v_\eps)\big)\big]^2\dd x
    \le \kappa\lyp(u,v)
  \end{align}
  holds for all $(u,v)\in\prbtwotwo$ with $\lyp(u,v) < \infty$ and for $\omega\colon\Rnn^2\to\R$ being any of the following four functions:
  $\partial_\rho\theta_u$, $\partial_\eta\theta_u$, $\partial_\rho\theta_v$, or $\partial_\eta\theta_v$.
\end{lemma}
\begin{proof}
  We shall derive~\eqref{eq:omegaest} as a consequence of a pointwise estimate on the integrand,
  namely
  \begin{align}
    \label{eq:preomegaest}
    (u+v)\big[\omega\big(F'(u),G'(v)\big)-\omega\big(F'(\bar u_\eps),G'(\bar v_\eps)\big)\big]^2
    \le \kappa\big(d_F(u|\bar u_\eps) + d_G(v|\bar v_\eps)\big),
  \end{align}
  where~$d_F$ and~$d_G$ were introduced in~\eqref{eq:bregman}.
  In view of the definition of~$\lyp$ in~\eqref{eq:lypuv}--\eqref{eq:deflyp},
  an integration in~$x$ yields the claim~\eqref{eq:omegaest}.

  We now prove~\eqref{eq:preomegaest}, only with~$u$ instead of $(u+v)$ (as the case with~$v$ instead of $(u+v)$ is analogous). Let~$\bar U$ be such that $\bar u_\eps,\bar v_\eps\le\bar U$ for all $\eps\in[0,\eps^*]$. 
  We distinguish three cases.

  \emph{Case 1:} $u>3\bar U$.
  By $2$-boundedness and $2$-degeneracy of $(F,G,h)$, we have 
  \begin{align*}
    u\big[\omega\big(F'(u),G'(v)\big)-\omega\big(F'(\bar u_\eps),G'(\bar v_\eps)\big)\big]^2 \le 4A^2u,
  \end{align*}
  see Remark~\ref{remark_consequence_hypotheses}~(3). The multiple of~$u$ on the right-hand side can be estimated by a multiple of $d_F(u|\bar u)$ as follows: thanks to convexity of~$F$,
  \begin{align}
    d_F(u|\bar u_\eps)
    & \ge F(2\bar U) + (u-2\bar U)F'(2\bar U) - \big[F(\bar u_\eps)+(u-\bar u_\eps)F'(\bar u_\eps)\big] \nonumber \\
    & \ge (u-2\bar U)F'(2\bar U) + (2\bar U-\bar u_\eps)F'(\bar u_\eps) - (u-\bar u_\eps)F'(\bar u_\eps) \nonumber \\
    & = (u-2\bar U)\big[F'(2\bar U) - F'(\bar u_\eps)\big]
    \ge (u-2\bar U)\big[F'(2\bar U)-F'(\bar U)\big]. \label{eq:convexity_large_u}
  \end{align}
  By strict convexity, $F'(2\bar U) - F'(\bar U)>0$,
  and with $u>3\bar U$, we clearly have $4A^2u\le \kappa d_F(u|\bar u_\eps)$
  for an appropriate constant~$\kappa$ independent of~$u$ and~$\eps$.

  \emph{Case 2:} $0\le u\le3\bar U$ and $v>3\bar U$.
  Using again the global bound~$A$ on~$\omega$,
  we obtain
  \begin{align*}
    u\big[\omega\big(F'(u),G'(v)\big)-\omega\big(F'(\bar u_\eps),G'(\bar v_\eps)\big)\big]^2 \le 4 A^2\bar U,
  \end{align*}
  i.e., the left-hand side is bounded by an expression that is independent of~$u$,~$v$ and~$\eps$.
  Clearly, this expression is estimated by $\kappa d_G(v|\bar v_\eps)$
  with an appropriate~$\kappa$ independent of~$v$;
  this follows in analogy to the estimate in the first case above.

  \emph{Case 3:} $0\le u,v\le3\bar U$.
  As a first step, we show that, for an appropriate constant~$L$ independent of~$\eps$, we have
  \begin{align}
    \label{eq:bregmanbelow}
    [F'(u)-F'(\bar u_\eps)]^2\le L d_F(u|\bar u_\eps)
  \end{align}
  and correspondingly 
  \begin{equation*}
   [G'(v)-G'(\bar v_\eps)]^2\le L d_G(v|\bar v_\eps).
  \end{equation*}  
  Considering both sides of~\eqref{eq:bregmanbelow} as functions in the variable~$u$, we obviously have equality in the case $u = \bar u_\eps$. Using the definition of~$d_F$, we find 
  \begin{equation*}
   \frac{\dn}{\dd u} [F'(u)-F'(\bar u_\eps)]^2 = 2 [F'(u)-F'(\bar u_\eps)] F''(u) = 2 F''(u) \frac{\dn}{\dd u} d_F(u|\bar u_\eps)
  \end{equation*}
  for all $u > 0$. Notice that this expression, by convexity of~$F$, is negative for $u < \bar u_\eps$ and positive for $u > \bar u_\eps$. With the help of hypothesis~\eqref{hyp:powerF} and the convexity of~$F$, there exists a positive constant~$C$ such that
  \begin{equation*}
   F''(r) \le C \quad \text{for all } r \in [0,3\bar U].
  \end{equation*}
  Hence, with the choice $L \coloneqq  2 C$, we conclude that the left-hand side of~\eqref{eq:bregmanbelow} decreases faster on $(0,\bar u_\eps)$ and increases slower on $(\bar u_\eps,3 \bar U)$ than the right-hand side. This suffices to have the validity of inequality~\eqref{eq:bregmanbelow} for all $u \in [0,3\bar U]$.
  
  Next, define $\rho_s\coloneqq sF'(u)+(1-s)F'(\bar u_\eps)$ and $\eta_s\coloneqq sG'(v)+(1-s)G'(\bar v_\eps)$ for $s\in[0,1]$.
  Then
  \begin{align*}
    & \big[\omega\big(F'(u),G'(v)\big)-\omega\big(F'(\bar u_\eps),G'(\bar v_\eps)\big)\big]^2 \\
    &\le \int_0^1\big[
    \big(F'(u)-F'(\bar u_\eps)\big)\partial_\rho\omega(\rho_s,\eta_s)
    + \big(G'(v)-G'(\bar v_\eps)\big)\partial_\eta\omega(\rho_s,\eta_s)
    \big]^2\dd s \\
    &\le 2\bigg(\sup_{0\le\mu,\nu\le3 \bar U}\big|\dff\omega\big|\bigg)^2
    \big[\big(F'(u)-F'(\bar u_\eps)\big)^2+\big(G'(v)-G'(\bar v_\eps)\big)^2\big].
  \end{align*}
  The supremum above is a finite quantity~$B$,
  thanks to $2$-boundedness of $(F,G,h)$.
  Therefore, with~$L$ from~\eqref{eq:bregmanbelow}, we find
  \begin{align*}
    u\big[\omega\big(F'(u),G'(v)\big)-\omega\big(F'(\bar u_\eps),G'(\bar v_\eps)\big)\big]^2
    \le 6\bar U B L \big(d_F(u|\bar u_\eps)+d_G(v|\bar v_\eps)\big),
  \end{align*}
  proving the pointwise estimate~\eqref{eq:preomegaest} also in the final case.
\end{proof}

\subsection{Proof of the core inequality}
\label{sct:dissipation}
Finally, we prove Proposition~\ref{prp:core}.
Again, the additional hypothesis~\eqref{eq:addhypo} enters only indirectly via Lemma~\ref{lem:optpotentials}.
\begin{proof}[Proof of Proposition~\ref{prp:core} under the additional hypothesis~\eqref{eq:addhypo}]
  Let $P_F(r)=rF'(r)-F(r)$ for $r \geq 0$.
  Recall that $(\varphi_u,\psi_u)$ is the optimal pair of $c$-conjugate potentials for the transport from~$u^*$ to~$\hat u$.
  Since~$\lyp_1$ is displacement convex, the following ``above tangent formula'' holds,
  see e.g.~\cite[Proposition 5.29 \& Theorem 5.30]{Villani},
  \begin{align*}
    \lyp_1(\hat u) - \lyp_1(u^*)
    \ge \intrd P_F(u^*)\Delta^\text{ac}\varphi_u\dd x-\intrd u^*\nabla[\Phi+\eps\bar\Theta_u]\cdot\nabla\varphi_u \dd x, 
  \end{align*}
  where $\Delta^\text{ac}\varphi_u$ is the absolutely continuous part of the signed measure
  defined by the distributional Laplacian $\Delta\varphi_u$.
  Thanks to the regularity of~$u^*$, we may re-write the first integral on the right-hand side using integration by parts.
  Indeed, observe that
  \[ \nabla P_F(u^*)\cdot\nabla\varphi_u=u^*\nabla F'(u^*)\cdot\nabla\varphi_u\in L^1(\R^d) \]
  since $\nabla F'(u^*)\in L^2(\R^d;u^*\leb)$ by Lemma~\ref{lem:onestepuFprime}
  and $\nabla\varphi_u\in L^2(\R^d;u^*\leb)$ in view of~\eqref{eq:monge}.
  Now, since $P_F\ge0$ by convexity of $F$,
  and since $\Delta^{ac}\varphi_u\ge \Delta\varphi_u$ because~$\varphi_u$ is semi-concave,
  we have
  \begin{align}
    \label{eq:halfcore1}
    \lyp_1(\hat u) - \lyp_1(u^*)\ge - \intrd u^*\nabla \big[F'(u^*)+\Phi+\eps\bar\Theta_u\big]\cdot\nabla\varphi_u\dd x
    = \tau Z_1(u^*,v^*),
 \end{align}
  where~$Z_1$ can be made more explicit by substitution of the potential ~$\varphi_u$ from~\eqref{eq:optpotentials}:
  \begin{align*}
    Z_1(u,v) \coloneqq  \intrd u\,\nabla\big[F'(u)+\Phi+\eps\bar\Theta_u\big]\cdot\nabla\big[F'(u) + \Phi+\eps \partial_uh(u,v) \big]\dd x.   
  \end{align*}
  We estimate $Z_1(u,v)$ using a combination of the previously shown results.
  For brevity, we use --- only in the calculations below ---
  in addition to~$\bar\Theta_u$ and~$\bar\Theta_v$ introduced in~\eqref{eq:abbrev}
  the notations
  \begin{align*}
    \Theta_u &\coloneqq  \theta_u\big(F'(u),G'(v)\big) = \partial_uh(u,v), \\
    \Theta_{u,\rho} &\coloneqq  \partial_\rho\theta_u\big(F'(u),G'(v)\big) = \frac{\partial_{uu}h(u,v)}{F''(u)},
  \end{align*}
  and so on. First, the Cauchy--Schwarz inequality yields 
  \begin{align*}
    Z_1(u,v) &=
        \intrd u\left|\nabla\big[F'(u)+\Phi+\eps\bar\Theta_u\big]\right|^2\dd x
        + \eps\intrd u \,\nabla\big[F'(u)+\Phi+\eps\bar\Theta_u\big]\cdot\nabla\big[\Theta_u-\bar\Theta_u\big]\dd x \\
      &\ge \left(1-\frac{\eps}2\right)\dss_1(u) - \frac{\eps}2 \intrd u\left|\nabla\big[\Theta_u-\bar\Theta_u\big]\right|^2\dd x.
  \end{align*}
  Inside the last integral, we have
  \begin{align*}
    \nabla\big[\Theta_u-\bar\Theta_u\big]
    &= \Theta_{u,\rho}\nabla F'(u) + \Theta_{u,\eta}\nabla G'(v)
      - \bar\Theta_{u,\rho}\nabla F'(\bar u_\eps) - \bar\Theta_{u,\eta}\nabla G'(\bar v_\eps) \\
    &= \Theta_{u,\rho}\nabla\big[F'(u)+\Phi+\eps\bar\Theta_u\big]  + \Theta_{u,\eta}\nabla\big[G'(v)+\Psi+\eps\bar\Theta_v\big]\\
    & \quad - \Theta_{u,\rho}\nabla\big[F'(\bar u_\eps)+\Phi+\eps\bar\Theta_u\big] - \Theta_{u,\eta}\nabla\big[G'(\bar v_\eps)+\Psi+\eps\bar\Theta_v\big] \\
    & \quad  +\big(\Theta_{u,\rho}-\bar\Theta_{u,\rho}\big)\nabla F'(\bar u_\eps) 
      +\big(\Theta_{u,\eta}-\bar\Theta_{u,\eta}\big)\nabla G'(\bar v_\eps).
  \end{align*}
  The third and the fourth term above can be simplified using 
  that by combination of the Euler--Lagrange system~\eqref{eq:EL}
  with the identity $\Phi-U_\eps=(\Phi-U_\eps)_+-(U_\eps-\Phi)_+$,
  one has
  \begin{align*}
    \nabla\big[F'(\bar u_\eps)+\Phi+\eps\bar\Theta_u\big] = -\nabla(\Phi-U_\eps)_+, \quad
    \nabla\big[G'(\bar v_\eps)+\Psi+\eps\bar\Theta_v\big] = -\nabla(\Psi-V_\eps)_+.    
  \end{align*}
  This yields
  \begin{align}
    \nonumber
    &\intrd u\big|\nabla\big[\Theta_u-\bar\Theta_u\big]\big|^2 \dd x \\
    \label{eq:1oo4}
    & \le 6\intrd u\Theta_{u,\rho}^2
      \big|\nabla\big[F'(u)+\Phi+\eps\bar \Theta_u\big]\big|^2\dd x
      + 6\intrd u\Theta_{u,\eta}^2
      \big|\nabla\big[G'(v)+\Psi+\eps\bar \Theta_v\big]\big|^2\dd x
    \\
    \label{eq:2oo4}
    & \quad
      + 6\intrd u\Theta_{u,\rho}^2\big|\nabla(\Phi-U_\eps)_+\big|^2\dd x
      + 6\intrd u\Theta_{u,\eta}^2\big|\nabla(\Psi-V_\eps)_+\big|^2\dd x  
    \\
    \label{eq:3oo4}
    &\quad
      +6 \intrd u \big(\Theta_{u,\rho}-\bar\Theta_{u,\rho}\big)^2 |\nabla F'(\bar u_\eps)|^2\dd x
      +6 \intrd u \big(\Theta_{u,\eta}-\bar\Theta_{u,\eta}\big)^2 |\nabla G'(\bar v_\eps)|^2\dd x .    
  \end{align}
  For further estimation, we observe that $2$-boundedness and $2$-degeneracy of $(F,G,h)$ imply
  \begin{align}
    \label{eq:Ahelp}
    |\Theta_{u,\rho}|\le A\min\big\{1,F'(u),G'(v)\big\},
    \quad
    |\Theta_{u,\eta}|\le A\min\big\{1,F'(u),G'(v)\big\},    
  \end{align}
  see~\eqref{eqn_def_A}. The first integral in~\eqref{eq:1oo4} is now easily estimated using that thanks to~\eqref{eq:Ahelp},
  \begin{align*}
    \intrd u\Theta_{u,\rho}^2
    \big|\nabla\big[F'(u)+\Phi+\eps\bar\Theta_u\big]\big|^2\dd x
    \le A^2\dss_1(u).
  \end{align*}
  For estimation of the second integral in~\eqref{eq:1oo4},
  we use instead that $(F,G,h)$ satisfies the swap condition~\eqref{hyp:swap}:
  thus $u\Theta_{u,\eta}^2\le W^2v$
  and consequently
  \begin{align*}
    \intrd u\Theta_{v,\eta}^2
    \big|\nabla\big[G'(v)+\Psi+\eps\bar\Theta_v\big]\big|^2\dd x
    \le W^2\dss_2(v).
  \end{align*}
  For estimation of the first integral in~\eqref{eq:2oo4}, two ingredients are needed.
  First, recall that $|\nabla\Phi|^2\le \frac{2M^2}{\Lambda}\Phi$ by~\eqref{eq:Phigradbound},
  and conclude that on $\{\Phi>U_\eps\}$:
  \begin{align*}
    \big|\nabla(\Phi-U_\eps)_+\big|^2
    = |\nabla\Phi|^2
    \le \frac{2M^2}{\Lambda}\Phi
    =    \frac{2M^2}{\Lambda}U_\eps +  \frac{2M^2}{\Lambda}(\Phi-U_\eps)_+\,.
  \end{align*}
  Second, we claim that there is a constant~$B$ such that
  \begin{align*}
    u\Theta_{u,\rho}^2 \le B F(u).    
  \end{align*}
  For $u\ge1$ this is a trivial consequence of the convexity of~$F$.
  For $0<u<1$, we use that in view of hypotheses~\eqref{eq:degeneracy} and~\eqref{hyp:powerF},
  there are constants~$c_0$ and $C_0$ such that $ F'(u) \le C_0 u^{m-1}$ and $F(u) \ge c_0 u^m$ are satisfied. Therefore, employing also~\eqref{eq:Ahelp}, we have
  \begin{align*}
    u\Theta_{u,\rho}^2
    \le u A^2F'(u)
    \le \frac{C_0A^2}{c_0}F(u).
  \end{align*}
  Now we combine these ingredients, recalling again~\eqref{eq:Ahelp} and
  bearing in mind that the integral is actually an integral on $\{\Phi>U_\eps\}$ only,
  where $d_F(u|\bar u_\eps)=F(u)$ thanks to~\eqref{eq:degeneracy}:
  \begin{align*}
    \intrd u\Theta_{u,\rho}^2\big|\nabla(\Phi-U_\eps)_+\big|^2 \dd x
    &\le \frac{2M^2 B U_\eps}{\Lambda}\intrd d_F(u|\bar u_\eps)\dd x + \frac{2M^2A^2}{\Lambda}\intrd u(\Phi-U_\eps)_+\dd x \\
    &\le \frac{2M^2}{\Lambda}\max\big\{BU_\eps,A^2\big\}\lyp_1(u).
  \end{align*}
  The second integral in~\eqref{eq:2oo4} is estimated in a completely analogous manner.

  Finally, the integrals in~\eqref{eq:3oo4} are both estimated by means of Lemma~\ref{lem:bregman}.
  We combine this with the boundedness of $|\nabla F'(\bar u_\eps)|$ and $|\nabla G'(\bar v_\eps)|$, respectively:
  by Proposition~\ref{prp:regularity}, we have that $F'(\bar u_\eps),G'(\bar u_\eps)\in W^{1,\infty}(\R^d)$,
  and that
  \begin{align*}
    |\nabla F'(\bar u_\eps)| \le \tilde{B}, \quad |\nabla G'(\bar v_\eps)|\le \tilde{B}
  \end{align*}
  a.e.~on~$\R^d$, with~$\tilde{B}$ independent of~$\eps$.
  We thus obtain 
  \begin{align*}
    \intrd u\big(\Theta_{u,\rho}-\bar\Theta_{u,\rho}\big)^2\big|\nabla F'(\bar u_\eps)\big|^2\dd x
    + \intrd u\big(\Theta_{u,\eta}-\bar\Theta_{u,\eta}\big)^2\big|\nabla G'(\bar v_\eps)\big|^2\dd x
    \le 2\tilde{B}^2\kappa\lyp(u,v).
  \end{align*}
  To summarize so far, we have shown that, with a suitable constant~$C$,
  \begin{align*}
    Z_1(u,v) \ge \left(1-\frac{\eps}2\big[1+A^2\big]\right)\dss_1(u) - \frac{\eps}2 W^2\dss_2(v) - \frac{\eps}2 C\lyp(u,v).
  \end{align*}
  This finishes our estimate on~$Z_1$.
  The pendant of~\eqref{eq:halfcore1} for~$v$ in place of~$u$ is
  \begin{equation*}
    \lyp_2(\hat v) - \lyp_2(v^*) \ge Z_2(u^*,v^*)
  \end{equation*}
  with
  \begin{equation*}
    Z_2(u,v) \coloneqq  \intrd v\,\nabla\big[G'(u)+\Psi+\eps\bar\Theta_v\big]\cdot\nabla\big[G'(v) + \Psi+\eps \partial_vh(u,v) \big]\dd x.
  \end{equation*}
  Estimating~$Z_2$ in analogy to~$Z_1$ as above leads to
  \begin{align*}
    Z_1(u,v) + Z_2(u,v)
    \ge \left(1-\frac{\eps}2\big[1+A^2+W^2\big]\right)(\dss_1(u)+\dss_2(v))
    - \eps C\lyp(u,v) .
  \end{align*}
  With an application of~\eqref{eq:mildconvex} and an appropriate choice of $K>0$,
  the claim~\eqref{eq:coreest} has been shown.
\end{proof}

\subsection{Removal of the additional hypothesis on the datum}
\label{sct:removehypo}
Let $(\hat u,\hat v)\in\prbtwotwo$ with $\nrg_\eps(\hat u,\hat v)<\infty$ with $\eps \in [0,\bar{\eps}]$ be given,
that does not necessarily satisfy the additional hypothesis~\eqref{eq:addhypo} for any $R>0$,
and let $(u^*,v^*)\in\prbtwotwo$ be the unique minimizer of the Yosida-regularized energy, according to Lemma~\ref{lem:jko}.

Consider sequences of radii $R\to\infty$ and pairs $(\hat u_R,\hat v_R)\in\prbtwotwo$ that satisfy~\eqref{eq:addhypo} for the corresponding $R$.
Moreover, we assume that $(\hat u_R,\hat v_R)$ approximates $(\hat u,\hat v)$ in the following sense:
$\hat u_R$, $F(\hat u_R)$ and $|x|^2\hat u_R$ converge to $\hat u$, $F(\hat u)$ and $|x|^2\hat u$ in $L^1(\R^d)$, respectively,
and likewise for $\hat v_R$.
An immediate consequence is:
\begin{align}
  \label{eq:energylim}
  \nrg_\eps(\hat u_R,\hat v_R)\to\nrg_\eps(\hat u,\hat v),\quad
  \lyp(\hat u_R,\hat v_R) \to \lyp(\hat u,\hat v), \quad
  \entt(\hat u_R,\hat v_R) \to \entt(\hat u,\hat v).
\end{align}
Next, consider the associated sequence of functionals $\auxe_R$ on $\prbtwotwo$ given by
\begin{align*}
  \auxe_R(u,v) = \nrg_{\eps,\tau}\big((u,v)\big|(\hat u_R,\hat v_R)\big).
\end{align*}
By Lemma~\ref{lem:jko}, there is a unique minimizer $(u^*_R,v^*_R)$ for each $\auxe_R$.
As an intermediate step, we show convergence of these minimizers to $(u^*,v^*)$.
\begin{lemma}
  \label{lem:Gamma}
  As $R\to\infty$, the functionals $\auxe_R$ $\Gamma$-converge
  to $\nrg_{\eps,\tau}\big(\cdot\big|(\hat u,\hat v)\big)$ in the narrow topology.

  Moreover, let $\overline{\auxe} \coloneqq \sup_R\nrg_\eps(\hat u_R,\hat v_R)$, which is finite by~\eqref{eq:energylim};
  the sublevel sets $\auxe_R\le\overline{\auxe}$ are non-empty,
  and consist of pairs of densities with $R$-uniformly bounded second moment and $L^2$-norm.
\end{lemma}
\begin{proof}
  We verify the definition of $\Gamma$-convergence:
  first, the ``liminf-property'' is a direct consequence of
  the lower semicontinuity of $\nrg_\eps$ and that of $\wass$ (with respect to both components)
  under narrow convergence.  
  Second, a recovery sequence $(\tilde u_R,\tilde v_R)$ for a given pair $(u,v)$ is the constant one, $(\tilde u_R,\tilde v_R):=(u,v)$.
  Indeed, for that choice,
  \begin{align*}
    \auxe_R(\tilde u_R,\tilde v_R)-\nrg_{\eps,\tau}\big((u,v)\big|(\hat u,\hat v)\big)
    = \frac1{2\tau}\big[\big(\wass(u,\hat u_R)^2-\wass(u,\hat u)^2\big)+\big(\wass(v,\hat v_R)^2-\wass(v,\hat v)^2\big)\big],
  \end{align*}
  which tends to zero for $R \to \infty$
  since $\wass(u,\hat u_R)\to\wass(u,\hat u)$ and $\wass(v,\hat v_R)\to\wass(v,\hat v)$ follows
  by convergence of $\hat u_R$ and $\hat v_R$ in $L^1(\R^d)$, and convergence of their second moments, to the respective limits.  

  Concerning the sublevel set $\auxe_R\le\overline{\auxe}$:
  By definition of $\auxe_R$, it contains $(\hat u_R,\hat v_R)$ and is thus non-empty.
  Moreover, any $(u,v)$ in that sublevel satisfies
  \begin{align*}
    \nrg_\eps(u,v)\le\overline{\auxe}, \quad
    \wass(u,\hat u_R)^2\le 2\tau\overline{\auxe},\quad
    \wass(v,\hat v_R)^2\le 2\tau\overline{\auxe}.
  \end{align*}
  The first inequality provides an $R$-uniform bound on $F(u)$ and $G(v)$ in $L^1(\R^d)$,
  which in view of the at least quadratic growth of $F$ and $G$ implies a bound of $u$ and $v$ in $L^2(\R^d)$.
  The second estimate provides an $R$-uniform bound on the second moment of~$u$ and~$v$:
  let $T$ be an optimal map for the transport of $u$ to $\hat u_R$, then
  \begin{equation*}
    \frac12\intrd |x|^2u(x)\dd x
    \le \intrd |T(x)|^2u(x)\dd x + \intrd|T(x)-x|^2u(x)\dd x
    = \intrd |y|^2\hat u_R(y)\dd y + \wass(u,\hat u_R)^2;
  \end{equation*}
  recall that the second moments of the $\hat u_R$ are $R$-uniformly bounded by construction.
  The argument for $v$ is analogous.
\end{proof}
As a consequence of Lemma~\ref{lem:Gamma}, and by uniqueness of the minimizer $(u^*,v^*)$ for the limiting functional,
we have that $u^*_R\to u^*$ and $v^*_R\to v^*$ narrowly.
\begin{proof}[Proof of Proposition~\ref{prp:core}]
  Since the pair $(\hat u_R,\hat v_R)$ satisfies the additional hypothesis~\eqref{eq:addhypo},
  inequality~\eqref{eq:coreest} is valid for $(u^*_R,v^*_R)$ and $(\hat u_R,\hat v_R)$ in place of $(u^*,v^*)$ and $(\hat u,\hat v)$, i.e.,
  \begin{align*}
    \lyp(\hat u_R,\hat v_R) \ge \big[1+2\tau(\Lambda-K\eps)\big]\lyp(u^*_R,v^*_R).
  \end{align*}
  By~\eqref{eq:energylim}, the left-hand side converges to $\lyp(\hat u,\hat v)$, while we use lower semi-continuity of~$\lyp$ with respect to narrow convergence on the right-hand side.
  This yields~\eqref{eq:coreest}, as desired.
\end{proof}
\begin{lemma}
  \label{lem:filippo}
  $F'(u^*_R)$ and $G'(v^*_R)$ converge to their respective limits $F'(u^*)$ and $G'(v^*)$, weakly in $H^1(\R^d)$, and strongly in $L^2(\R^d)$.
\end{lemma}
\begin{proof}
  We show that $F'(u^*_R)$ and $G'(v^*_R)$ are $R$-uniformly bounded in $H^1(\R^d)$.
  The argument for $F'(u^*_R)$ is the following:
  first, observe that $F'(s)\le C(s+F(s))$ with some constant~$C$,
  which is true for small and for large values of $s\ge0$, respectively, because of~\eqref{eq:degeneracy} and~\eqref{eq:tripling}.
  This implies an $R$-uniform $L^1$-bound on $F'(u^*_R)$ since
  \begin{align*}
    \intrd F'(u^*_R) \dd x \le C\left(\intrd u^*_R\dd x + \intrd F(u^*_R)\dd x\right) \le C\big(1+2\nrg_\eps(u^*_R,v^*_R)\big).
  \end{align*}
  Next, estimate~\eqref{eq:allregular} holds with $(u^*_R,v^*_R)$ and $(\hat u_R,\hat v_R)$ in place of $(u^*,v^*)$ and $(\hat u,\hat v)$, respectively,
  since $(\hat u_R,\hat v_R)$ satisfies hypothesis~\eqref{eq:addhypo}.
  By~\eqref{eq:energylim}, the terms $\nrg_\eps(\hat u_R,\hat v_R)$ and $\entt(\hat u_R,\hat v_R)$ are $R$-uniformly bounded from above, and by Lemma~\ref{lem:gero}, $\entt(u^*_R,v^*_R)$ is $R$-uniformly bounded from below.
  Together, this implies an $R$-uniform bound on $\nabla F'(u^*_R)$ in $L^2(\R^d)$.
  By interpolation with the bound in $L^1(\R^d)$ above, the bound in $H^1(\R^d)$ follows.

  The claim now follows by Alaoglu's theorem, and by Rellich's theorem,
  bearing in mind that the sequences $F'(u^*_R)$ and $G'(v^*_R)$ have narrow limits $F'(u^*)$ and $G'(v^*)$, respectively. 
\end{proof}
\begin{proof}[Proof of Proposition~\ref{prp:regular}]
  From estimate~\eqref{eq:allregular}, applied to $(\hat u_R,\hat v_R)$, we have that
  \begin{multline*}
    \frac C\tau\big[\nrg_\eps(u^*_R,v^*_R)+\entt(u^*_R,v^*_R)\big] + \intrd \big[|\nabla F'(u^*_R)|^2 + |\nabla G'(v^*_R)|^2 \big]\dd x
    \\
    \le \frac C\tau\big[\tau+(1+\tau)\nrg_\eps(\hat u_R,\hat v_R) + \entt(\hat u_R,\hat v_R)\big].
  \end{multline*}
  By means of~\eqref{eq:energylim}, we can pass to the limit on the right-hand side.
  By means of lower semi-continuity of $\nrg_\eps$ and $\entt$ with respect to narrow convergence,
  and of the $L^2$-norm with respect to weak convergence in $L^2(\R^d)$,
  we can pass via the previous Lemma~\ref{lem:filippo} to the limit also on the left-hand side.
  This gives~\eqref{eq:allregular} with datum $(\hat u,\hat v)$.
\end{proof}
\begin{proof}[Proof of Proposition~\ref{prp:dweak}]
  Fix $\zeta\in C^\infty_c(\R^d)$.
  Since~\eqref{eq:dweak} holds under the hypothesis~\eqref{eq:addhypo}, we have
  \begin{align*}
    \intrd \frac{u^*_R-\hat u_R}{\tau}\zeta\dd x = \intrd u_R^*\nabla\big[F'(u_R^*)+\eps\partial_uh(u_R^*,v_R^*)+\Phi\big]\cdot\nabla\zeta\dd x + R_u.
  \end{align*}
  We can easily pass to the limit $R\to\infty$ on the left-hand side by narrow and $L^1$-convergence of~$u_R^*$ and~$\hat{u}_R$, respectively.
  For the integral on the right-hand side, observe that $u_R^*\nabla\zeta\to u^*\nabla\zeta$ in $L^2(\R^d)$,
  because of Lemma~\ref{lem:filippo}, and since $F'$ has at least linear growth.
  Moreover, $\nabla F'(u_R^*)\rightharpoonup \nabla F'(u^*)$ follows by Lemma~\ref{lem:filippo} as well.
  To conclude that also $\nabla\partial_uh(u_R^*,v_R^*)\rightharpoonup\nabla\partial_uh(u^*,v^*)$ in $L^2(\R^d)$,
  observe that
  \begin{align*}
    \nabla\partial_uh(u_R^*,v_R^*)
    = \partial_\rho\theta_u\big(F'(u_R^*),G'(v_R^*)\big)\,\nabla F'(u_R^*)
    + \partial_\eta\theta_u\big(F'(u_R^*),G'(v_R^*)\big)\,\nabla G'(v_R^*).
  \end{align*}
  For both products on the right-hand side, weak convergence in $L^2(\R^d)$ is easily concluded
  from the weak convergence of $\nabla F'(u_R^*)$ and of $\nabla G'(v_R^*)$,
  and from the convergence in measure of
  the bounded functions $\partial_\rho\theta_u\big(F'(u_R^*),G'(v_R^*)\big)$ and $\partial_\eta\theta_u\big(F'(u_R^*),G'(v_R^*)\big)$,
  again thanks to Lemma~\ref{lem:filippo}, and to the 2-boundedness of $(F,G,h)$. 

  In a completely analogous way, we can pass to the limit $R\to\infty$ in the $v$-equation in~\eqref{eq:addhypo}.
  The bound on $|R_u|+|R_v|$ is preserved thanks to~\eqref{eq:energylim} and lower semicontinuity of~$\nrg_\eps$ and $\entt$ with respect to narrow convergence. 
\end{proof}


\section{Existence of weak solutions}
\label{sct:existence}

The Yosida-regularized energy functional~$\nrg_{\eps,\tau}$ is now used to obtain
a time-discrete approximation $(u_\tau^n,v_\tau^n)_{n\in\N_0}$ of the solution to~\eqref{eq:eq}
for given initial data $u(0)=u_0$, $v(0)=v_0$ with finite energy $\nrg_\eps(u_0,v_0)<\infty$
by means of the \emph{minimizing movement scheme}.
Inductively, define $(u_\tau^0,v_\tau^0)\coloneqq (u_0,v_0)$,
and for each $n \in \N$ let $(u_\tau^n,v_\tau^n)$ be the minimizer
--- which exists and is unique by Lemma~\ref{lem:jko} ---
of the functional
\begin{align*}
  \prbtwotwo\ni(u,v)\mapsto \nrg_{\eps,\tau}\big((u,v)\big|(u_\tau^{n-1},v_\tau^{n-1})\big).
\end{align*}
Further, define the piecewise constant ``interpolations'' $\tilde u_\tau,\tilde v_\tau \colon [0,\infty)\to\prbtwo$ (depending of course on~$\eps$)
in the usual way:
\begin{align*}
  \tilde u_\tau(t) = u_\tau^n, \ \tilde v_\tau(t) = v_\tau^n\quad \text{for $(n-1)\tau<t\le n\tau$}.
\end{align*}
The result of this section is the following convergence.

\begin{proposition}
  \label{prp:existence}
  For every $\eps \in [0,\eps^*]$, the interpolations~$\tilde u_\tau$,~$\tilde v_\tau$ converge, for a suitable sequence $\tau\downarrow0$, to H\"older-continuous limit curves $u_*,v_* \colon [0,\infty)\to\prbtwo$,
  weakly in $L^1(\R^d)$ at every $t\ge0$:
  Moreover, $F'(\tilde u_\tau),G'(\tilde v_\tau)$ converge to the respective limits $F'(u_*),G'(v_*)$,
  weakly in $L^2(0,T;H^1(\R^d))$ and strongly in $L^2((0,T)\times\R^d)$, for any $T>0$.
  Furthermore, the limits are weak solutions to~\eqref{eq:eq} in the following sense:
  \begin{equation}
    \label{eq:weaksol}
    \begin{split}
      0 &= \int_0^\infty \intrd\Big(u_*\partial_t\xi - u_*\nabla\big[F'(u_*)+\Phi+\eps \partial_uh(u_*,v_*)\big]\cdot\nabla\xi \Big)\dd x\dd t , \\
      0 &= \int_0^\infty \intrd\Big(v_*\partial_t\xi - v_*\nabla\big[G'(v_*)+\Psi+\eps\partial_vh(u_*,v_*)\big]\cdot\nabla\xi \Big)\dd x\dd t ,
    \end{split}
  \end{equation}
  holds for arbitrary test functions $\xi\in C^\infty_c((0,\infty)\times\R^d)$.
\end{proposition}

With the solution $u_*,v_* \colon [0,\infty)\to\prbtwo$ from Proposition~\ref{prp:existence} we have shown the existence of a transient solution to the initial value problem for~\eqref{eq:eq}, as stated in Theorem~\ref{thm:transient}. 

\subsection{Multi-step estimates}
We next prove several \emph{$\tau$-independent} estimates for $(\tilde{u}_\tau,\tilde{v}_\tau)$, which in the subsequent section then allows us to establish convergence for a sequence $\tau \downarrow 0$. We start by recalling the classical estimate that follows directly from the variational construction.
\begin{lemma}
  \label{lem:classical}
  For each $n \in \N$, we have
  \begin{equation}
  \label{eq:classical_1}
   \nrg_\eps(u_\tau^n,v_\tau^n) + \frac1{2\tau} \big(\wass(u_\tau^n,u_\tau^{n-1})^2+\wass(v_\tau^n,v_\tau^{n-1})^2\big) \leq \nrg_\eps(u_\tau^{n-1},v_\tau^{n-1}) 
  \end{equation}
  and for each $N \in \N$
  \begin{align}
    \label{eq:classical}
    \nrg_\eps(u_\tau^N,v_\tau^N)
    + \frac1{2\tau}\sum_{n=1}^N\big(\wass(u_\tau^n,u_\tau^{n-1})^2+\wass(v_\tau^n,v_\tau^{n-1})^2\big)
    \le \nrg_\eps(u_0,v_0).
  \end{align}
\end{lemma}
\begin{proof}
  The first inequality~\eqref{eq:classical_1} rephrases \eqref{eq:onestepmono}.
  Summing these inequalities for $n=1,2,\ldots,N$, we then end up with the second inequality~\eqref{eq:classical}.
\end{proof}
The following three conclusions of Lemma~\ref{lem:classical} are important in the following:
\begin{itemize}
\item The values of $\nrg_\eps(u_\tau^n,v_\tau^n)$ are monotonically decreasing in~$n$,
  and in particular bounded by $\nrg_\eps(u_0,v_0)$.
  For $\eps \in [0,\eps^*]$, the hypothesis~\eqref{hyp:hbound} then implies a uniform bound on $F(u_\tau^n)$ and $G(v_\tau^n)$ in $L^1(\R^d)$,
  \begin{align}
    \label{eq:Emono}
    \intrd \big[F(u_\tau^n) + G(v_\tau^n)\big]\dd x \le 2 \nrg_\eps(u_0,v_0).
  \end{align}
\item Another consequence of energy monotonicity: for $\eps \in [0,\eps^*]$ we obtain,
  thanks to non-negativity of~$H_\eps$,
  and to the lower bounds on~$\Phi$ and~$\Psi$ by quadratic functions, see~\eqref{eq:Phiupbound}, a uniform bound on the second moments of~$u_\tau^n$ and~$v_\tau^n$,
  \begin{align}
    \label{eq:mom}
    \intrd |x|^2\big(u_\tau^n+v_\tau^n\big)\dd x \le 2 |\underline x_\Phi|^2+  2 |\underline x_\Psi|^2+\frac4\Lambda\nrg_\eps(u_0,v_0).
  \end{align}
\item By non-negativity of~$\nrg_\eps$,
  one can pass for $\eps \in [0,\eps^*]$ to the limit $N\to\infty$ in~\eqref{eq:classical} to obtain
  \begin{align}
    \label{eq:classical2}
    \frac1{\tau}\sum_{n=1}^\infty\big(\wass(u_\tau^n,u_\tau^{n-1})^2+\wass(v_\tau^n,v_\tau^{n-1})^2\big)
    \le 2\nrg_\eps(u_0,v_0).
  \end{align}
  This gives rise to the following uniform estimate on the modulus of quasi-continuity.
\end{itemize}
\begin{lemma}
  \label{lem:holder}
  There is a $\tau$-independent constant~$C$ such that for any $s,t\ge0$,
  \begin{align}
    \label{eq:holder}
    \wass\big(\tilde u_\tau(t),\tilde u_\tau(s)\big) \le C\sqrt{|t-s|+\tau},
    \quad
    \wass\big(\tilde v_\tau(t),\tilde v_\tau(s)\big) \le C\sqrt{|t-s|+\tau}.
  \end{align}
\end{lemma}
\begin{proof}
  Assume $0\le s<t$, and let $\underline n,\overline n\in\N_0$ be
  such that $(\underline n-1)\tau<s\le\underline n\tau$ and $(\overline n-1)\tau<t\le\overline n\tau$,
  i.e., $\tilde u_\tau(t) = u_\tau^{\overline n}$ and $\tilde u_\tau(s)=u_\tau^{\underline n}$,
  with $(\overline n-\underline n)\tau\le (t-s)+\tau$.
  If $\overline n=\underline n$, then~\eqref{eq:holder} trivially holds.
  Otherwise, it follows from~\eqref{eq:classical2}
  via the triangle inequality for~$\wass$ and H\"older's inequality for sums that
  \begin{align*}
    \wass\big(\tilde u_\tau(t),\tilde u_\tau(s)\big)
    \le \sum_{n=\underline n+1}^{\overline n}\wass(u_\tau^n,u_\tau^{n-1})
    &\le \Bigg(\frac1{\tau}\sum_{n=1}^\infty\big(\wass(u_\tau^n,u_\tau^{n-1})^2\Bigg)^{1/2}
    \Bigg(\sum_{n=\underline n+1}^{\overline n}\tau\Bigg)^{1/2} \\
    &\le \sqrt{2\nrg_\eps(u_0,v_0)}\sqrt{(t-s)+\tau}.
  \end{align*}
  This proves the first inequality in~\eqref{eq:holder},
  the second follows in the analogous way.
\end{proof}
\begin{lemma}
  There is a $\tau$-independent constant~$C$ such that,
  for each $T>0$,
  \begin{align}
    \label{eq:bpriori}
    \int_0^T \intrd \big(|\nabla F'(\tilde u_\tau)|^2 + |\nabla G'(\tilde v_\tau)|^2\big)\dd x
    \le C(1+T+\nrg_\eps(u_0,v_0)).
  \end{align}
\end{lemma}
\begin{proof}  
  Assume $T=N\tau$ for simplicity.
  Apply estimate~\eqref{eq:allregular} to $(\hat u,\hat v)=(u_\tau^{n-1},v_\tau^{n-1})$ and $((u^*,v^*)=(u_\tau^n,v_\tau^n)$,
  and sum over $n=1,\ldots,N$.
  This yields
  \begin{multline*}
    \tau\sum_{n=1}^N\intrd\big[|\nabla F'(u_\tau^n)|^2+|\nabla G'(v_\tau^n)|^2\big]\dd x \\
    \le CN\tau \big(1 + \nrg_\eps(u_0,v_0)\big)
    + C\big[\nrg_\eps(u_0,v_0)-\nrg_\eps(u_\tau^N,v_\tau^N) + \entt(u_0,v_0) - \entt(u_\tau^N,v_\tau^N)\big].
  \end{multline*}
  The left-hand side of this inequality coincides with the left-hand side of~\eqref{eq:bpriori}.
  On the right hand side, first observe that $CN\tau=CT$,
  and that $\nrg_\eps(u_\tau^N,v_\tau^N)$ is positive and thus negligible.
  To arrive at~\eqref{eq:bpriori}, it suffices to show that
  \begin{align*}
    -C\big(1+\nrg_\eps(u,v)\big) \le \entt(u,v) \le C\big(1+\nrg_\eps(u,v)\big).
  \end{align*}
  The lower bound is easily obtained by combination of Lemma~\ref{lem:gero} from the Appendix
  with the following estimate, that is a consequence of~\eqref{eq:Phiupbound} and~\eqref{eq:PhibyLyp}:
  \begin{align*}
    \frac\Lambda2\intrd\big[|x-\underline x_\Phi|^2u+|x-\underline x_\Psi|^2v\big]\dd x
    \le \intrd \big[\Phi u+\Psi v\big]\dd x \le C+\ent_\eps(u,v).
  \end{align*}
  The control of~$\entt$ from above is a simple consequence
  of $u\log u\le C(1+F(u))$ and $v\log v\le C(1+G(v))$ thanks to the at least quadratic growth of~$F$ and~$G$, combined with~\eqref{hyp:hbound}.
\end{proof}
\begin{lemma}
  Let $p,q>1$ be such that
  \begin{align}
    \label{eq:pqrelation}
    (d-2)p<2d\quad\text{and}\quad\frac{q}{p'}< 1+\frac{2}{d}.
  \end{align}
  For each $T>0$, there is a $\tau$-independent constant~$C_T$ such that
  \begin{align}
    \label{eq:L2H1L1}
    \int_0^T\|F'(\tilde u_\tau)\|_{L^p}^q \dd t \le C_T. 
  \end{align}
\end{lemma}
\begin{proof}
  Thanks to~\eqref{eq:pqrelation}, we have that
  \begin{align}
    \theta\coloneqq  \frac{2d}{d+2}\frac{1}{p'}<1
    \quad\text{and}\quad
    q\theta<2.
  \end{align}
  Therefore, by the Gagliardo--Nirenberg interpolation inequality,
  \begin{align*}
    \int_0^T\|F'(\tilde u_\tau)\|_{L^p}^q\dd t
    \le CT^{1-q\theta/2}\sup_{0<t<T}\|F'(\tilde u_\tau(t))\|_{L^1}^{q(1-\theta)} \left(\int_0^T\|\nabla F'(\tilde u_\tau)\|_{L^2}^2\dd t\right)^{q\theta/2}. 
  \end{align*}
  From~\eqref{eq:bpriori} the $\tau$-uniform boundedness of the term with $\nabla F'(\tilde u_\tau)$ follows.
  For the other term, we first observe that $F'(s)\le C(s+F(s))$ which for $s \geq 1$ is a direct consequence of~\eqref{eq:tripling}, while for $s \leq 1$, it is obtained from~\eqref{eq:degeneracy} in combination with the uniform boundedness of $F''(t)$ for $t \in (0,1]$.
  Since~$\tilde u_\tau$ is of unit mass, and because of the uniform bound~\eqref{eq:Emono} on $F(\tilde u_\tau)$ in $L^1(\R^d)$,  
  also $F'(\tilde u_\tau)$ is bounded in $L^1(\R^d)$, uniformly in $t\in[0,T]$ and in~$\tau$.
\end{proof}

\subsection{Convergence proofs}
The statements of Proposition~\ref{prp:existence} are proven 
in the three Lemmas~\ref{lem:Wconv}, \ref{lem:Lconv}, and~\ref{lem:ELconv} below.
\begin{lemma}
  \label{lem:Wconv}
  For every $\eps \in [0,\eps^*]$, there are curves $u_*,v_* \colon [0,\infty)\to\prbtwo$, H\"older continuous in~$\wass$,
  such that, along a suitable sequence $\tau\downarrow0$,
  the interpolations~$\tilde u_\tau(t)$ and~$\tilde v_\tau(t)$ converge to~$u_*(t)$ and~$v_*(t)$, respectively,
  weakly in $L^1(\R^d)$, at every $t\ge0$.
\end{lemma}
\begin{proof}
  This lemma is a consequence of the generalized Arzel\`a--Ascoli theorem~\cite[Proposition 3.3.1]{AGS}.
  Lemma~\ref{lem:holder} above provides
  a uniform modulus of (quasi-)continuity for~$\tilde u_\tau$ and~$\tilde v_\tau$ in~$\wass$;
  the topology induced by~$\wass$ is stronger than the narrow one.
  Further, the values of~$\tilde u_\tau$ and~$\tilde v_\tau$ belong to a narrowly compact set
  thanks to the uniform boundedness of second moments~\eqref{eq:mom}.
  The aforementioned proposition yields the narrow convergence along a sequence $\tau\downarrow0$
  of $\tilde u_\tau(t)$ and $\tilde v_\tau(t)$ to respective limits $u_*(t)$ and $v_*(t)$ for each $t\ge0$, and~$u_*$,~$v_*$ are continuous with respect to~$\wass$.
  By lower semi-continuity of~$\wass$ under narrow convergence,
  the estimate~\eqref{eq:holder} is inherited by the limits~$u_*$,~$v_*$ in the form
  \begin{align*}
    \wass\big(u_*(t),u_*(s)\big) \le C|t-s|^{1/2}, \quad \wass\big(v_*(t),v_*(s)\big) \le C|t-s|^{1/2},
  \end{align*}
  which is the claimed Hölder continuity of~$u_*$,~$v_*$ with respect to~$\wass$. Finally, the upgrade from narrow convergence of $\tilde u_\tau(t)$ and $\tilde v_\tau(t)$ to weak convergence in $L^1(\R^d)$ after passage to a suitable subsequence $\tau\downarrow0$
  is obtained from the boundedness of $F(\tilde u_\tau(t))$ and $G(\tilde v_\tau(t))$ in $L^1(\R^d)$,
  see~\eqref{eq:Emono}.
  Indeed, since~$F$ and~$G$ are super-linear at infinity in view of~\eqref{hyp:quadratic_growth}, the Dunford--Pettis criterion applies.
\end{proof}
\begin{lemma}
  \label{lem:Lconv}
  For every $\eps \in [0,\eps^*]$ and every $T>0$, we have along a suitable sequence $\tau\downarrow0$
  \begin{align}
    \label{eq:uinL2}
    \tilde u_\tau&\to u_*\quad \text{strongly in  $L^2([0,T]\times\R^d)$}, \\
    \label{eq:Fuloc}
    F'(\tilde u_\tau) &\to F'(u_*) \quad \text{strongly in $L^2_{loc}([0,T] \times \R^d)$}, \\
    \label{eq:FuinH1}
    \nabla F'(\tilde u_\tau)&\rightharpoonup\nabla F'(u_*)\quad \text{weakly in $L^2([0,T]\times\R^d)$}.
  \end{align}
\end{lemma}
\begin{proof}
  For the proof of~\eqref{eq:uinL2}, 
  we apply the generalized version~\cite[Theorem 2]{Rossi} of the Aubin--Lions Lemma.
  On the Banach space $B\coloneqq L^2(\R^d)$, define the \emph{normal coercive integrand} by
  \begin{align*}
    \mathcal{F}(u) =  \|F'(u)\|_{H^1}^2 + \intrd |x|^2u\dd x,
  \end{align*}
  and the compatible map $g \colon B\times B\to[0,\infty]$ by
  \begin{align*}
    g(u,u') \coloneqq  \wass(u,u'),
  \end{align*}
  with the conventions
  that $\mathcal{F}(u)=\infty$ unless $u\in\prbtwo$ with $F'(u) \in H^1(\R^d)$,
  and that $g(u,u')=\infty$ unless $u,u'\in\prbtwo$.
  Below we verify lower semi-continuity of~$\mathcal{F}$ and compactness of sublevel sets;
  we further show that
  \begin{align}
    \label{eq:Kolmogorov}
    \int_0^{T-h} g\big(\tilde u_\tau(t+h),\tilde u_\tau(t)\big)\dd t \to 0 \quad\text{uniformly in~$\tau$ as $h\downarrow0$},
  \end{align}
  and that
  \begin{align}
    \label{eq:coercive}
    \int_0^T\mathcal{F}(\tilde u_\tau(t))\dd t \le C_T
  \end{align}
  with some constant~$C_T$ depending only on $T>0$.
  In conclusion,~\cite[Theorem 2]{Rossi} then yields that $\tilde u(t)$ converges in $B=L^2(\R^d)$,
  in measure with respect to $t\in(0,T)$, along a sequence $\tau\downarrow0$.
  That limit necessarily coincides with~$u_*$ obtained in the proof of Lemma~\ref{lem:Wconv} above.

  To prove lower semi-continuity of~$\mathcal{F}$,
  consider a sequence~$(u_n)$ converging to~$u_*$ in $L^2(\R^d)$ such that $(\mathcal{F}(u_n))$ has a finite limit.
  Without loss of generality, we may assume that~$u_n$ even converges pointwise a.e..
  Fatou's lemma directly yields
  \begin{align}
    \label{eq:momlsc}
    \intrd |x|^2u \dd x\le \liminf_{n\to\infty}\intrd |x|^2u_n\dd x.
  \end{align}
  Further, by boundedness of the sequence $(F'(u_n))$ in the reflexive space $H^1(\R^d)$,
  a suitable subsequence has a weak $H^1$-limit~$\nu_*$.
  By Rellich's theorem, this convergence is actually strong in $L^2_\text{loc}(\R^d)$,
  and thus also pointwise a.e..
  By a.e.~pointwise convergence almost of~$(u_n)$ and the continuity of~$F'$,
  we conclude $\nu_*=F'(u_*)$.
  Weak lower semi-continuity of norms implies
  \begin{align*}
    \|F'(u_*)\|_{H^1}^2 =\|\nu_*\|_{H^1}^2\le \liminf_{n\to\infty} \|F'(u_n)\|_{H^1}^2,
  \end{align*}
  and thus with~\eqref{eq:momlsc} we conclude that
  \begin{align*}
    \mathcal{F}(u_*)\le\lim_{n\to\infty}\mathcal{F}(u_n).
  \end{align*}
  Concerning compactness, consider the $\bar{\mathcal{F}}$-sublevel set~$S$ of $\mathcal{F}$.
  By boundedness of the $H^1$-norm of $F'(u)$ for all $u \in S$,
  there is a sequence $(u_n)$ in~$S$ for which $F'(u_n)$ converges strongly in $L^2_\text{loc}(\R^d)$.
  Since~$F'$ has at least linear growth, also~$(u_n)$ itself converges strongly in $L^2_\text{loc}(\R^d)$.
  To show that this convergence is not just locally, observe that by interpolation,
  there is a $p>2$ such that $F'(u)$ is uniformly bounded in $L^p(\R^d)$ for all~$u$ in the sublevel~$S$.
  Using that $s^p \le C(s + F'(s)^p)$ for a suitable constant~$C$ thanks to the at least linear growth of~$F'$,
  it follows that there is a $L^p$-uniform bound for the $u_n\in S$ as well.
  H\"older's inequality yields
  \begin{align*}
    \intrd |x|^{2(p-2)/(p-1)}u_n^2 \dd x \le \|u_n\|_{L^p}^{p/(p-1)} \left(\intrd |x|^2u_n\dd x\right)^{(p-2)/(p-1)},
  \end{align*}
  and thus~$(u_n)$ is tight and converges in $L^2(\R^d)$.
  This proves compactness.
  
  The property~\eqref{eq:Kolmogorov} follows from~\eqref{eq:holder}.
  More precisely, we have
  \begin{align*}
    \int_0^{T-h}g\big(\tilde u_\tau(t+h),\tilde u_\tau(t)\big)\dd t \le C T \big(\sqrt h+\sqrt \tau\big).
  \end{align*}
  Next, the estimate~\eqref{eq:coercive} is a consequence of the a priori bound~\eqref{eq:bpriori} on $\nabla F'(\tilde u_\tau)$,
  estimate~\eqref{eq:L2H1L1} with $p=q=2$,
  and the moment control~\eqref{eq:mom}.
  By means of \cite[Theorem 2]{Rossi}, we obtain convergence of $\tilde u_\tau(t)$ to $u_*(t)$ in $L^2(\R^d)$,
  in measure with respect to $t\in[0,T]$, along a sequence $\tau\downarrow0$.
  Estimate~\eqref{eq:L2H1L1} with $2=p<q<2+4/d$
  further yields a $\tau$-uniform control  on $\|\tilde u_\tau(t)\|_{L^2}$ in $L^q(0,T)$, since~$F'$ has at least linear growth,
  and thus uniform integrability to exponent two in time.
  This finishes the proof of~\eqref{eq:uinL2}.
  
  Next, we show the convergence~\eqref{eq:Fuloc} of $(F'(\tilde u_\tau))$ to $F'(u_*)$ locally in $L^2([0,T] \times \R^d)$.
  By~\eqref{eq:uinL2}, we may assume without loss of generality that the chosen sequence $(\tilde u_\tau)$ converges pointwise to~$u_*$ a.e.~on $[0,T]\times\R^d$.
  By continuity of~$F'$, the sequence $(F'(\tilde u_\tau))$ converges to $F'(u_*)$ pointwise almost everywhere.
  Moreover, estimate~\eqref{eq:L2H1L1} with $2<p=q<2+2/d$ provides
  $\tau$-uniform integrability of $F'(\tilde u_\tau)$ on $[0,T]\times\R^d$ to exponent two.
  Hence, by Vitali's convergence theorem, we get convergence of $(F'(\tilde u_\tau))$ to $F'(u_*)$ in $L^2_{loc}([0,T] \times \R^d)$.
  
  It remains to verify~\eqref{eq:FuinH1}, i.e.,
  the weak convergence of $(\nabla F'(\tilde u_\tau))$ to $\nabla F'(u_*)$ in $L^2([0,T]\times\R^d)$.
  In fact, weak $L^2$-convergence to \emph{some} limit $\nabla\zeta$ follows immediately
  from the boundedness~\eqref{eq:bpriori} and the local convergence~\eqref{eq:Fuloc} via Alaoglu's theorem.
  Using once again the local convergence of $(F'(\tilde u_\tau))$ to $F'(u_*)$ we can identify the limit~$\zeta$ as $F'(u_*)$.
\end{proof}
\begin{lemma}
  \label{lem:ELconv}
  The limits $(u_*,v_*)$ obtained in Lemmas~\ref{lem:Wconv}\&\ref{lem:Lconv}
  satisfy the weak formulations~\eqref{eq:weaksol} for every test function $\xi\in C^\infty_c((0,\infty)\times\R^d)$.
\end{lemma}
\begin{proof}
  By abuse of notation,~$\tau$ will always denote an element of the sequence $\tau\downarrow0$
  along which $(\tilde u_\tau,\tilde v_\tau)$ converges to $(u_*,v_*)$ in the sense of Lemmas~\ref{lem:Wconv}\&\ref{lem:Lconv}.

  Assume that the support of~$\xi$ lies in $(0,T)\times\Omega$, for some bounded open set $\Omega\subset\R^d$.
  For each $t\in(\tau,T)$, let~$n$ be such that $(n-1)\tau<t\le n\tau$,
  and use $\zeta\coloneqq \xi(t;\cdot)$ as test function in the first equation of~\eqref{eq:dweak} for that~$n$.
  Integrate these equations in $t\in(\tau,T)$.
  The result can be written as:
  \begin{equation}
    \label{eq:toEL}
    \begin{split}
    &\int_\tau^T\intrd \frac{\tilde u_\tau(t)-\tilde u_\tau(t-\tau)}\tau\xi(t)\dd x\dd t \\
    &= \int_\tau^T\intrd \tilde u_\tau(t)\nabla\big[F'(\tilde u_\tau(t))+\Phi+\eps \partial_u h(\tilde u_\tau(t),\tilde v_\tau(t))\big]\cdot\nabla\xi(t)\dd x\dd t
    + \int_\tau^TR_u(t)\dd t,
    \end{split}
  \end{equation}
  where, thanks to~\eqref{eq:rest},
  \begin{align*}
    \left|\int_\tau^TR_u(t)\dd t\right| 
    & \le \tau\sum_{n=1}^\infty \sup_{0<t<T}\|\xi(t;\cdot)\|_{C^2}\big(\nrg_\eps(u_\tau^{n-1},v_\tau^{n-1}) - \nrg_\eps(u_\tau^n,v_\tau^n)\big) \\
    & \le \tau \sup_{0<t<T}\|\xi(t;\cdot)\|_{C^2}\nrg_\eps(u_0,v_0),
  \end{align*}
  and so $\int_0^TR_u(t)\dd t\to0$ as $\tau\downarrow0$.
  For the integral on the left-hand side of~\eqref{eq:toEL}, we obtain
  \begin{align*}
    \int_\tau^T\intrd \frac{\tilde u_\tau(t)-\tilde u_\tau(t-\tau)}\tau\xi(t)\dd x\dd t
    = \int_0^{T} \intrd \tilde u_\tau(t)\frac{\xi(t+\tau)-\xi(t)}\tau\dd x\dd t
  \end{align*}
  for $\tau > 0$ sufficiently small (recall that the support of~$\xi$ is contained in $(0,T)\times\Omega$),
  which in turn implies
  \begin{align*}
    \int_\tau^T\intrd \frac{\tilde u_\tau(t)-\tilde u_\tau(t-\tau)}\tau\xi(t)\dd x\dd t \to \int_0^T \intrd u_*(t)\partial_t\xi(t)\dd x\dd t,
  \end{align*}
  thanks to the convergence of~$\tilde u_\tau$ to~$u_*$ in $L^1((0,T)\times \Omega)$,
  and the uniform convergence of difference quotients of~$\xi$.

  It remains to verify the convergence of the integral on the right-hand side of~\eqref{eq:toEL}.
  By the strong $L^2$-convergence~\eqref{eq:uinL2} of~$\tilde u_\tau$ to~$u_*$,
  and since $\nabla\xi$ is smooth and has compact support inside $(0,T)\times\Omega$,
  it suffices to verify weak convergence of
  \begin{align*}
    \nabla\big[F'(\tilde u_\tau)+\eps \partial_uh(\tilde u_\tau,\tilde v_\tau)\big] \rightharpoonup \nabla\big[F'(u_*)+\eps \partial_uh(u_*,v_*)\big]
  \end{align*}
  in $L^2((0,T)\times\Omega)$.
  But this is clear:
  on the one hand, $F'(\tilde u_\tau)$ and $G'(\tilde v_\tau)$ converge weakly in $L^2(0,T;H^1(\R^d))$, see~\eqref{eq:FuinH1}.
  On the other hand, recalling the $2$-boundedness of $(F,G,h)$,
  the local $L^2$-convergence~\eqref{eq:Fuloc} of $F'(\tilde u_\tau)$ and $G'(\tilde v_\tau)$
  implies convergence in measure of $\partial_\rho\theta_u\big(F'(\tilde u_\tau),G'(\tilde v_\tau)\big)$ and of $\partial_\eta\theta_u\big(F'(\tilde u_\tau),G'(\tilde v_\tau)\big)$.
  By boundedness and continuity of the derivatives of~$\theta_u$,
  the weak convergence of $\nabla F'(\tilde u_\tau)$ and $\nabla G'(\tilde v_\tau)$
  in $[L^2((0,T)\times\R^d)]^d$ is inherited by
  \begin{equation*}
    \nabla \partial_uh(\tilde u_\tau,\tilde v_\tau)
    =  \partial_\rho\theta_u\big(F'(\tilde u_\tau),G'(\tilde v_\tau)\big)\nabla F'(\tilde u_\tau)
    + \partial_\eta\theta_u\big(F'(\tilde u_\tau),G'(\tilde v_\tau)\big)\nabla G'(\tilde v_\tau).  \qedhere
  \end{equation*}
\end{proof}


\section{Convergence to equilibrium}
\label{sct:longtime}
In preparation of the proof of Theorem~\ref{thm:longtime},
we provide an adapted version of the Csiszar--Kullback inequality for~$\lyp$.
\begin{lemma}
  \label{lem:ck}
  There is a constant~$C$, independent of $\eps\in[0,\bar\eps]$,
  such that for all $u,v\in\prbtwo$ with $\lyp_1(u)<\infty$ and $\lyp_2(v)<\infty$, there hold
  \begin{align}
    \label{eq:ck}
    \|u-\bar u_\eps\|_{L^1}^2\le C \lyp_1(u), \quad
    \|v-\bar v_\eps\|_{L^1}^2\le C \lyp_2(v).
  \end{align}
\end{lemma}
\begin{proof}
  It suffices to prove the first inequality in~\eqref{eq:ck}.
  The point of departure is that both~$u$ and~$\bar u_\eps$ have unit mass,
  and therefore
  \begin{align}
    \label{eq:toscani-2}
    \|u-\bar u_\eps\|_{L^1} = 2\int_{\{u<\bar u_\eps\}}(\bar u_\eps-u)\dd x.
  \end{align}
  It is thus sufficient to estimate the integral of $\bar u_\eps-u$ on $\{u<\bar u_\eps\}$, which is a subset of~$\spt{u}$.
  Let~$\bar U$ be an upper bound on~$\bar u_\eps$, uniformly in $\eps\in[0,\bar\eps]$, see Proposition~\ref{prp:steady}.
  By hypothesis~\eqref{hyp:powerF} there is a constant $c_0>0$ such that $F''(r)\ge c_0 r^{m-2}$ for all $r\le\bar U$,
  and thus we have that
  \begin{align*}
    d_F(u|\bar u_\eps) & = \int_u^{\bar u_\eps} (r-u)F''(r)\dd r \\
    & \ge c_0 \int^{\bar u_\eps}_{\frac{u + \bar u_\eps}{2}}(u-r)r^{m-2}\dd r 
      \ge \frac{c_0}{2^{m-2}} \bar u_\eps^{m-2} \int^{\bar u_\eps}_{\frac{u + \bar u_\eps}{2}} (u-r) \dd r
      = \frac{3c_0}{2^{m+1}} \bar u_\eps^{m-2} (u-\bar u_\eps)^2.
  \end{align*}
  This implies by means of the Cauchy-Schwarz inequality that
  \begin{align}
    \int_{\{u<\bar u_\eps\}} (u-\bar u_\eps)\dd x
    &\le \left(\int_{\{u<\bar u_\eps\}}\bar u_\eps^{-(m-2)}\dd x \right)^{1/2}\left(\int_{\{u<\bar u_\eps\}}\bar u_\eps^{m-2}(u-\bar u_\eps)^2 \dd x \right)^{1/2} \nonumber \\
    &\le \sqrt{2^{m+1}/(3c_0)}\left(\int_{\spt{u}}\bar u_\eps^{-(m-2)}\dd x\right)^{1/2}\Bigg(\intrd d_F(u|\bar u_\eps)\dd x\Bigg)^{1/2}. \label{eq:u_u_eps_estimate_1}
  \end{align}
  It remains to be shown that the integral of $\bar u_\eps^{-(m-2)}$ over~$\spt{u}$ is finite.
  For the estimation of the integrand, we obtain thanks to~\eqref{eqn_def_A}
  \begin{align*}
    \partial_uh(\bar u_\eps,\bar v_\eps) = \theta_u\big(F'(\bar u_\eps),G'(\bar v_\eps)\big) \le  AF'(\bar u_\eps),
  \end{align*}
  and therefore the first Euler--Lagrange equation in~\eqref{eq:EL} implies that
  \begin{align*}
    (1+A\bar\eps)F'(\bar u_\eps) \ge (U_\eps-\Phi)_+\,.
  \end{align*}
  Using further that $F'(\bar u_\eps)\le K\bar u_\eps^{m-1}$, again thanks to~\eqref{eq:degeneracy} and~\eqref{hyp:powerF}, we conclude that on~$\spt{u}$,
  \begin{align*}
    \bar u_\eps \ge c (U_\eps-\Phi)^{1/(m-1)} \quad \text{with}\quad c\coloneqq \big(K(1+A\bar\eps)\big)^{-1/(m-1)}.
  \end{align*}
  We can now estimate the integral of $\bar u_\eps^{-(m-2)}$ by means of the coarea formula.
  Two observations:
  first, $|\nabla\sqrt\Phi|\ge\sqrt{2\Lambda/M}$ by~\eqref{eq:Phigradbound}, and second, the diameter of~$\spt{u}$ is bounded uniformly in $\eps\in[0,\bar\eps]$, see Proposition~\ref{prp:steady}. Hence, the $(d-1)$-dimensional Hausdorff measures $\hausdorff^{d-1}$ of the surfaces of the convex sets $\{\Phi<r^2\}$
  are uniformly bounded by some~$S$ for every~$r$ with $r^2 \le  U_\eps$.
  Now we estimate
  \begin{align*}
    \int_{\spt{u}}\bar u_\eps^{-(m-2)}\dd x
    &\le  \frac{M^{1/2}}{(2\Lambda)^{1/2}c^{m-2}}\int_{\{\sqrt\Phi<\sqrt{U_\eps}\}}(U_\eps-\Phi)^{-(m-2)/(m-1)}|\nabla\sqrt \Phi|\dd x \\
    & \le \frac{M^{1/2}}{(2\Lambda)^{1/2}c^{m-2}}\int_0^{\sqrt{U_\eps}} (U_\eps-r^2)^{-(m-2)/(m-1)}\hausdorff^{d-1}\big(\partial\big\{\sqrt \Phi<r\big\}\big)\dd r \\
    &\le \frac{M^{1/2}S}{(2\Lambda)^{1/2}c^{m-2}}\int_0^{\sqrt{U_\eps}} (U_\eps-r^2)^{-(m-2)/(m-1)}\dd r,
  \end{align*}
  and this integral has a finite value since $(m-2)/(m-1)<1$, which is bounded independently of $\eps\in[0,\bar\eps]$. Combining this with~\eqref{eq:u_u_eps_estimate_1} and~\eqref{eq:toscani-2}, we obtain
  \begin{equation*}
   \|u-\bar u_\eps\|_{L^1} \leq C \left(\intrd d_F(u|\bar u_\eps)\dd x\right)^{1/2},
  \end{equation*}
  for some constant~$C$, which is uniform in $\eps\in[0,\bar\eps]$. With the definition of~$\lyp_1(u)$ from~\eqref{eq:lypuv}, the proof of the first claim in~\eqref{eq:ck} is complete.
\end{proof}  
\begin{proof}[Proof of Theorem~\ref{thm:longtime}]
  We apply Proposition~\ref{prp:core} with $(u^*,v^*)=(u_\tau^n,v_\tau^n)$ and $(\hat u,\hat v)=(u_\tau^{n-1},v_\tau^{n-1})$,
  and use~\eqref{eq:mildconvex}.
  This yields
  \begin{align*}
    \lyp(u_\tau^n,v_\tau^n) \le \big(1+2 \Lambda_\eps \tau\big)\lyp(u_\tau^{n-1},v_\tau^{n-1})
  \end{align*}
  with $\Lambda_\eps = \Lambda - K \eps$ and $K\coloneqq K_0+\Lambda K_1$,
  and then, after iteration on $n=1,2,\ldots,N$,
  \begin{align*}
    \lyp(u_\tau^n,v_\tau^n) \le \big(1+2 \Lambda_\eps \tau\big)^{-n}\lyp(u_0,v_0).
  \end{align*}
  Since~$\lyp$ is a convex functional and thus lower semi-continuous with respect to convergence in~$\wass$, it follows in the limit $\tau\downarrow0$ for the limiting curve from Theorem~\ref{thm:transient} that
  \begin{align*}
    \lyp(u_t,v_t) \le \exp\big(-2\Lambda_\eps t\big)\lyp(u_0,v_0).
  \end{align*}
  The $\lyp(u_0,v_0)$ on the right-hand side is easily estimated in terms of $\nrg_\eps(u_0,v_0)$, see~\eqref{eqn:comp_L_E}.
  Thanks to~\eqref{eq:ck}, the left-hand side controls the $L^1$-norms of $u_t-\bar u_\eps$ and $v_t-\bar v_\eps$.
\end{proof}

\appendix

\section{Verification of Example~\ref{xmp:hypos}}
The required properties for the nonlinearities~$F$ and~$G$ are immediately checked.
Concerning the coupling~$h$:
differentiation of $h(u,v) = u^pv^pe^{-\lambda(u+v)}$ yields
\begin{align*}
  \dff h(u,v)
  &= \begin{pmatrix} p/u-\lambda \\ q/v-\lambda \end{pmatrix} h(u,v), \\
  \dff^2 h(u,v)
  &=
  \begin{pmatrix}
    (p/u-\lambda)^2-p/u^2 & (p/u-\lambda)(q/v-\lambda) \\
    (p/u-\lambda)(q/v-\lambda) & (q/v-\lambda)^2-q/v^2
  \end{pmatrix}
                                 h(u,v).
\end{align*}
For the vanishing of~$h$ itself on $\partial\Rnn^2$, it is sufficient that $p>0$ and $q>0$.
Further, $\partial_uh(u,v)=(p/u-\lambda)h(u,v)$ vanishes on $\partial\Rnn^2$
if $h(u,v)/u$ vanishes, i.e., $p>1$ and $q>0$ are sufficient.
Similarly, $q>1$ and $p>0$ is sufficient for the vanishing of $\partial_vh$.

For the convexity hypothesis, we need to show that 
\begin{align*}
  \begin{pmatrix}
    (m-1)u^{m-2} & 0 \\ 0 & (n-1)v^{n-2}
  \end{pmatrix}
                            +\eps^*\dff_{(u,v)}^2h(u,v) \ge0
\end{align*}
for some $\eps^*>0$.
For that, it is sufficient that
\begin{align*}
  2\eps^*\big|(p/u-\lambda)^2-p/u^2\big|h(u,v) &\le (m-1)u^{m-2}, \\
  2\eps^*\big|(q/v-\lambda)^2-q/v^2\big|h(u,v) &\le (n-1)v^{m-2}, \\
  4(\eps^*)^2 (p/u-\lambda)^2(q/v-\lambda)^2 h(u,v)^2 &\le (m-1)(n-1)u^{m-2}v^{n-2}.
\end{align*}
We make use of the following elementary fact:
for arbitrary $a,b\ge0$, there is a constant~$C$ such that
\begin{align}
  \label{eq:elementary}
  x^ay^b \le Ce^{x+y} \quad\text{for all $x,y>0$}.
\end{align}
The first inequality is equivalent to
\begin{align*}
  2\big|(p-\lambda u)^2-p\big|u^{p-m}v^q \le \frac{m-1}{\eps^*}e^{\lambda(u+v)},
\end{align*}
which is true, uniformly in $u,v>0$ for a sufficiently small choice of $\eps^*>0$
provided that $p\ge m$ and $q\ge0$.
In an analogous manner, it follows that the other two inequalities are satisfiable for some $\eps^*>0$ provided that $p\ge 0$, $q\ge n$ and $p \geq m/2$, $q \geq n/2$, respectively.

Next, we notice that
\begin{align*}
  \theta_u(\rho,\eta) &= \big( p\rho^{(p-1)/(m-1)} - \lambda \rho^{p/(m-1)} \big) \eta^{q/(n-1)}\exp\big(-\lambda\big[\rho^{1/(m-1)}+\eta^{1/(n-1)}\big]\big),
  \\
  \theta_v(\rho,\eta) &=  \rho^{p/(m-1)} \big(q \eta^{(q-1)/(n-1)} - \lambda \eta^{q/(n-1)} \big) \exp\big(-\lambda\big[\rho^{1/(m-1)}+\eta^{1/(n-1)}\big]\big).
\end{align*}

To verify the swap condition, observe that
\begin{align*}
  \partial_\eta\theta_u\big(F'(u), G'(v)\big)
  & = \frac{1}{n-1} \big( p u^{p-1} - \lambda u^p\big) \big(qv^{q-(n-1)}-\lambda v^{(q+1)-(n-1)}\big) e^{-\lambda(u+v)}, \\
  \partial_\rho\theta_u\big(F'(u), G'(v)\big) & = \frac{1}{m-1} \big(p u^{p-(m-1)}-\lambda u^{(p+1)-(m-1)}\big) \big( q v^{q-1} - \lambda v^q \big) e^{-\lambda(u+v)}.
\end{align*}
The estimates in~\eqref{hyp:swap} are obviously equivalent to
\begin{align*}
  \frac{1}{n-1}\big(pu^{p-1/2} - \lambda u^{p+1/2}\big) \big(qv^{q+1/2-n}-\lambda v^{q+3/2 - n}\big)  & \le We^{\lambda(u+v)}, \\
  \frac{1}{m-1}\big(pu^{p+1/2-m}-\lambda u^{p+3/2 - m}\big) \big( q v^{q-1/2} - \lambda v^{q+1/2} \big) & \leq We^{\lambda(u+v)},
\end{align*}
which are true, thanks to~\eqref{eq:elementary}, since $p\ge1/2$, $q\ge n-1/2$ and $p\ge m-1/2$, $q\ge 1/2$, respectively.

For the partial derivatives OF $\theta_u,\theta_v$ up to order~$k$, we need to study boundedness on~$\Rnn^2$ and vanishing on~$\partial\Rnn^2$.
From the chain and product rule of differentiation, we first observe
\begin{multline*}
  \partial_\rho^a\partial_\eta^b\big[\rho^\mu\eta^\nu\exp\big(-\lambda\big[\rho^{1-\gamma}+\eta^{1- \delta}\big]\big) \big] \\
  = \left(\sum_{\substack{a', a'', b', b'' \geq 0 \\ a'+a''=a, b'+b''=b}} c_{a'a''b'b''}\rho^{\mu-a'-a''\gamma}\eta^{\nu-b'-b''\delta}\right)\exp\big(-\lambda\big[\rho^{1-\gamma}+\eta^{1- \delta}\big]\big),
\end{multline*}
where the $c_{a'a''b'b''}$ are combinatorial coefficients depending also on the exponents $\mu, \nu, \gamma, \delta$.
According to~\eqref{eq:elementary},
the expression on the right-hand side is globally bounded for $(\rho,\eta)\in\Rnn^2$
if all exponents of~$\rho$ and~$\eta$ appearing in the sum are non-negative and $1-\gamma$, $1-\delta$ are positive.
Furthermore, these expressions vanish on~$\partial\Rnn^2$ if all exponents are positive.
In our case, we have $\gamma=1-1/ (m-1)= (m-2)/(m-1)\in[0,1)$ and likewise $\delta = (n-2)/(n-1)\in[0,1)$.
Therefore, the smallest possibly occurring exponents for~$\rho$ and~$\eta$ are $\mu-a$ and $\nu-b$, respectively.
Plugging in the ``worst case scenarios'' $\mu=(p-1)/(m-1)$ with $a=k$, and $\nu=(q-1)/(n-1)$ with $b=k$,
we see that the condition~\eqref{eq:puremadness} is indeed sufficient.

\section{A lower bound on the entropy}
The following has been obtained e.g.~in~\cite{JKO}; we recall the proof for convenience.
\begin{lemma}
  \label{lem:gero}
  For any $u\in\prbtwo$, any $\beta>0$, and any $\underline x\in\R^d$,
  \begin{align}
    \label{eq:gero}
    \ent^u(u) = \int u\log u \ge 1-\big(\pi/\beta)^{d/2} - \beta \intrd |x-\underline x|^2u\dd x
  \end{align}
  \textup{(}with $z \log z$ interpreted as zero for $z=0$\textup{)}. In particular,~$\ent^u$ is nowhere~$-\infty$ on $\prbtwo$.
\end{lemma}
\begin{proof}
  By Legendre duality, $zv\le z\log z -z + e^v$ for all $z\ge0$ and $v\in\R$. With the choices $z\coloneqq u(x)$ and $v\coloneqq -\beta|x-\underline x|^2$ this gives
  \begin{align*}
    - \beta\intrd |x-\underline x|^2u\dd x \le \intrd u\log u\dd x - \intrd u\dd x + \intrd e^{-\beta|x-\underline x|^2}\dd x,
  \end{align*}
  which is just~\eqref{eq:gero}.
\end{proof}

\section*{Acknowledgements}
The authors thank Filippo Santambrogio for significant help on the rigorous justification of the Euler--Lagrange equations.

\end{document}